\documentclass[10pt,reqno]{amsart}

\usepackage{amsmath}
\usepackage{amsthm}
\usepackage{amssymb}
\usepackage{amsfonts}
\usepackage{amsxtra}
\usepackage{fullpage}
\usepackage{amscd}
\usepackage{mathrsfs}
\usepackage[all]{xypic}
\usepackage{enumerate}
\usepackage[mathscr]{eucal}
\usepackage{verbatim}
\usepackage{amsmath, amssymb, amsthm, float}
\usepackage{amsmath,amsfonts,amssymb,amsthm}
\usepackage{mathrsfs}
\usepackage[english]{babel}
\usepackage{amsmath,scalerel}
\usepackage[utf8]{inputenc}
\usepackage{tikz}
\usepackage{tikz-cd}
\usetikzlibrary{matrix,arrows,decorations.pathmorphing}
\usepackage{datetime} 
\newcommand{\mc}{\mathcal{C}}

\numberwithin{equation}{section}
\numberwithin{equation}{section}
\usepackage[maxbibnames=99]{biblatex}
\usepackage{amsmath, amsfonts, amssymb, amsthm, calligra, mathrsfs, colonequals, upgreek, bbm, mathtools}
\usepackage[mathscr]{eucal}

\usepackage[colorlinks=true, linkcolor=blue, citecolor=magenta, linktoc=page]{hyperref}

\usepackage[T1]{fontenc}
\usepackage{lmodern}
\usepackage{amsmath}
\usepackage{amsthm}
\usepackage{amssymb}
\usepackage{amsfonts}
\usepackage{amsxtra}
\usepackage{fullpage}
\usepackage{amscd}
\usepackage{mathrsfs}
\usepackage[all]{xypic}
\usepackage{enumerate}
\usepackage[mathscr]{eucal}
\usepackage{verbatim}
\usepackage{amsmath, amssymb, amsthm, float}
\usepackage{amsmath,amsfonts,amssymb,amsthm}
\usepackage{mathrsfs}
\usepackage[english]{babel}
\usepackage{amsmath,scalerel}
\usepackage{tikz}
\usepackage{tikz-cd}
\usetikzlibrary{matrix,arrows,decorations.pathmorphing}
\usetikzlibrary{decorations.pathreplacing}

\usepackage{ytableau}
\usepackage{datetime} 
\usepackage{stackengine} 

\numberwithin{equation}{section}

\usepackage{listings}

\let\nc\newcommand
\let\renc\renewcommand
\numberwithin{equation}{section}
\theoremstyle{plain}
\newtheorem{thm}{Theorem}[section]
\newtheorem{prop}{Proposition}[section]
\newtheorem{cor}{Corollary}[section]
\newtheorem{lem}{Lemma}[section]

\theoremstyle{definition}
\newtheorem{defn}{Definition}[section]
\newtheorem{example}{Example}[section]
\newtheorem{remark}{Remark}[section]

\numberwithin{equation}{section}

\makeatletter
\renewcommand{\subsection}{\@startsection{subsection}{2}{0pt}{-3ex
plus -1ex minus -0.2ex}{-2mm plus -0pt minus
-2pt}{\normalfont\bfseries}} \makeatother

\DeclareMathOperator{\supp}{\mathrm{Supp}}

\DeclareMathOperator{\End}{\mathrm{End}}

\DeclareMathOperator{\gr}{\mathrm{gr}}

\newcommand{\beq}{\begin{equation}\label}
\newcommand{\eeq}{\end{equation}}

\nc{\Z}{\mathbb{Z}}

\newcommand{\C}{\mathbb{C}}
\newcommand{\h}{\mathfrak{h}}
\newcommand{\spec}{\mathrm{Spec}}

\newcommand{\Fun}{\mbox{\mathrm{Fun}}\,}

\nc{\rank}{\textrm{rank} \,}
\nc{\ds}{\dots}
\let\mf\mathfrak

\nc{\HW}{\bar{H}_{\mathbf{c}}(W)}
\nc{\HK}{\bar{H}_{\mathbf{c}}(K)}
\nc{\HtK}{\widetilde{H}_{\mathbf{c}}(K)}
\nc{\CMW}{\textsf{CM}_{\mbf{c}}(W)}
\nc{\CMK}{\textsf{CM}_{\mbf{c}}(K)}
\nc{\mbf}{\mathbf}
\nc{\LK}{\textsf{Irr}(K)}
\nc{\LW}{\textsf{Irr}(W)}
\nc{\Res}{\mathsf{Res} \, }
\nc{\Ind}{\mathsf{Ind} \, }

\nc{\cont}{\textrm{cont}}

\nc{\eWb}{\mathbf{e}_{W_b}}
\nc{\eW}{\mathbf{e}_{W}}
\nc{\msf}{\mathsf}
\nc{\Ui}{\mc{U}_{i,+}}
\nc{\Uone}{\mc{U}_{1,+}}
\nc{\Utwo}{\mc{U}_{2,+}}

\nc{\minusone}{-1}
\nc{\minustwo}{-2}
\nc{\Mod}{\mathrm{Mod} \,}
\nc{\ms}{\mathscr}
\nc{\Frac}{\mathrm{Frac} \,}
\nc{\ra}{\rightarrow}
\nc{\hra}{\hookrightarrow}
\nc{\lab}{\label}
\renc{\O}{\mc{O}}
\nc{\Tan}{\mc{T}}
\nc{\ul}{\underline}
\nc{\s}{\mathfrak{S}}
\nc{\g}{\mf{g}}
\nc{\pa}{\partial}
\nc{\tit}{\textit}
\nc{\Maxspec}{\mathrm{Maxspec} \, }
\nc{\gldim}{\mathrm{gl.dim}}
\nc{\rkm}{\mathrm{rk} \, (\mf{m})}
\nc{\sm}{\mathrm{sm}}
\nc{\PD}{\mathbb{PD}}
\nc{\hilb}{\textrm{Hilb}}
\nc{\<}{\langle}
\renc{\>}{\rangle}
\nc{\T}{\mathbb{T}}
\nc{\X}{\mathbb{X}}
\nc{\W}{\mathscr{W}}
\nc{\kt}{\mbf{k}}
\nc{\ko}{\mbf{k}(0)}
\nc{\Ok}{\mc{O}_G \boxtimes \kt_X}
\nc{\Oko}{\mc{O}_G \boxtimes \ko_X}
\nc{\OYk}{\mc{O}_Y \boxtimes \kt_X}
\nc{\id}{\msf{id}}
\nc{\A}{\mathbb{A}}
\nc{\Grel}{\mc{Grel}}
\nc{\Grat}{\mc{Grat}}
\nc{\Squo}[1]{\A^{(#1)}}
\nc{\twist}{\mathrm{twist}}
\nc{\Cd}{\mc{C}}
\nc{\Span}{\mathrm{Span}}
\nc{\Grass}{\mathrm{Gr}}
\nc{\Fr}{\mathrm{Fr}}
\nc{\pco}[1]{k[V]^{p\mathrm{co} #1}}
\renc{\o}{\otimes}
\renc{\gr}{\mathsf{gr}}
\nc{\U}{\mathsf{U}}
\nc{\algD}{\mf{D}}
\nc{\hr}{\mf{h}_{\textrm{reg}}}
\nc{\D}{\mathscr{D}}
\nc{\PIdeg}{\mathrm{P.I.-degree}}
\nc{\ch}{\mathrm{ch}}
\nc{\ev}{\mathsf{ev}}
\nc{\Stab}{\mathrm{Stab}}
\nc{\Der}{\mathrm{Der}}
\nc{\rightsim}{\stackrel{\sim}{\longrightarrow}}
\nc{\HZ}{H_{\mbf{h},\Z}(\Z_m)}
\nc{\sing}{\mathrm{sing}}
\nc{\dd}{\mathscr{D}}
\nc{\GKdim}{\mathrm{G.K. dim}}
\nc{\PIdegree}{\mathrm{P.I. degree}}
\renc{\H}{\mathsf{H}}
\nc{\rH}{\overline{\mathsf{H}}}
\renc{\Fun}{\mathrm{Fun}}
\nc{\bc}{\mathbf{c}}
\nc{\vc}{\underline{\mathbf{c}}}
\nc{\ba}{\mathbf{a}}

\begin{document}

\title{An explicit presentation of the centre of the restricted rational Cherednik algebra}

\author{Niall Hird}

\email{hirdniall@gmail.com}

\begin{abstract}
In this paper we give an explicit presentation of the centre of the restricted rational Cherednik algebra $\overline{H}_c(S_n\wr\mathbb{Z}/\ell\mathbb{Z})$ for a particular choice of parameter $c$. More precisely, we describe the centre of the indecomposable blocks of $\overline{H}_c(S_n\wr\mathbb{Z}/\ell\mathbb{Z})$ in terms of generators and relations. The parameters $c$ for which this presentation is valid are such that the Calogero-Moser space is smooth. Furthermore, we demonstrate how the explicit presentation of the centre of $\overline{H}_c(S_n\wr\mathbb{Z}/\ell\mathbb{Z})$ can be directly derived from the set of $\ell$-multipartitions of $n$.
\end{abstract}
\maketitle

\section{Introduction}\label{sec1}

Restricted rational Cherednik algebras are objects of interest due to their connections with rational Cherednik algebras and hence integrable systems, algebraic symplectic geometry and algebraic combinatorics. Rational Cherednik algebras are defined for any complex reflection group $(W,\mathfrak{h})$ and parametrised by two values $t$ and $c$. When $t\neq 0$ their centre is trivial, but in the case $t=0$ the centre is large. To be precise, the rational Cherednik algebra is a finite dimensional module over its centre. As a consequence, much of the representation theory of the rational Cherednik algebra can be understood via its centre. The restricted rational Cherednik algebra is a quotient by a subalgebra contained within the centre and it too has a rich structure. The main result of this paper (Theorem~\ref{thm32}) gives an explicit presentation of the centre of the restricted rational Cherednik algebra for the wreath product groups $S_n\wr\mathbb{Z}/\ell\mathbb{Z}$ for a particular parameter $\overline{c}$. 

We consider the restricted rational Cherednik algebra as a direct sum of its indecomposable blocks and find the centre of each of these. These blocks have the important property that they are isomorphic to a tensor product of two graded rings $A(\lambda)^-$ and $A(\lambda)^+$, where $\lambda\in\mathrm{Irr}\,W$. The ring $A(\lambda)^+$ is the endomorphism ring of the baby Verma module for $\lambda$ and in the symmetric group case is isomorphic to $A(\lambda)^-$ with the opposite grading. In the more general case of wreath products these rings are still closely related, for precise details see Theorem~\ref{thm4}. This reduces the problem of understanding the centre of the restricted rational Cherednik algebra to understanding a family of endomorphism rings. To achieve our main result we first find the centre of the restricted rational Cherednik algebra in the symmetric group case $S_n$. 
\begin{thm}(\ref{thm31})
There is an isomorphism of the centre of $\overline{H}_c(S_n)$ for $c\neq 0$  
\[
Z(\overline{H}_c(S_n))\cong \bigoplus_{\lambda\in\mathrm{Irr} S_n} A_c(\lambda)^-\otimes A_c(\lambda)^+.
\]
A presentation of the algebra $A_c(\lambda)^+$ is given by Theorem~\ref{thm25} and $A_c(\lambda)^-$ is isomorphic to $A_c(\lambda)^+$ with the opposite grading.
\end{thm}

In Section $4$ we prove Corollary~\ref{cor2} that in the symmetric case there is an isomorphism between $A(\lambda)^+$ and the ring of functions on the scheme theoretic fibre of a special map denoted $\pi$. This will allow us to exploit a connection to Schubert cells via the Wronski map. Schubert cells are well understood and crucially the fibres of the Wronskian map can be explicitly written in terms of generators and relations. This is explained in more detail in Section $6$. Unfortunately this connection only exists for the symmetric group case and not the more general wreath product groups. This is enough to completely describe the endomorphism rings of the baby Verma modules and hence the entire centre in the symmetric group case. By labelling the irreducible representations of $S_n$ by partitions of $\lambda\vdash n$ we derive Theorem~\ref{thm24} which explicitly describes $A(\lambda)^+$ in terms of generators and relations. Using elementary results concerning partitions we can show that $A(\lambda)^+$ can be directly computed from the Young diagram $D_\lambda$ of $\lambda$.
\begin{thm}(\ref{thm25})
Let $\lambda\vdash n$ be a partition. The algebra $A(\lambda)^+$ is the quotient 
\[
A(\lambda)^+\cong \mathbb{C}[D_\lambda]/I
\]
by the ideal $I$ that is generated by $n$ homogeneous elements $r_1,\dots,r_n$. The $r_s$ are ordered so that $\mathrm{deg}(r_s)=s$. The monomials appearing in $r_i$ are products of cells which share neither a row or column in $D_\lambda$. In other words if $\square_{i,j}\square_{k,\ell}$ is a factor of some monomial in the $r_s$ we must have that $i\neq k$ and $j\neq \ell$. The coefficients of the generators of $I$ are given by Proposition~\ref{prop8}. 
\end{thm}
Since we can describe the rings $A(\lambda)^+$ explicitly in terms of generators and relations it follows we can do the same for the indecomposable blocks of $\overline{H}_c(S_n)$ and hence the entire centre.

The main result of this paper is a generalisation of the first theorem to the wreath product case. As previously mentioned the rings $A_{\overline{c}}(\underline{\lambda})^-$ are closely related to $ A_{\overline{c}}(\underline{\lambda})^+$. In the case of the wreath product group we have that $A_{\overline{c}}(\underline{\lambda})^-$ is isomorphic to $A_{\overline{c}}(\underline{\lambda}^\star)^+$ but for a different simple module $\lambda^\star$. 
\begin{thm}(\ref{thm32})
There is an isomorphism of the centre of $\overline{H}_{\overline{c}}(S_n\wr\mathbb{Z}/\ell\mathbb{Z})$ for the parameters $\overline{c}$,  
\[
Z(\overline{H}_{\overline{c}}(S_n\wr\mathbb{Z}/\ell\mathbb{Z}))\cong \bigoplus_{\underline{\lambda}\in\mathrm{Irr} S_n\wr\mathbb{Z}/\ell\mathbb{Z}} A_{\overline{c}}(\underline{\lambda})^-\otimes A_{\overline{c}}(\underline{\lambda})^+.
\]
The algebra $A_{\overline{c}}(\underline{\lambda})^+$ is given by Theorem~\ref{thm30} and $A_{\overline{c}}(\underline{\lambda})^-$ is isomorphic to $A_{\overline{c}}(\underline{\lambda}^\star)^+$ with the opposite grading.
\end{thm}
Similarly to the symmetric case we need only describe the endomorphism rings of the baby Verma modules to understand the centre. It is important for our purposes to fix a convention for labeling the irreducible representations of $S_n\wr\mathbb{Z}/\ell\mathbb{Z}$. The irreducible representations of $S_n\wr\mathbb{Z}/\ell\mathbb{Z}$ are indexed by $\ell$-multipartitions of $n$. Further, the $\ell$-multipartitions of $n$ are in bijection with the partitions of $n\ell$ with trivial $\ell$-core. For precise definitions of these terms, see Section $2$.

The wreath product case is more difficult as there is no direct connection to Schubert cells via the Wronski map. Instead we use an isomorphism due to Bonnaf\'e and Maksimau \cite[Theorem 4.21]{bonnafe2021fixed} between the Calogero-Moser space of $S_n\wr\mathbb{Z}/\ell\mathbb{Z}$ for our special parameter $\overline{c}$ to an irreducible component of the fixed point subspace of the Calogero-Moser space of $S_{n\ell}$. Section $5$ is then dedicated to using this isomorphism to show that the endomorphism rings of the baby Verma modules for the wreath product $S_n\wr\mathbb{Z}/\ell\mathbb{Z}$ are quotients of $A(\lambda)^+$ and $A(\lambda)^-$ for $S_{n\ell}$. This is the content of Theorem~\ref{thm18} which states that there are isomorphisms $A(\mathrm{quo}_\ell(\lambda))^+\cong A(\lambda)^+_{\mathbb{Z}/\ell\mathbb{Z}}$ and $A_{\overline{c}}(\mathrm{quo}_\ell(\lambda)^\sharp)^-\cong A_c(\lambda)^-_{\mathbb{Z}/\ell\mathbb{Z}}$. It is then possible to rewrite the negative case solely in terms of the positive case using Corollary~\ref{cor30}. This reduces understanding the centre to simply understanding $A(\mathrm{quo}_\ell(\lambda))^+$, similarly to the symmetric case.

Theorem~\ref{thm18} is a powerful result, which easily generalises Theorem~\ref{thm25} to Theorem~\ref{thm30} below. Note that $h(i,j)$ means the hook length of the cell $(i,j)$ in the Young diagram.
\begin{thm}(\ref{thm30})
Let $\lambda\vdash n\ell$ be a partition with trivial $\ell$-core and $\mathrm{quo}_\ell(\lambda)$ its $\ell$-quotient. The algebra $A_{\overline{c}}(\mathrm{quo}_\ell(\lambda))^+$ is the quotient
\[
A_{\overline{c}}(\mathrm{quo}_\ell(\lambda))^+\cong \mathbb{C}[D^\ell_\lambda]/I
\]
where $D_\lambda^\ell$ is the subdiagram of $D_\lambda$ (the younger diagram) excluding the cells $(i,j)$ such that $h(i,j)$ is not divisible by $\ell$. The ideal $I$ is generated by $n$ homogeneous elements $r_\ell, r_{2\ell},\dots,r_{n\ell}$. The $r_{s\ell}$ are ordered so that $\mathrm{deg}(r_{s\ell})=s\ell$. The monomials appearing in $r_{s\ell}$ are products of cells which share neither a row or column in $D_\lambda^\ell$. In other words if $\square_{i,j}\square_{k,m}$ is a factor of some monomial appearing in the $r_{s\ell}$, we must have that $i\neq k$ and $j\neq m$. The coefficients of the generators of $I$ are given by Proposition~\ref{prop8}. 
\end{thm}
As in the symmetric group case, if we can describe the endomorphism rings of the baby Verma modules explicitly then we can do the same for the entire centre. In Section $9$, after proving Theorem~\ref{thm32} we provide two examples by calculating the centres of $\overline{H}_c(S_2\wr\mathbb{Z}/2\mathbb{Z})$ and $\overline{H}_c(S_3\wr\mathbb{Z}/2\mathbb{Z})$. The main results contained within this paper first appeared in the doctoral thesis of the author \cite{hird2022representation}.

\section{Combinatorics and partitions}
This section provides the basic definitions and results we require for the combinatorics in this paper. Of particular importance is the notion of $\ell$-cores and $\ell$-quotients. We will begin with the definition of a partition before defining $\ell$-multipartitions and bead diagrams.
\begin{defn}
Let $n$ be a positive integer. A $partition$ of $n$ is a tuple $(\lambda_1,\dots,\lambda_n)$ of non-negative integers such that $\lambda_{i}\geq \lambda_{i+1}$ for all $1\leq i\leq n-1$, and 
\[
|\lambda|:=\sum_{i=1}^n\lambda_i=n.
\]
The $length$ of $\lambda$ is the positive integer $t$ such that $\lambda_t\neq 0$ and $\lambda_{t+1}=0$.
\end{defn}
Young diagrams are a common way to represent partitions. The Young diagram for the partition $(\lambda_1,\dots, \lambda_n)\vdash n$ consists of left aligned rows, with the $i^{th}$ row having $\lambda_i$ cells. We count the columns from left to right and the rows top to bottom. This means that the cell $(2,3)$ is the second row down and the third column along to the right. The hook length is a function that assigns an integer to each cell in the Young diagram. The hook length is calculated for each cell by summing the cells to the right and the cells directly below, then adding one, for the cell itself. Below are the Young diagrams of the partitions $(4,0,0,0)$, $(3,1,0,0)$, $(2,2,0,0)$, $(2,1,1,0)$ and $(1,1,1,1)$, with the hook lengths of each cell written inside.
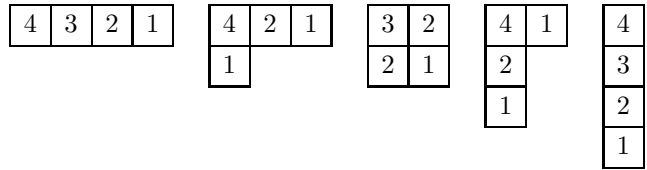
\begin{center}
\begin{ytableau}
       4 & 3 & 2 & 1
\end{ytableau}
\quad
\begin{ytableau}
       4 & 2 & 1  \\
       1
\end{ytableau}
\quad
\begin{ytableau}
       3 & 2  \\
       2 & 1
\end{ytableau}
\quad\begin{ytableau}
     4 & 1\\
     2 \\
     1
\end{ytableau}
\quad\begin{ytableau}
     4 \\
     3 \\
     2 \\
     1
\end{ytableau}
\end{center}
Let us now give a formula for the hook length of a cell. For a given partition $\lambda$ we denote by $\lambda^T$ its transpose. Then $\lambda^T_j$ denotes the number of cells in the $j^{th}$ column of $\lambda$. The hook length $h(i,j)$ of the cell $(i,j)$ is then
\begin{equation}\label{equation1}
h(i,j)=\lambda_i-j+\lambda^T_j-i+1.
\end{equation}
\begin{lem}\label{lem1}
Let $\lambda\vdash n$ and consider the set $P=\{d_1,\dots,d_n\}$, where $d_i=\lambda_i+n-i$. Then 
\[
|\{j\,|\,d_i-j\not\in P \textnormal{ for } 1\leq j\leq d_i\}|=\lambda_i.
\]
\end{lem}
\begin{proof}
Note that $d_i>d_j$ if $i<j$. Hence there are $n-i$ elements $d_k\in P$ such that $d_k<d_i$. The number of elements $d_i-j\not\in P$ is therefore
\[
d_i-(n-i)=\lambda_i+n-i-n+i= \lambda_i.
\]
\end{proof}
If we let $P$ and $d_i$ be the same as in Lemma~\ref{lem1} we can prove the following.
\begin{lem}\label{lem2}
The set of hook lengths of the row $i$ equals the set $\{j\,|\,d_i-j\not\in P\}$.
\end{lem}
\begin{proof}
Let us show that the set of hook lengths of row $i$ is contained in the set $\{j\,|\,d_i-j\not\in P\}$. The result will then follow from Lemma~\ref{lem1}. Fix $i\in\{1,2,\dots ,n\}$. If there exists $d_m$ such that $d_i-h(i,j)=d_m$ for some $j$ then 
\[
h(i,j)=d_i-d_m.
\]
Hence
\[
\lambda_i-j+\lambda^T_j-i+1=\lambda_i+n-i-(\lambda_m+n-m),
\]
which simplifies to
\begin{equation}\label{eq:hooky}
 \lambda^T_j-j+1=m-\lambda_m.
\end{equation}
We now have two cases to consider $\lambda^T_j\geq m$, and $\lambda^T_j<m$. 

First let $\lambda^T_j\geq m$. Then $\lambda^T_j=m+k$ for some non-negative integer $k$ and so equation~\eqref{eq:hooky} becomes $j=k+\lambda_m+1$. Therefore $j>\lambda_m$. Since $\lambda^T_j>m$ we have that $\lambda_m\geq\lambda_{\lambda^T_j}$, hence $j>\lambda_{\lambda^T_j}$. But this contradicts the definition of $\lambda^T_j$. 

The second case is when $\lambda^T_j<m$. Then $\lambda^T_j=m-k\geq i$ for some positive integer $k$. Equation \eqref{eq:hooky} $\lambda_m +1=j+k$ now since $k$ is a positive integer we have $j\leq \lambda_m$. But since $m>\lambda^T_j$ we have that $j>\lambda_m$ and so we have a contradiction. Therefore there is no $d_m$ such that $d_i-h(i,j)=d_m$ and so $h(i,j)\in \{j\,|\,d_i-j\not\in P\}$
\end{proof}
\begin{lem}\label{lem3}
Given a partition $\lambda\vdash n$, we have $d_i-h(i,j)=d_k-h(k,j)$. 
\end{lem}
\begin{proof}
Using formula \eqref{equation1} the following calculations give the result
\[
d_i-h(i,j)=\lambda_i+n-i-\lambda_i+j-\lambda^T_j+i-1=n+j-\lambda^T_j-1
\]
and 
\[
d_k-h(k,j)=\lambda_k+n-k-\lambda_k+j-\lambda^T_j+k-1=n+j-\lambda^T_j-1.
\]
\end{proof}
\begin{defn}
An $\ell$-multipartition of $n$ is a $\ell$-tuple $(\lambda^1,\dots, \lambda^\ell)$ such that each $\lambda^i$ is a partition and $\sum_{i=1}^\ell|\lambda^i|=n$.
\end{defn}
The $first$ $column$ $hook$ $lengths$ of a partition is the set of hook lengths of the cells on the leftmost column of the Young diagram for a partition. For example, the partition $(3,2,1,1)$ has Young diagram
\begin{center}
\begin{ytableau}
       6 & 3 & 1\\
       4 & 1 \\
       2 \\
       1
\end{ytableau}
\end{center}
and the first column hook lengths are $\{6,4,2,1\}$. The first column hook lengths are important because the original partition can always be recovered from the first column hook lengths. To see this, recall the formula for hook length given above,
\[
h(i,j)=\lambda_i-j+\lambda^T_j-i+1.
\]
The first column hook lengths are given by fixing $j=1$. Hence
\[
h(i,1)=\lambda_i-1+\lambda^T_j-i+1=\lambda_i+\lambda^T_j-i,
\]
and thus $\lambda_i=h(i,1)-\lambda^T_j+i$. 
\begin{defn}\label{betanum}
let $\lambda=(\lambda_1,\dots,\lambda_m,0\dots,0)\vdash n$ and $p\geq m$ then define
\[
\beta_i^p=\lambda_i+p-i\textnormal{ for }1\leq i\leq p.
\]
The set $\{\beta_i^p|1\leq i\leq p\}$ is called a set of $\beta$-numbers for $\lambda$.
\end{defn}
The set of first column hook lengths is a set of $\beta$-numbers and as we shall see these allow us to define $\ell$-cores and $\ell$-quotients by bead diagrams.
\begin{defn}
We refer to elements of the set $\mathbb{Z}_{\leq -1}\times\{0,\dots, \ell-1\}$ as $points$. A $bead$ $diagram$ is a function $f:\mathbb{Z}_{\leq -1}\times\{0,\dots, \ell-1\}\rightarrow \{0,1\}$ which takes the value $1$ for only finitely many points. If $f(i,j)=1$ then the point is said to be occupied by a $bead$. If $f(i,j)=0$, then the point $(i,j)$ is $empty$. 
\end{defn}
We can construct a bead diagram if given any partition $\lambda\vdash n$ and integer $\ell\leq n$ in the following way. Following the notation of Definition~\ref{betanum} let $p$ be the smallest multiple of $\ell$ such that $p\geq m$. Then denote by $K$ the set of corresponding $\beta$-numbers and place the beads according to the following rule $f(i,j)=1$ if and only if $-(i+1)\cdot \ell+j\in K$. Denote this bead diagram by $\mathfrak{B}_\ell(\lambda)$.

There is an easier way to construct the bead diagram associated to a partition, that agrees with the condition given above. Consider the diagram of $\ell$ columns of empty beads. Then count left to right and begin with the top row (this bead corresponds to $0$), the full beads correspond to the elements of $K$. An example makes this clear.
\begin{example}
Let $\ell=3$, then we shall write bead diagram for the partition $(3,2,1,1)$. Since $p=6$ is the smallest multiple of $3$ that is greater than $4$ our set of $\beta$-numbers is $\{8,6,4,3,1,0\}$. Therefore the bead diagram has three columns and the full beads are placed according to the $\beta$-numbers counting left to right and top to bottom. Hence $\mathfrak{B}_3(3,2,1,1)$ is
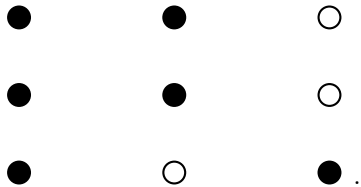
\begin{center}
\begin{tikzpicture}[thick, scale=\textwidth/4cm]
    \begin{scope}[xshift=0, yshift=0]
      \coordinate (a) at (0,0); 
      \coordinate (b) at (0.5,0);
      \coordinate (c) at (1,0);
      \coordinate (d) at (0,0.25);
      \coordinate (e) at (0.5,0.25); 
      \coordinate (f) at (1,0.25);
      \coordinate (g) at (0,0.5);
      \coordinate (h) at (0.5,0.5);
      \coordinate (i) at (1,0.5); 
      \draw[fill] (a) circle (1pt) node [above=4.6pt] {};
      \draw[] (b) circle (1pt) node [right=2pt] {};
      \draw[fill] (c) circle (1pt) node [above=2pt] {};
      \draw[fill] (d) circle (1pt) node [above=2pt] {};
      \draw[fill] (e) circle (1pt) node [above=4.6pt] {};
      \draw[] (f) circle (1pt) node [right=2pt] {};
      \draw[fill] (g) circle (1pt) node [above=2pt] {};
      \draw[fill] (h) circle (1pt) node [above=2pt] {};
      \draw[] (i) circle (1pt) node [above=4.6pt] {};
    \end{scope}
  \end{tikzpicture}.
  \end{center}
\end{example}
Any bead diagram gives rise to a unique partition by following the reverse process, with the caveat that we begin counting the beads from the first empty bead. The set of numbers that we recover is the set of first column hook lengths for the partition. For instance consider the bead diagram 
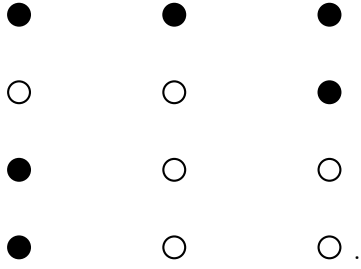
\begin{center}
\begin{tikzpicture}[thick, scale=\textwidth/4cm]
    \begin{scope}[xshift=0, yshift=0]
      \coordinate (a) at (0,0); 
      \coordinate (b) at (0.5,0);
      \coordinate (c) at (1,0);
      \coordinate (d) at (0,0.25);
      \coordinate (e) at (0.5,0.25); 
      \coordinate (f) at (1,0.25);
      \coordinate (g) at (0,0.5);
      \coordinate (h) at (0.5,0.5);
      \coordinate (i) at (1,0.5); 
      \coordinate (j) at (0,0.75);
      \coordinate (k) at (0.5,0.75);
      \coordinate (l) at (1,0.75); 
      \draw[fill] (a) circle (1pt) node [above=4.6pt] {};
      \draw[] (b) circle (1pt) node [right=2pt] {};
      \draw[] (c) circle (1pt) node [above=2pt] {};
      \draw[fill] (d) circle (1pt) node [above=2pt] {};
      \draw[] (e) circle (1pt) node [above=4.6pt] {};
      \draw[] (f) circle (1pt) node [right=2pt] {};
      \draw[] (g) circle (1pt) node [above=2pt] {};
      \draw[] (h) circle (1pt) node [above=2pt] {};
      \draw[fill] (i) circle (1pt) node [above=4.6pt] {};
      \draw[fill] (j) circle (1pt) node [above=2pt] {};
      \draw[fill] (k) circle (1pt) node [above=2pt] {};
      \draw[fill] (l) circle (1pt) node [above=4.6pt] {};
    \end{scope}
  \end{tikzpicture}.
  \end{center}
The first column hook lengths are $\{2,3,6\}$ which corresponds to the partition $(4,2,2)$.
\begin{defn}
 Let $\lambda\vdash n$ and fix a positive integer $\ell\leq n$. Consider the bead diagram $\mathfrak{B}_\ell(\lambda)$ of $\lambda$ with $\ell$ columns. If we slide the beads upwards as much as possible we obtain a new bead diagram. The partition corresponding to this new bead diagram is called the $\ell$-core of $\lambda$.  
\end{defn}

\begin{defn}
Consider the bead diagram $\mathfrak{B}_\ell(\lambda)$. The columns can be considered as bead diagrams for $\ell=1$. Denote the partition defined by the first column by $\lambda^1$, the second by $\lambda^2$ and so on. We define the $\ell$-$quotient$ of $\lambda$ to be the $\ell$-multipartition $(\lambda^1,\dots, \lambda^\ell)$.
\end{defn}
\begin{example}
Let us find the $3$-core and $3$-quotient of $(4,2,2)$. First write $\mathfrak{B}_3(4,2,2)$ which is 
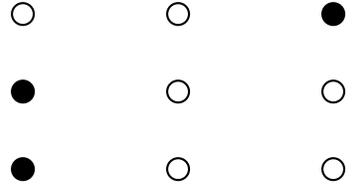
\begin{center}
\begin{tikzpicture}[thick, scale=\textwidth/4cm]
    \begin{scope}[xshift=0, yshift=0]
      \coordinate (a) at (0,0); 
      \coordinate (b) at (0.5,0);
      \coordinate (c) at (1,0);
      \coordinate (d) at (0,0.25);
      \coordinate (e) at (0.5,0.25); 
      \coordinate (f) at (1,0.25);
      \coordinate (g) at (0,0.5);
      \coordinate (h) at (0.5,0.5);
      \coordinate (i) at (1,0.5); 

      \draw[fill] (a) circle (1pt) node [above=4.6pt] {};
      \draw[] (b) circle (1pt) node [right=2pt] {};
      \draw[] (c) circle (1pt) node [above=2pt] {};
      \draw[fill] (d) circle (1pt) node [above=2pt] {};
      \draw[] (e) circle (1pt) node [above=4.6pt] {};
      \draw[] (f) circle (1pt) node [right=2pt] {};
      \draw[] (g) circle (1pt) node [above=2pt] {};
      \draw[] (h) circle (1pt) node [above=2pt] {};
      \draw[fill] (i) circle (1pt) node [above=4.6pt] {};

    \end{scope}
  \end{tikzpicture}
  \end{center}
  to find the $3$-core we shift all beads up as far as they can go, to obtain the new bead diagram 
  
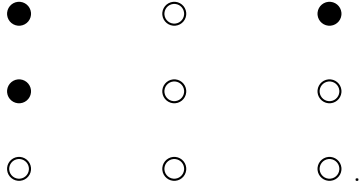
\begin{center}
\begin{tikzpicture}[thick, scale=\textwidth/4cm]
    \begin{scope}[xshift=0, yshift=0]
      \coordinate (a) at (0,0); 
      \coordinate (b) at (0.5,0);
      \coordinate (c) at (1,0);
      \coordinate (d) at (0,0.25);
      \coordinate (e) at (0.5,0.25); 
      \coordinate (f) at (1,0.25);
      \coordinate (g) at (0,0.5);
      \coordinate (h) at (0.5,0.5);
      \coordinate (i) at (1,0.5); 
      \draw[] (a) circle (1pt) node [above=4.6pt] {};
      \draw[] (b) circle (1pt) node [right=2pt] {};
      \draw[] (c) circle (1pt) node [above=2pt] {};
      \draw[fill] (d) circle (1pt) node [above=2pt] {};
      \draw[] (e) circle (1pt) node [above=4.6pt] {};
      \draw[] (f) circle (1pt) node [right=2pt] {};
      \draw[fill] (g) circle (1pt) node [above=2pt] {};
      \draw[] (h) circle (1pt) node [above=2pt] {};
      \draw[fill] (i) circle (1pt) node [above=4.6pt] {};
    \end{scope}
  \end{tikzpicture}.
  \end{center}
Beginning the count from the first empty bead we get first column hook lengths $\{1,2\}$ which corresponds to the partition $(1,1)$. To find the $3$-quotient we consider the columns of $\mathfrak{B}_3(4,2,2)$ as their own bead diagrams. Then the set of first column hook lengths for $\lambda^2$ and $\lambda^3$ are empty. The first bead column has first column hook lengths $\{1,2\}$. Therefore the $3$-quotient of $(4,2,2)$ is $((1,1),\emptyset,\emptyset)$.
\end{example}
Let us introduce one last piece of notation. Denote by $\mathcal{P}(n)$ the set of all partitions of $n$, and $\mathcal{P}(n,\ell)$ the set of all $\ell$-multipartitions of $n$. Also, denote by $\mathcal{P}^{\ell}_\lambda(n)$ the set of all partitions of $n$ with $\ell$-core $\lambda$. A partition is uniquely determined by its $\ell$-core and its $\ell$-quotient. The following is \cite[Theorem 2.7.30]{JK} which allows us to equate $\ell$-multipartitions of $n$ with partitions of $n\ell$ that have trivial $\ell$-core.
 \begin{thm}\label{thm1}
There is a bijection between the set of partitions of $n\ell$ with trivial $\ell$-core and the $\ell$-multipartitions of $n$
\[
\mathcal{P}^{\ell}_{\emptyset}(n\ell)\rightarrow \mathcal{P}(n,\ell),
\]
given by $\lambda\rightarrow quo(\lambda)$.
\end{thm}
\section{The blocks of the centre}
In this section we will study the essential properties of the indecomposable blocks of the centre of the restricted rational Cherednik algebra. Namely, we recall that these blocks are isomorphic to a tensor product of endomorphism rings of baby Verma modules. This fact is of critical importance to later results, as it is these rings that we will describe. Further, we will show that the blocks are indexed by the irreducible representations of the complex reflection group $W$ when the centre is a regular algebra. We begin by recalling the definition of the rational Cherednik algebra.

Let $W\subset GL(\mathfrak{h})$ be a complex reflection group and let $V=\mathfrak{h} \oplus \mathfrak{h}^{*}$. There is a natural pairing $(-,-): \mathfrak{h} \times \mathfrak{h}^{*} \rightarrow \mathbb{C}$ given by $(y, x) := x(y)$. Then the standard
symplectic form $\omega$ on $V$ is given by
\[
\omega (y_1 \oplus x_1, y_2 \oplus x_2) = (y_1, x_2)-(y_2, x_1).
\]
Let $S$ be the set of complex reflections in $W$. For $s\in S$ we denote the symplectic 2-form that is equal to $\omega$ on $\mathrm{Im}(1-s)$ and $0$ on $\mathrm{Ker}(1-s)$ by $\omega_s$.

Given a complex reflection group $(W,\mathfrak{h})$, $t\in\mathbb{C}$ and a class function $c:S\rightarrow \mathbb{C}$. We define the $rational$ $Cherednik$ $algebra$
\[
H_{t,c}(W):=T(V)\rtimes W/\langle x\otimes y-y\otimes x-\kappa(x,y) \,|\,\forall x,\, y\in V\rangle
\]
where 
\[
\kappa(x,y)=t\cdot\omega(x,y)\cdot 1
+\sum_{s\in S}c(s)\cdot\omega_s(x,y)\cdot s.
\]
Let us fix some notation. For any algebra $A$ we denote by $Z(A)$ its centre. For brevity we also define $Z_c(W):=Z(H_c(W))$ and $X_c(W):=\spec\, Z_c(W)$. It is shown \cite[Proposition 3.6]{Baby} that there is an inclusion of algebras
\[
i:\mathbb{C}[\mathfrak{h}]^W\otimes\mathbb{C}[\mathfrak{h}^*]^W\hookrightarrow Z_c(W).
\]
The dual map on spectra is then denoted
\[
\gamma:X_c(W)\rightarrow \mathrm{Spec}\,(\mathbb{C}[\mathfrak{h}]^W\otimes\mathbb{C}[\mathfrak{h}^*]^W)=\mathfrak{h}/W\times \mathfrak{h}^*/W. 
\]

The $restricted$ $rational$ $Cherednik$ $algebra$ is the quotient algebra
\[
\overline{H}_{c}(W):=H_{0,c}(W)/R_+H_{0,c}(W).
\]
Here $R_+=\mathbb{C}[\mathfrak{h}]_+^W\otimes\mathbb{C}[\mathfrak{h}^*]^W+\mathbb{C}[\mathfrak{h}]^W\otimes\mathbb{C}[\mathfrak{h}^*]_+^W$ is the ideal of the central subalgebra $\mathbb{C}[\mathfrak{h}]^W\otimes\mathbb{C}[\mathfrak{h}^*]^W$ generated by elements with no constant term.

The primary object of this paper is the restricted rational Cherednik algebra of the wreath product group $S_n\wr\mathbb{Z}/\ell\mathbb{Z}$, and so we take this opportunity to give its definition. 
\begin{defn}
The wreath product $S_n\wr\mathbb{Z}/\ell\mathbb{Z}$ is the semidirect product $(\mathbb{Z}/\ell\mathbb{Z})^n\rtimes S_n$. This means that group multiplication is given by 
\[
(s,a_1,\dots,a_n)\cdot (t,b_1,\dots,b_n)=(st,a_{t(1)}b_1,\dots,a_{t(n)}b_n)
\]
where $s,t\in S_n$ and $a_i, b_j$ are $\ell^{th}$ roots of unity for all $1\leq i\leq n$ and $1\leq j\leq n$.
\end{defn}
Write the restricted rational Cherednik algebra as the direct sum of its indecomposable blocks
\begin{equation}\label{equation3}
\overline{H}_c(W)=\bigoplus_{j\in I}B_j.
\end{equation}
The blocks of $\overline{H}_c(W)$ are in bijection with the points of $\gamma^{-1}(0)$ \cite[Corollary 2.7]{Ramifications}. Therefore, we can replace the indexing set $I$ with $\gamma^{-1}(0)$ and write \eqref{equation3} as
\begin{equation}\label{equation4}
    \overline{H}_c(W)=\bigoplus_{p\in \gamma^{-1}(0)}B_p.
\end{equation}
To better understand the blocks we require two key facts, firstly that the blocks have the form of a matrix algebra and secondly that there is a bijection between the irreducible representation of $W$ and the blocks. Both of these require that $X_c(W)$ is smooth. Thankfully, this condition is not hard to guarantee due to \cite[Corollary 1.14]{EG} which says the following.
\begin{lem}\label{lem4}
For a suitably generic class function $c$ (one which is a complement to finitely many hyperplanes on the set of conjugacy classes of reflections in $S_n\wr \mathbb{Z}/\ell\mathbb{Z}$) the variety $X_c(S_n\wr \mathbb{Z}/\ell\mathbb{Z})$ is smooth.
\end{lem}
\textbf{Throughout the rest of this paper it is assumed} $c$ \textbf{is chosen so that $X_c(W)$ is smooth}. Also, we say that $c$ is generic if the algebra $Z(H_c(S_n\wr\mathbb{Z}/\ell\mathbb{Z}))$ is regular.

Let us describe the blocks in \eqref{equation3}. The algebra $B_p$ is a matrix algebra over the ring of functions at the points $p\in \gamma^{-1}(0)$ by \cite[Proposition 2.2]{brown2001}, 
\begin{equation}\label{equation5}
B_p=\mathrm{Mat}_{|W|}(\mathcal{O}_p). 
\end{equation}
Here $\mathcal{O}_p=(Z_c(W)/R_+Z_c(W))_p$ is the scheme theoretic fibre of $\gamma$ at $0$ localised at the point $p$. The next proposition demonstrates that the blocks can be labeled by the irreducible representations of $W$. 

To prove that the blocks of the centre are labeled by irreducible representations of $W$ we use the baby Verma modules and their opposites. Recall that these are the standard modules for the restricted rational Cherednik algebra. They are defined for any irreducible representation $\lambda\in \mathrm{Irr}\,W$ as 
\begin{equation}\label{equation6}
    \Delta_c(\lambda)=\overline{H}_c(W)\otimes_{\mathbb{C}[\mathfrak{h}^*]^{coW}\rtimes W}\lambda.
\end{equation}
Here $\mathfrak{h}\subset \mathbb{C}[\mathfrak{h}^*]^{coW}$ acts as $0$ on $\lambda$.
\begin{prop}\label{prop1}
The blocks of $\overline{H}_c(W)$ are in bijection with the irreducible representations of $W$.
\end{prop}
\begin{proof}
Write 
\[
\overline{H}_c(W)=\bigoplus_{i\in \gamma^{-1}(0)} B_i,
\]
a sum of indecomposable subalgebras. Since $\Delta(\lambda)$ is a $\overline{H}_c(W)$-module, 
\[
\Delta(\lambda)=\overline{H}_c(W)\cdot \Delta(\lambda)=\bigoplus_{i\in \gamma^{-1}(0)} B_i\cdot \Delta(\lambda).
\]
Since $\Delta(\lambda)$ is indecomposable we must have $B_i\cdot \Delta(\lambda)=0$ for all $i\neq j$ for some (unique) $j$.

Let $L(\lambda)$ be the unique simple quotient of $\Delta(\lambda)$ \cite[Proposition 4]{Baby}. Since $\mathcal{O}_p$ is a local ring it has a unique simple module. Equation~\eqref{equation5} then implies that $B_p$ also has a unique simple module. If $\Delta(\lambda) =B_p \cdot \Delta(\lambda)$ and $\Delta(\mu)=B_p\cdot \Delta(\mu)$ then the simple module of $B_p$ equals both $L(\lambda)$ and $L(\mu)$. As the standard modules $\Delta(\lambda)$ have a unique simple quotient $L(\lambda)$ this forces $\lambda=\mu$. 

The mapping from irreducible representations of $W$ to the blocks defined above as $\lambda\rightarrow p$ where $B_p\cdot \Delta(\lambda)\neq 0$ is surjective. To see this note that each simple $\overline{H}_c(W)$-module is a quotient of a baby Verma module \cite[Proposition 4.3]{Baby}
\end{proof}
Let us now show that the centre of $\overline{H}_c(W)$ is the sum of the centres of the blocks. Using the bijection from Proposition~\ref{prop1}, let $B_\lambda$ correspond to the block $B_p$. Consider the following maps, the inclusion map $i:Z_c(W)\rightarrow H_c(W)$, the quotient map $q:H_c(W)\rightarrow \overline{H}_c(W)$ defined by $q(z)= z+R_+H_c(W)$ and the projection $\phi_p:\overline{H}_c(W)\rightarrow B_p$. Denote by $A(\lambda)$ the image of $Z_c(W)$ under the composition of these maps 
\begin{equation}\label{equation7}
Z_c(W)\hookrightarrow H_c(W)\twoheadrightarrow \overline{H}_c(W)\twoheadrightarrow B_\lambda. 
\end{equation}
We will show that $A(\lambda)=\mathcal{O}_p$.
\begin{lem}\label{lem5}
The image of the centre $Z_c(W)$ under the composition of the inclusion and quotient map is equal to the centre of $\overline{H}_c(W)$. That is, 
\[
q\circ i(Z_c(W))=Z(\overline{H}_c(W)).
\]
\end{lem}
\begin{proof}
The corollary \cite[Corollary 2.7]{brown2001} implies the statement of the lemma when the ideal $R_+Z_c(W)$ is contained in a maximal ideal corresponding to an Azumaya point. By \cite[Theorem 1.7]{EG} the Azumaya points of $H_c(W)$ are precisely the points in the smooth locus of $X_c(W)$, but we have assumed that $X_c(W)$ is smooth.
\end{proof}
\begin{thm}\label{thm2}
The image of $Z_c(W)$ under the composition of maps \eqref{equation7} is equal to $\mathcal{O}_p$. In particular $\mathcal{O}_p=A(\lambda)$.
\end{thm}
\begin{proof}
By Lemma~\ref{lem5} the image of $Z_c(W)$ is $Z(\overline{H}_c(W))$. By \cite[Lemma 4.5]{Cuspidal} the kernel of $q\circ i$ is $R_+Z_c(W)$ therefore, $Z_c(W)/R_+Z_c(W)=Z(\overline{H}_c(W))$. 
Write the block decomposition 
\[
Z_c(W)/R_+Z_c(W)=Z(\overline{H}_c(W))=\bigoplus_{i\in\gamma^{-1}(0)}A_i\subset  \bigoplus_{i \in \gamma^{-1}(0)}B_i.
\]
Hence $A_i=Z(B_i)$. The  image of $Z_c(W)/R_+Z_c(W)$ under the map $\phi_p$ is then $A_p$.

To localise at the point $p$ we do the following. There is a unique maximal ideal $\mathfrak{m}_p\subset A_p$ that corresponds to the point $p$ and so a maximal ideal $A_1\oplus A_2\dots\oplus \mathfrak{m}_p\oplus\dots \oplus A_r$ in $Z_c(W)/R_+Z_c(W)$. We make every element not contained in this ideal invertible. Since we had the block decomposition of $Z_c(W)/R_+Z_c(W)$ there is a set of orthogonal idempotents which we shall label $e_1,\dots, e_n$ so that $A_i=Ae_i$. Therefore the maximal ideal $A_1\oplus A_2\dots \oplus \mathfrak{m}_p\oplus\dots \oplus A_r$ contains every $e_j\neq e_p$. By localising at $p$ we have made $e_p$ invertible and so for any other $e_i$ we have $e_i=e_ie_pe_p^{-1}=0$. Therefore, 
\[
(Z_c(W)/R_+Z_c(W))_p=A_p=\phi_p\circ q\circ i(Z).
\]
\end{proof}

\begin{cor}\label{cor1}
The centre of the block $B_\lambda$ is $A(\lambda)$.
\end{cor}
\begin{proof}
Equation~\eqref{equation5} implies that $\mathcal{O}_p$ is the centre of $B_\lambda$. Therefore, Theorem~\ref{thm2} implies that  $A(\lambda)$ is the center of the block.
\end{proof}
A consequence of the above corollary is that if we can describe $A(\lambda)$ for each $\lambda\in\mathrm{Irr}\, W$ then we have described the entire centre of $\overline{H}_c(W)$. To calculate $A(\lambda)$ we require \cite[Theorem 8.14]{bellamy2018highest} which is given below and states that the algebras $A(\lambda)$ are isomorphic to the tensor product of two endomorphism rings. It is these endomorphism rings that we will explicitly describe in later sections. In \cite{bellamy2018highest} the following definitions are given
\[
A_c(\lambda)^+:=\mathrm{End}_{\overline{H}_c(W)}\Delta(\lambda)
\textrm{ and } A_c(\lambda)^-:=\mathrm{End}_{\overline{H}_c(W)}\Delta^-(\lambda),
\]
where $\Delta^-(\lambda)=\overline{H}_c(W)\otimes_{\mathbb{C}[\mathfrak{h}]^{coW}\rtimes W}\lambda$.

The definition of $A^-_c(\lambda)$ is incorrect in \cite{bellamy2018highest}. This algebra should be the endomorphism ring of a dual baby Verma module $\Delta^-_{c}(\mu)$ as written, except that we have $\mu \neq \lambda$ in general. The argument of \textit{loc. cit.} goes through provided that $\Delta^-_{c}(\mu)$ is the dual baby Verma module whose composition factors are all isomorphic to $L(\lambda)$. 

We fix gradings $\deg \h = -1, \deg \h^* = 1$ and $\deg W = 0$ so that the weights of $\Delta(\lambda)$ are all non-negative, with $\Delta(\lambda)_0 \cong \lambda$, and the weights of $\Delta^-(\mu)$ are all non-positive, with $\Delta^-(\mu)_0 \cong \mu$.  

The graded simple module $L_{c}(\lambda)$ is characterized by the fact that its lowest weight space (i.e. graded piece of lowest degree) is $L_{c}(\lambda)_0 \cong \lambda$. However, the simple quotient $L^-_{c}(\mu)$ of $\Delta^-_{c}(\mu)$ is characterized by the fact that its \textit{highest weight} space is $\mu$. Thus, we need to compute the highest weight space of $L(\lambda)$. Up to a shift in degree, $\Delta_{c}(\lambda)$ contains a graded copy of $L(\lambda)$ in its socle. Therefore, if $N$ is the largest integer for which $\Delta_{c}(\lambda)_N \neq 0$ then the highest weight space of $L(\lambda)$ must equal $\Delta_{c}(\lambda)_N$. As a graded $W$-module, $\Delta_{c}(\lambda) = \C[\h]^{co W} \otimes \lambda$, so $\Delta_{c}(\lambda)_N \cong \C[\h]^{co W}_N \otimes \lambda$. 

The space $\C[\h]^{co W}_N$ is one-dimensional, spanned by the product of roots raised to appropriate powers, see \cite[Theorem 4.18]{broue2010introduction} where it is also explained how to compute the linear character of $W$ associated to $\C[\h]^{co W}_N$. We note that $N$ is the number of reflections in $W$ \cite[Proposition 4.16]{broue2010introduction}. Let $\delta \in \C[\h]^{co W}$ span the space of highest degree. Then $\delta$ is a $W$-semi-invariant and we can twist $W$-modules by it.

\begin{lem}
Assume that the centre of $H_{c}(W)$ is regular. There exists a graded isomorphism  $f : L^-_{c}(\lambda \o \delta)[N] \ra \mathrm{Soc} \, \Delta_{c}(\lambda)$. In particular, $L^-_{c}(\lambda \o \delta) \cong L_{c}(\lambda)$ as ungraded modules.
\end{lem}

\begin{proof}
First, we construct $f : \Delta^-_{c}(\lambda \o \delta)[N] \ra \Delta_{c}(\lambda)$. Since the space $\delta \o \lambda$ lies in the highest degree $N$ of $\Delta_{c}(\lambda)$, we have $\h^* \cdot \delta \o \lambda = 0$, where $\h^* \subset \C[\h]^{co W}$. Therefore, by definition, the $W$-isomorphism $\lambda \o \delta \ra \delta \o \lambda$ extends to a non-zero morphism $f : \Delta^-_{c}(\lambda \o \delta)[N] \ra \Delta_{c}(\lambda)$.

Since $\delta \in \C[\h]^{co W}$ is the socle of $\C[\h]^{co W}$ considered as a $\C[\h]^{co W}$-module, $\delta \o \lambda$ is the socle of $\Delta_{c}(\lambda)$ considered as a $\C[\h]^{co W} \rtimes W$-module. Let $L$ be the $H_{c}$-submodule of $\Delta_{c}(\lambda)$ generated by $\delta \o \lambda$. Then $L$ is simple and hence is in the socle of $\Delta_{c}(\lambda)$ (in fact is the socle of $\Delta_{c}(\lambda)$). To see this, let $U \subset L$ be a simple sub-module. Then $U$ is a $\C[\h]^{co W} \rtimes W$-submodule of $L$, hence must contain $\delta \o \lambda$ since this is the socle of all submodules of $\Delta_{c}(\lambda)$ as a $\C[\h]^{co W} \rtimes W$-module. Hence $U = L$.

Since the image of $f : \Delta^-_{c}(\lambda \o \delta) \ra \Delta_{c}(\lambda)$ is in the socle, $f$ must factor through the head of $\Delta^-_{c}(\lambda \o \delta)$ which is $L^-_{c}(\lambda \o \delta)$.
\end{proof}

In the case of the wreath product $S_n\wr \mathbb{Z}/\ell\mathbb{Z}$, the linear character $\nu$ afforded by $\C[\h]^{co W}_N$ can be computed using Stembridge's formula for the fake degrees of $S_n\wr \mathbb{Z}/\ell\mathbb{Z}$. We see that 
\[
\nu = (\emptyset, \dots, \emptyset, (1^n))
\]
and the fake polynomial is $t^{N}$ where $N = (\ell-1) n + \ell \frac{n(n-1)}{2} = \frac{\ell n(n+1)}{2} -n$. This is indeed the number of reflections in $S_n\wr \mathbb{Z}/\ell\mathbb{Z}$. 

Let $W(\lambda)$ be the irreducible $S_n$-module labelled by $\lambda$ and $\chi_i$ the irreducible $\mathbb{Z}/\ell\mathbb{Z}$-module with $\chi_i(\gamma) = \zeta^i$. One can compute that if $\underline{\lambda} = (\lambda^1, \dots, \lambda^\ell)$ and $V(\underline{\lambda})$ the associated representation then
\begin{align*}
V(\underline{\lambda}) \o V(\nu) & \cong \Ind_{S_{\underline{\lambda}} \ltimes (\mathbb{Z}/\ell\mathbb{Z})^n}^{S_n\wr \mathbb{Z}/\ell\mathbb{Z}} \left( (W(\lambda^{1}) \o \chi_1^{\o n_1}) \o \cdots \o (W(\lambda^{\ell}) \o \chi_{\ell}^{\o n_{\ell}}) \o (V(\nu) |_{S_{\underline{\lambda}} \ltimes (\mathbb{Z}/\ell\mathbb{Z})^n}) \right) \\
& \cong \Ind_{S_{\underline{\lambda}} \ltimes (\mathbb{Z}/\ell\mathbb{Z})^n}^{S_n\wr \mathbb{Z}/\ell\mathbb{Z}} \left( (W((\lambda^{1})^{T}) \o \chi_{\ell}^{\o n_1}) \o \cdots \o (W((\lambda^{\ell})^T) \o \chi_{\ell-1}^{\o n_{\ell}}) \right)\\
& = V((\lambda^{2})^T,\ds, (\lambda^{\ell})^T,(\lambda^{1})^T)
\end{align*}
where $n_i = |\lambda^{i}|$ because
\[
V(\nu) |_{S_{\underline{\lambda}} \ltimes (\mathbb{Z}/\ell\mathbb{Z})^n} = (W((1^{n_1})) \o \chi_{\ell}^{\o n_1}) \o \cdots \o (W((1^{n_{\ell}})) \o \chi_{\ell}^{\o n_{\ell}}).
\]

In summary, if we define 
\[
\underline{\lambda}^{\star} = ((\lambda^{2})^T,\ds, (\lambda^{\ell})^T,(\lambda^{1})^T),
\]
then
\[
A_{c}^-(\underline{\lambda}) = \End_{H_{c}}(\Delta_{c}^-(\underline{\lambda}^{\star} )). 
\]
It can be convenient to have a reverse operation so let us define $\underline{\lambda}^\sharp:=((\lambda^{\ell})^T,(\lambda^{1})^T,\dots,(\lambda^{\ell-1})^T)$, then $(\underline{\lambda}^\star)^\sharp=\underline{\lambda}$ and $(\underline{\lambda}^\sharp)^\star=\underline{\lambda}$. Therefore,
\[
A_{c}^-(\underline{\lambda}^\sharp) = \End_{H_{c}}(\Delta_{c}^-(\underline{\lambda})). 
\]
With these details clarified we now present \cite[Theorem 8.14]{bellamy2018highest}.
\begin{thm}\label{thm3}
Multiplication induces an isomorphism
\[
A_c(\lambda)^-\otimes_{\mathbb{C}} A_c(\lambda)^+\cong A_c(\lambda)
\]
\end{thm}
It turns out that in the case of the wreath product we need only describe one of these rings, as we can then deduce a description of the other.

The final theorem of this section proves the algebra $A_{c'}(\mu)^-$ is isomorphic to $A_c(\lambda)^+$, but for different generic $c'$ and simple modules $\lambda$, $\mu$. It is crucial to the description of $c'$ that we set some notation for $H_c(S_n\wr\mathbb{Z}/\ell\mathbb{Z})$. Fix a generator $\gamma\in\mathbb{Z}/\ell\mathbb{Z}$, let $\gamma_{i}\in S_n\wr\mathbb{Z}/\ell\mathbb{Z}$ denote the element $\gamma$ in the $i^{th}$ component. Then $\gamma_i s_{i,j}=s_{i,j}\gamma_j$, where $s_{i,j}\in S_n$ is the transposition of $i$ and $j$. The space $\mathfrak{h}$ has basis $\{x_1,\dots, x_n\}$ with action $\gamma_i x_i=\omega x_i$ and $\gamma_ix_j=x_j$ for $j\neq i$ where $\omega$ is a primitive $\ell^{th}$ root of unity. The permutations act as follows $s x_i=x_{s(i)}$.

The defining relations for $H_c(S_n\wr\mathbb{Z}/\ell\mathbb{Z})$ are 
\begin{equation}\label{equation30}
  \begin{aligned}
      & [x_i,x_j]=0,\quad [y_i,y_j]=0,\\
      & [y_i,x_i]=c(s_{i,j}\gamma_{i}^k\gamma_j^{-k}) \sum_{i\neq j}\sum_{k=0}^{\ell} s_{i,j}\gamma_i^k\gamma_j^{-k} + \sum_{k=1}^{\ell-1}c(\gamma_{i}^k)\gamma_i^k\\
      & [y_i,x_j]= -c(s_{i,j}\gamma_i^k\gamma_j^{-k})\sum_{k=1}^\ell\omega^k s_{i,j}\gamma_i^k\gamma_j^{-k}.
  \end{aligned}
\end{equation}

Then parameter $c'$ is defined by
\[
c'(s_{i,j}\gamma_{i}^k\gamma_j^{-k}):=c(s_{i,j}\gamma_{i}^k\gamma_j^{-k})\textnormal{ and } c'(\gamma_i^k):=c(\gamma_i^{-k}).
\]
The irreducible modules of $S_n\wr\mathbb{Z}/\ell\mathbb{Z}$ are labeled by $\ell$-multipartitions of $n$. If $\underline{\lambda}=(\lambda^1,\dots, \lambda^\ell)$ then define $\underline{\lambda}^*:=(\lambda^1,\lambda^\ell,\dots,\lambda^2)$.

Given an isomorphism $\psi:H_{c'}(W)\rightarrow H_c(W)$ and a $H_c(W)$-module $M$, we can make a $H_{c'}(W)$-module $M^\psi$ that, as a set equals $M$, but with the action given by
\[
h\cdot m=\psi(h) m,
\] 
for $h\in H_{c'}(W)$ and $m\in M$.
\begin{thm}\label{thm4}
There is an anti-graded isomorphism 
\[
A_{c'}(\underline{\mu})^-\cong A_c(\underline{\lambda})^+,
\]
where if $\underline{\lambda}=(\lambda^{1},\dots,\lambda^{\ell})$ then $\underline{\mu}=((\lambda^{2})^T,(\lambda^{1})^T,(\lambda^{(\ell)})^T,\dots,(\lambda^{3})^T)$.
Moreover $c'$ is generic if and only if $c$ is generic.
\end{thm}
\begin{proof}
The map $\phi:H_c(S_n\wr\mathbb{Z}/\ell\mathbb{Z})\rightarrow H_{c'} (S_n\wr\mathbb{Z}/\ell\mathbb{Z})$ given by 
\[
\phi(x_i)=y_i,\quad \phi(y_i)=-x_i,\quad \phi(s)=s\textnormal{ and } \phi(\gamma_i)=\gamma_i^{-1}
\]
for all $i$ and for all $s\in S_n$ is an anti-graded isomorphism of algebras. Recall the baby Verma module \eqref{equation6}
\[
\Delta_c(\underline{\lambda}):=\overline{H}_c\otimes_{\mathbb{C}[\mathfrak{h}]^{coW}\rtimes W}\underline{\lambda}
\]
and 
\[
\Delta_c^-(\underline{\lambda}):=\overline{H}_c\otimes_{\mathbb{C}[\mathfrak{h}^*]^{coW}\rtimes W}\underline{\lambda}.
\]
In an abuse of notation let $\phi$ denote its restriction to the group algebra $\mathbb{C}S_n\wr\mathbb{Z}/\ell\mathbb{Z}$, then as $\mathbb{C}S_n\wr\mathbb{Z}/\ell\mathbb{Z}$-modules $\underline{\lambda}^\phi=\underline{\lambda}^*$; this follows from the construction of irreducible $\mathbb{C}S_n\wr\mathbb{Z}/\ell\mathbb{Z}$-modules as in \cite[Section 5.3]{MoDegen}. Since $\phi(\mathbb{C}[\mathfrak{h}]^W_+)=\mathbb{C}[\mathfrak{h}^*]^W_+$, this then implies that 
\[
\Delta_{c}(\underline{\lambda})^\phi= \Delta^-_{c'}(\underline{\lambda}^*).
\]
Now $M\rightarrow M^\phi$ is a functor $H_c$-mod$\rightarrow H_{c'}$-mod which is an equivalence. This means that the map
\[
\mathrm{End}_{H_c}(\Delta_c(\underline{\lambda}))\rightarrow \mathrm{End}_{H_{c'}}(\Delta_{c}(\underline{\lambda})^\phi)=\mathrm{End}_{H_{c'}}(\Delta^-_{c'}(\underline{\lambda}^*)) 
\]
is an isomorphism. It implies that 
\[
A_{c'}(\underline{\mu})^-=\mathrm{End}_{H_{c'}}\Delta_{c'}^-(\underline{\lambda}^*)\cong \mathrm{End}_{H_c}\Delta_c(\underline{\lambda})=A_c(\underline{\lambda})^+,
\]
and so $\underline{\mu}=(\underline{\mu}^\star)^\sharp=(\underline{\lambda}^*)^\sharp$.
Since $\phi$ is an isomorphism, $Z(H_c(S_n\wr\mathbb{Z}/\ell\mathbb{Z}))$ is regular if and only if $Z(H_{c'}(S_n\wr\mathbb{Z}/\ell\mathbb{Z}))$ is regular. Thus $c$ is generic if and only if $c'$ is generic. Therefore, if we know $A(\underline{\lambda})^+$ for generic $c$ and for all $\underline{\lambda}$  then we know $A(\underline{\mu})^-$ for generic $c'$ and for all $\underline{\mu}$.
\end{proof}
\begin{remark}
In the case of the symmetric group $\ell=1$, $c'=c$ and $\mu=\lambda^T$. Therefore $A_c(\lambda)^+\cong A_c(\lambda^T)^-$.
\end{remark}
Later in the paper after we explain the specific $c$ required for the isomorphism of Bonnaf\'e and Maksimau we will see that Theorem~\ref{thm4} cannot be directly applied in the wreath product case. Instead we will have to find both $A(\lambda)^+$ and $A(\lambda)^-$ using the theory in Section 5. Then in Section 7 we will complete the argument, showing that we can indeed deduce a similar result to Theorem~\ref{thm4} for the wreath product case and the correct $c$. This will reduce the problem of finding the entire centre to simply giving an explicit description of $A(\lambda)^+$, the endomorphism rings of baby Verma modules.

\section{The symmetric group}
The aim of this section is to show that the endomorphism rings $A(\lambda)^+$ are isomorphic to the ring of functions on the preimage of a particular map of spectra denoted $\pi$. We begin with the following theorem which states two key facts (for any complex reflection group $W$). The first is that the baby Verma modules are quotients of the Verma modules. Recall that the Verma module is defined as $\underline{\Delta}_c(\lambda):=H_c(W)\otimes_{\mathbb{C}[\mathfrak{h}^*]\rtimes W}\lambda$,
where $\mathfrak{h}\subset \mathbb{C}[\mathfrak{h}^*]$ acts by $0$ on $\lambda$, its opposite is defined similarly $\underline{\Delta}_c(\lambda)^-:=H_c(W)\otimes_{\mathbb{C}[\mathfrak{h}]\rtimes W}\lambda$. The second, is that for any irreducible representation $\lambda\in\mathrm{Irr}W$ there is a surjection from the centre of the restricted rational Cherednik algebra onto the endomorphism ring of the Verma module when $X_c(W)$ is smooth.
\begin{thm}\label{thm5}
For all $\lambda\in\mathrm{Irr}\, W$,
\begin{enumerate}
    \item $\Delta(\lambda)=\underline{\Delta}(\lambda)/R_+\underline{\Delta}(\lambda)$ and $\Delta^-(\underline{\lambda})=\underline{\Delta}^-(\underline{\lambda})/R_+\underline{\Delta}^-(\underline{\lambda})$.
    \item The map defined by multiplication by elements of $Z_c(W)$ on $\underline{\Delta}(\lambda)$ and $\underline{\Delta}^-(\underline{\lambda})$ as $H_c(W)$-modules gives surjections $Z_c(W)\twoheadrightarrow \mathrm{End}_{H_c(W)}(\underline{\Delta}(\lambda))$, $Z_c(W)\twoheadrightarrow \mathrm{End}_{H_c(W)}(\underline{\Delta}^-(\lambda))$.
\end{enumerate}
\end{thm}
\begin{proof}
Statement $(1)$ follows from the PBW Theorem \cite[Theorem 1.3]{EG}. For a proof of $(2)$ see \cite[Theorem 2]{BellEnd}.
\end{proof}
Let us consider the implications of the second fact. There are the following compositions of maps, in both cases the first is the inclusion map and the second the surjection from Theorem~\ref{thm5} $(2)$
\[
\mathbb{C}[\mathfrak{h}]^W\hookrightarrow Z_c(W)\twoheadrightarrow \mathrm{End}_{H_c(W)}(\underline{\Delta}(\lambda))
\]
\[
\mathbb{C}[\mathfrak{h}^*]^W\hookrightarrow Z_c(W)\twoheadrightarrow \mathrm{End}_{H_c(W)}(\underline{\Delta}^-(\lambda)).
\]
These induce maps of spectra
\begin{equation}\label{equation8}
\pi: \mathrm{Spec}\, \mathrm{End}_{H_c(W)}(\underline{\Delta}(\lambda))\rightarrow \mathrm{Spec}\, \mathbb{C}[\mathfrak{h}]^W=\mathfrak{h}/W.
\end{equation}
\begin{equation}\label{equation19}
\pi^-: \mathrm{Spec}\, \mathrm{End}_{H_c(W)}(\underline{\Delta}^-(\lambda))\rightarrow \mathrm{Spec}\, \mathbb{C}[\mathfrak{h}^*]^W=\mathfrak{h}^*/W.
\end{equation}
From the map~\eqref{equation8} we see that 
\begin{equation}\label{equation9}
\mathbb{C}[\pi^{-1}(0)]\cong\mathrm{End}_{H_c(W)}(\underline{\Delta}(\lambda))/\mathbb{C}[\mathfrak{h}]^W_+\mathrm{End}_{H_c(W)}(\underline{\Delta}(\lambda)).
\end{equation}
Similarly from~\eqref{equation19} we get
\begin{equation}\label{equation20}
\mathbb{C}[(\pi^-)^{-1}(0)]\cong\mathrm{End}_{H_c(W)}(\underline{\Delta}^-(\lambda))/\mathbb{C}[\mathfrak{h}^*]^W_+\mathrm{End}_{H_c(W)}(\underline{\Delta}^-(\lambda)).
\end{equation}
The goal of this section is to prove that the right hand side of \eqref{equation9} is in fact isomorphic to $A(\lambda)^+$. This is not obvious as the rings $A(\lambda)^+$ are the endomorphism rings of the baby Verma modules whereas the right hand side of \eqref{equation8} is a quotient of the endomorphism ring of the Verma module. We do not prove an equivalent statement for $A(\lambda)^-$, it will be enough for us to prove the following isomorphisms
\[
\mathrm{End}_{\overline{H}_c(W)}(\Delta(\lambda))\cong \mathrm{End}_{H_c(W)}(\underline{\Delta}(\lambda))/\mathbb{C}[\mathfrak{h}]^W_+\mathrm{End}_{H_c(W)}(\underline{\Delta}(\lambda))
\]
\[
\mathrm{End}_{\overline{H}_c(W)}(\Delta^-(\lambda))\cong \mathrm{End}_{H_c(W)}(\underline{\Delta}^-(\lambda))/\mathbb{C}[\mathfrak{h}^*]^W_+\mathrm{End}_{H_c(W)}(\underline{\Delta}^-(\lambda)).
\]
Let
\[
e:=\frac{1}{|W|}\sum_{w\in W}w
\]
denote the trivial idempotent in $\mathbb{C}W\subset H_c(W)$. Recall that there is an isomorphism $Z_c(W)\cong eH_c(W)e$ given by the map $z\rightarrow z\cdot e$ \cite[Theorem 3.1]{EG}, called the Satake isomorphism.
\begin{lem}\label{lem6}
Let $A$ be a finitely generated algebra and $e$ be an idempotent of $A$. For any $A$-module $M$ we have the following isomorphism
\[
eA\otimes_A M\cong eM.
\]
\end{lem}
\begin{proof}
Define a homomorphism $\phi:eA\otimes_A M\rightarrow  eM$ as follows $\phi(ea\otimes m)=eam$. We shall prove this is an isomorphism. It is clearly surjective and a morphism so we need only prove that it is injective. We prove that the kernel of $\phi$ is $0$. If $\phi(ea\otimes m)=0$ then $eam=0$, but note that $ea\otimes m=e^2a\otimes m=e\otimes eam=e\otimes 0=0$.  
\end{proof}
The following is a standard fact \cite[Corollary 1.6.3]{BellamySRAlecturenotes}.
\begin{thm}\label{thm7}
The functor
\[
e:H_c(W)\mathrm{-mod}\rightarrow eH_c(W)e\mathrm{-mod}
\]
is an equivalence of categories if and only if $e\cdot M=0$ implies that $M=0$ for all $M\in H_c(W)\mathrm{-mod}$.
\end{thm}
\begin{lem}\label{lem7}
The $Z_c(W)$-modules $e\underline{\Delta}(\lambda)$ and $e\underline{\Delta}^-(\lambda)$ are cyclic.
\end{lem}
\begin{proof}
In \cite[Theorem 4.1]{BellEnd} it is shown that $e\underline{\Delta}(\lambda)$ is a cyclic $\mathrm{End}_{H_c(W)}(\underline{\Delta}(\lambda))$-module. We also know from Theorem~\ref{thm5} that $Z_c(W)$ surjects onto $\mathrm{End}_{H_c(W)}(\underline{\Delta}(\lambda))$, hence $e\underline{\Delta}(\lambda)$ is a cyclic $Z_c(W)$-module. By the proof of \cite[Lemma 8.13]{bellamy2018highest} $e\underline{\Delta}^-(\lambda)$ is also cyclic.
\end{proof}
We can now prove that the endomorphism ring of the baby Verma module is a quotient of the endomorphism ring of the corresponding Verma module. In particular they are exactly the quotients we desire.
\begin{thm}\label{thm8}
There are isomorphisms
\[
\mathrm{End}_{\overline{H}_c(W)}(\Delta(\lambda))\cong \mathrm{End}_{H_c(W)}(\underline{\Delta}(\lambda))/\mathbb{C}[\mathfrak{h}]^W_+\mathrm{End}_{H_c(W)}(\underline{\Delta}(\lambda))
\]
and 
\[
\mathrm{End}_{\overline{H}_c(W)}(\Delta^-(\lambda))\cong \mathrm{End}_{H_c(W)}(\underline{\Delta}^-(\lambda))/\mathbb{C}[\mathfrak{h}^*]^W_+\mathrm{End}_{H_c(W)}(\underline{\Delta}^-(\lambda))
\]
\end{thm}
\begin{proof}
The argument is identical for the opposite baby Verma module and so we prove the first isomorphism. For brevity, write $H=H_c(W)$. Theorem~\ref{thm7} states that the spherical Cherednik algebra $eHe$ is Morita equivalent to $H$, hence $eHe\mathrm{-mod}\cong H\mathrm{-mod}$. Given an endomorphism $f\in \mathrm{End}_{H}(\underline{\Delta}(\lambda))$ we have an endomorphism 
\[
\overline{f}\in \mathrm{End}_{H}(\underline{\Delta}(\lambda)/\mathbb{C}[\mathfrak{h}]^W_+\underline{\Delta}(\lambda))
\]
where $\overline{f}(m+\mathbb{C}[\mathfrak{h}]^W_+)=f(m)+\mathbb{C}[\mathfrak{h}]^W_+\underline{\Delta}(\lambda)$.
In this way we have a map 
\[
\phi:\mathrm{End}_{H}(\underline{\Delta}(\lambda))\rightarrow  \mathrm{End}_{H}(\underline{\Delta}(\lambda)/\mathbb{C}[\mathfrak{h}]^W_+\underline{\Delta}(\lambda)).
\]
We wish to show that $\mathrm{Ker}\phi=R_+\mathrm{End}_{H}(\underline{\Delta}(\lambda))$. Since $eHe$ is Morita equivalent to $H$ we have the following commutative diagram
\[
\begin{tikzpicture}[scale=2]
\node (A) at (0,1) {$\mathrm{End}_{H}(\underline{\Delta}(\lambda))$};
\node (B) at (3,1) {$\mathrm{End}_{H}(\underline{\Delta}(\lambda)/\mathbb{C}[\mathfrak{h}]^W_+\underline{\Delta}(\lambda))$};
\node (C) at (0,0) {$\mathrm{End}_{eHe}(e\underline{\Delta}(\lambda))$};
\node (D) at (3,0) {$\mathrm{End}_{eHe}(e\underline{\Delta}(\lambda)/e\mathbb{C}[\mathfrak{h}]^W_+e\underline{\Delta}(\lambda))$};
\path[->,font=\scriptsize,>=angle 90]
(A) edge node[above]{$\phi$} (B)
(A) edge node[left]{$\cong$} (C)
(B) edge node[right]{$\cong$} (D)
(C) edge node[below]{} (D);
\end{tikzpicture}
\]
By Lemma~\ref{lem7}, $e\underline{\Delta}(\lambda)$ is a cyclic $eHe$-module. Hence $e\underline{\Delta}(\lambda)\cong eHe/I$, where $I$ is the annihilator of the generator. Therefore,
\[
\mathrm{End}_{eHe}(e\underline{\Delta}(\lambda))\cong \mathrm{End}_{eHe}(eHe/I)\cong eHe/I.
\]
Similarly, 
\[
\mathrm{End}_{eHe}(e\underline{\Delta}(\lambda)/e\mathbb{C}[\mathfrak{h}]^W_+e\underline{\Delta}(\lambda))\cong \mathrm{End}_{eHe}((eHe/I)/(e\mathbb{C}[\mathfrak{h}]^W_+eHe/I)),
\]
and
\[
\mathrm{End}_{eHe}((eHe/I)/(e\mathbb{C}[\mathfrak{h}]^W_+eHe/I))\cong eHe/(\mathbb{C}[\mathfrak{h}]^W_+eHe+I).
\]
Hence we have a new commutative diagram
\[
\begin{tikzpicture}[scale=2]
\node (A) at (0,2) {$\mathrm{End}_{H}(\underline{\Delta}(\lambda))$};
\node (B) at (3,2) {$\mathrm{End}_{H}(\underline{\Delta}(\lambda)/\mathbb{C}[\mathfrak{h}]^W_+\underline{\Delta}(\lambda))$};
\node (C) at (0,1) {$\mathrm{End}_{eHe}(e\underline{\Delta}(\lambda))$};
\node (D) at (3,1) {$\mathrm{End}_{eHe}(e\underline{\Delta}(\lambda)/e\mathbb{C}[\mathfrak{h}]^W_+e\underline{\Delta}(\lambda))$};
\node (E) at (0,0) {$eHe/I$};
\node (F) at (3,0) {$eHe/(\mathbb{C}[\mathfrak{h}]^W_+eHe+I)$};
\path[->,font=\scriptsize,>=angle 90]
(A) edge node[above]{} (B)
(A) edge node[left]{$\cong$} (C)
(B) edge node[right]{$\cong$} (D)
(D) edge node[right]{$\cong$} (F)
(C) edge node[left]{$\cong$} (E)
(E) edge node[right]{} (F)
(C) edge node[below]{} (D);
\end{tikzpicture}.
\]
It is easy to see from the diagram that the kernel of the bottom map is $\mathbb{C}[\mathfrak{h}]^W_+eHe+I$. Then, via a simple diagram chasing argument, we see that 
$\mathrm{Ker}\phi=\mathbb{C}[\mathfrak{h}]^W_+\mathrm{End}_H(\underline{\Delta}(\lambda))$. Hence
\[
  \mathrm{End}_{H}(\underline{\Delta}(\lambda)/\mathbb{C}[\mathfrak{h}]^W_+\underline{\Delta}(\lambda))\cong\mathrm{End}_{H}(\underline{\Delta}(\lambda))/\mathbb{C}[\mathfrak{h}]^W_+\mathrm{End}_{H}\underline{\Delta}(\lambda)),
\]
and by Theorem~\ref{thm5}, $\mathrm{End}_{\overline{H}}(\Delta(\lambda))\cong  \mathrm{End}_{H}(\underline{\Delta}(\lambda)/\mathbb{C}[\mathfrak{h}]^W_+\underline{\Delta}(\lambda))$.
Therefore
\[
\mathrm{End}_{\overline{H}}(\Delta(\lambda))\cong \mathrm{End}_{H}(\underline{\Delta}(\lambda))/\mathbb{C}[\mathfrak{h}]^W_+\mathrm{End}_{H}\underline{\Delta}(\lambda)).
\]
\end{proof}
Applying Theorem~\ref{thm8} proves the following important corollary.
\begin{cor}\label{cor2}
There is an isomorphism of algebras $\mathbb{C}[\pi^{-1}(0)]\cong A(\lambda)^+$. 
\end{cor}
\begin{proof}
From the map \eqref{equation8} we see that
\[
\mathbb{C}[\pi^{-1}(0)]=\mathrm{End}_{H_c(W)}(\underline{\Delta}(\lambda))/\mathbb{C}[\mathfrak{h}]^W_+\mathrm{End}_{H_c(W)}(\underline{\Delta}(\lambda)).
\]
Hence, by Theorem~\ref{thm8}, $\mathbb{C}[\pi^{-1}(0)]\cong \mathrm{End}_{\overline{H}_c(W)}(\Delta(\lambda))=A(\lambda)^{+}$.
\end{proof}
Corollary~\ref{cor2} reduces the problem of understanding the blocks of $\overline{H}_c(W)$ to understanding the scheme theoretic fiber of the map $\pi$. 
\section{The wreath product group}
Describing the centre of the restricted rational Cherednik algebra of $S_n\wr\mathbb{Z}/\ell\mathbb{Z}$ requires more finesse than the $S_n$ case. We cannot simply use Corollary~\ref{cor2}, as we shall see in Section 6 the connection between $\pi$ and the Wronskian only holds in the case of the symmetric group. Instead we will need to construct two maps $\pi_{n\ell}$ and $\pi_{n,\ell}$ (and their opposites $\pi^-_{n\ell}$ and $\pi^-_{n,\ell}$) for the symmetric group and the wreath product group respectively. Then we will prove a generalisation of Corollary~\ref{cor2}. Key to all of this is the following isomorphism \cite[Theorem 4.21]{bonnafe2021fixed} which provides a connection between $Z_c(S_n)$ and $Z_c(S_n\wr\mathbb{Z}/\ell\mathbb{Z})$ via their respective Calogero-Moser spaces.
\begin{thm}\label{thm9}
Let $\sigma\in \mathbb{Z}/\ell\mathbb{Z}\subset \mathbb{C}^\times$ be a root of unity and assume $X_c(W)$ is smooth. Then $X_c(W)^{\sigma}$, the subscheme of $X_c(W)$ fixed by the action of $\sigma$, is smooth. For each irreducible component $X_0\subset X_c(W)^{\sigma}$ there exists a reflection subquotient $W'\subset W$ and conjugacy function $\overline{c}$ such that there is a $\mathbb{C}^\times$-equivariant isomorphism of varieties
\begin{equation}\label{equation10}
    X_0\cong X_{\overline{c}}(W').
\end{equation}
\end{thm}
The function $\overline{c}$ in the above theorem is what forces our results to hold for a special choice of parameter. This is a very important point and we will return to it with a full description of how these parameters are defined and an explicit description for the wreath product case in Theorem~\ref{thm15}.

The following results hold for generic $c$ such that the Calogero-Moser space $X_c(W)$ is smooth, and so we will ignore the distinction between $c$ and $\overline{c}$ until Theorem~\ref{thm15}.

It is important at this stage to fix notation for the irreducible representations of $S_n\wr \mathbb{Z}/\ell\mathbb{Z}$. There is a bijection between the $\ell$-multipartitions of $n$ and the irreducible representations of $S_n\wr \mathbb{Z}/\ell\mathbb{Z}$ \cite[Page. 221]{MoDegen}. Then, by Theorem~\ref{thm1}, it makes sense to denote an irreducible representation of $S_n\wr \mathbb{Z}/\ell\mathbb{Z}$ as $\mathrm{quo}_\ell(\lambda)$, for a partition $\lambda\vdash n\ell$ with trivial $\ell$-core. The reason why we choose to label the irreducible representations by the $\ell$-quotients will become clearer later in the section, particularly in light of Lemma~\ref{lem12}. Note that when $\ell=1$ we are back in the symmetric case and $\mathrm{quo}_\ell(\lambda)=\lambda$. From this point on we will write $\mathrm{End}_{H_c(S_n\wr \mathbb{Z}/\ell\mathbb{Z})}\underline{\Delta}(\mathrm{quo}_\ell(\lambda))$ as $\mathrm{End}\underline{\Delta}(\mathrm{quo}_\ell(\lambda))$ (respectively $\mathrm{End}\underline{\Delta}^-(\mathrm{quo}_\ell(\lambda))$) for brevity.

As in \eqref{equation8}, we can construct two maps 
\[
\pi_{n,\ell}:\mathrm{Spec}\,\mathrm{End}\underline{\Delta}(\mathrm{quo}_\ell(\lambda))\rightarrow \mathbb{C}^n/(S_n\wr\mathbb{Z}/\ell\mathbb{Z}),
\]
and
\[
\pi_{n\ell}: (\mathrm{Spec}\, \mathrm{End}(\underline{\Delta}(\lambda)))^{\mathbb{Z}/\ell\mathbb{Z}}\rightarrow (\mathbb{C}^{n\ell}/S_{n\ell})^{\mathbb{Z}/\ell\mathbb{Z}}.
\]
Similarly for $A(\lambda^-)$ an analogue of Theorem~\ref{thm5} applies by \cite[Corollary 4.4]{BellEnd} and we get
\[
\pi^-_{n,\ell}:\mathrm{Spec}\,\mathrm{End}\underline{\Delta}^-(\mathrm{quo}_\ell(\lambda))\rightarrow (\mathbb{C}^n)^*/(S_n\wr\mathbb{Z}/\ell\mathbb{Z}),
\]
and
\[
\pi^-_{n\ell}: (\mathrm{Spec}\, \mathrm{End}(\underline{\Delta}^-(\lambda)))^{\mathbb{Z}/\ell\mathbb{Z}}\rightarrow ((\mathbb{C}^{n\ell})^*/S_{n\ell})^{\mathbb{Z}/\ell\mathbb{Z}}
\]

The strategy is to embed $\mathrm{Spec}\,\mathrm{End}\underline{\Delta}(\mathrm{quo}_\ell(\lambda))$ (respectively $\mathrm{Spec}\,\mathrm{End}\underline{\Delta}^-(\mathrm{quo}_\ell(\lambda))$) into $X_c(S_n\wr\mathbb{Z}/\ell\mathbb{Z})$ then, using \eqref{equation10}, realise it as a subvariety of $X_c(S_{n\ell})^{\mathbb{Z}/\ell\mathbb{Z}}$. We do this by proving that $\mathrm{Spec}\,\mathrm{End}\underline{\Delta}(\mathrm{quo}_\ell(\lambda))$ and $\mathrm{Spec}\,\mathrm{End}\underline{\Delta}^-(\mathrm{quo}_\ell(\lambda))$ are equal to subvarieties called attracting sets. 
\begin{defn}
Let $X$ be an affine scheme over $\mathbb{C}$ with a $\mathbb{C}^\times$-action and assume that $X^{\mathbb{C}^\times}$ is finite. An $attracting$ $set$ for the $\mathbb{C}^\times$-action is defined to be $\Omega_p:=\{x\in X\,|\,\lim_{t\to \infty}t\cdot x=x_p\}$ where $x_p$ is a fixed point. Similarly define the $opposite$ $attracting$ $set$ $\Omega_p^-:=\{x\in X\,|\,\lim_{t\to 0}t\cdot x=x_p\}$.
\end{defn}
To prove that the spectrums of the endomorphism rings can be identified with attracting sets in $X_c(S_n\wr\mathbb{Z}/\ell\mathbb{Z})$ we show that two equalities hold in the $A(\lambda)^+$ case. The first is 
\begin{equation}\label{equation11}
\mathrm{Spec}\,\mathrm{End}\underline{\Delta}(\mathrm{quo}_\ell(\lambda))= \supp_{Z_c(S_n\wr\mathbb{Z}/\ell\mathbb{Z})}(\underline{\Delta}(\mathrm{quo}_\ell(\lambda)),
\end{equation}
which is relatively straightforward. The second is
\begin{equation}\label{equation12}
\mathrm{Supp}_{Z_c(S_n\wr\mathbb{Z}/\ell\mathbb{Z})}(\underline{\Delta}(\mathrm{quo}_\ell(\lambda))=\Omega_{\mathrm{quo}_\ell(\lambda)},
\end{equation}
which requires significantly more work. We will also show that identical equalities hold for the $A(\lambda)^-$ case, with the appropriate changes to opposite attracting sets and opposite Verma modules. From now on we shorten $Z_c(S_n\wr\mathbb{Z}/\ell\mathbb{Z})$ to $Z_c$.
\begin{lem}\label{lem8}
There are equalities of supports
\[
\mathrm{Supp}_{Z_c} \underline{\Delta}(\mathrm{quo}_\ell(\lambda)) = \mathrm{Supp}_{Z_c} e\underline{\Delta}(\mathrm{quo}_\ell(\lambda))
\]
and
\[
\mathrm{Supp}_{Z_c} \underline{\Delta}^-(\mathrm{quo}_\ell(\lambda)) = \mathrm{Supp}_{Z_c} e\underline{\Delta}^-(\mathrm{quo}_\ell(\lambda))
\]
\end{lem}
\begin{proof}
Since $e\underline{\Delta}(\mathrm{quo}_\ell(\lambda))\subset \underline{\Delta}(\mathrm{quo}_\ell(\lambda))$ we have $\mathrm{ann}_{Z_c} \underline{\Delta}(\mathrm{quo}_\ell(\lambda))\subset \mathrm{ann}_{Z_c} e\underline{\Delta}(\mathrm{quo}_\ell(\lambda)) $. Hence
\[
\mathrm{Supp}_{Z_c} e\underline{\Delta}(\mathrm{quo}_\ell(\lambda)) \subset \mathrm{Supp}_{Z_c} \underline{\Delta}(\mathrm{quo}_\ell(\lambda)). 
\]
It remains to show the reverse inclusion. Consider $\mathfrak{p}\in  \mathrm{Supp}_{Z_c} \underline{\Delta}(\mathrm{quo}_\ell(\lambda))$. Then $\underline{\Delta}(\mathrm{quo}_\ell(\lambda)) \otimes_{Z_c} (Z_c)_{\mf{p}} \neq 0$.
Therefore, to prove the reverse inclusion all we need show is that $e\underline{\Delta}(\mathrm{quo}_\ell(\lambda)) \otimes_{Z_c} (Z_c)_{\mf{p}} \neq 0$. By Theorem~\ref{thm7}, the functor 
\[
e\cdot -: H_c(S_n\wr\mathbb{Z}/\ell\mathbb{Z}))\mathrm{-mod}\rightarrow eH_c(S_n\wr\mathbb{Z}/\ell\mathbb{Z}))e\mathrm{-mod}
\]
is an equivalence and hence maps the non-zero objects in $H_c(S_n\wr\mathbb{Z}/\ell\mathbb{Z}))\mathrm{-mod}$ to non-zero objects in $eH_c(S_n\wr\mathbb{Z}/\ell\mathbb{Z}))e\mathrm{-mod}$. Therefore $\underline{\Delta}(\mathrm{quo}_\ell(\lambda)) \otimes_{Z_c} (Z_c)_{\mf{p}} \neq 0$ if and only if $e\underline{\Delta}(\mathrm{quo}_\ell(\lambda)) \otimes_{Z_c} (Z_c)_{\mf{p}} \neq 0$. The argument is identical for $\underline{\Delta}^-(\mathrm{quo}_\ell(\lambda))$.
\end{proof}
We can now prove that the equality \eqref{equation11} and its associated version for opposite Verma modules holds.
\begin{thm}\label{thm10}
There are equalities of varieties
\[
\mathrm{Spec}\,\mathrm{End}\underline{\Delta}(\mathrm{quo}_\ell(\lambda))= \supp_{Z_c}\underline{\Delta}(\mathrm{quo}_\ell(\lambda))
\]
and
\[
\mathrm{Spec}\,\mathrm{End}\underline{\Delta}^-(\mathrm{quo}_\ell(\lambda))= \supp_{Z_c}\underline{\Delta}^-(\mathrm{quo}_\ell(\lambda)).
\]
\end{thm}
\begin{proof}
The argument is identical in the $\underline{\Delta}^-(\mathrm{quo}_\ell(\lambda))$ case, using the fact that $e\underline{\Delta}^-(\mathrm{quo}_\ell(\lambda))$ is a cyclic $Z_c$-module by the argument of \cite[Lemma 8.13]{bellamy2018highest}, so we will prove the first equality. Note that $e\underline{\Delta}(\mathrm{quo}_\ell(\lambda))$ is a cyclic $Z_c$-module by Lemma~\ref{lem7}. Therefore $Z_c/I\cong e\underline{\Delta}(\mathrm{quo}_\ell(\lambda))$ as left $Z_c$-modules for some ideal $I$. Since $H_c(S_n\wr\mathbb{Z}/\ell\mathbb{Z})$ is Morita equivalent to $eH_c(S_n\wr\mathbb{Z}/\ell\mathbb{Z})e$ and $eH_c(S_n\wr\mathbb{Z}/\ell\mathbb{Z})e\cong Z_c$ we have
\[
\mathrm{End}_{H_c(S_n\wr\mathbb{Z}/\ell\mathbb{Z})}(\underline{\Delta}(\mathrm{quo}_\ell(\lambda)))\cong \mathrm{End}_{eH_c(S_n\wr\mathbb{Z}/\ell\mathbb{Z})e}e\underline{\Delta}(\mathrm{quo}_\ell(\lambda))\cong \mathrm{End}_{Z_c}Z_c/I\cong Z_c/I.
 \]
Therefore
\[
\mathrm{Spec}\,\mathrm{End}_{H_c(S_n\wr\mathbb{Z}/\ell\mathbb{Z})}(\underline{\Delta}(\mathrm{quo}_\ell(\lambda)))\cong \mathrm{Spec}\,Z_c/I,
\]
and since $Z_c$ is commutative we see that $\mathrm{Spec}\,Z_c/I\cong \mathrm{Supp}_{Z_c} (e\underline{\Delta}(\mathrm{quo}_\ell(\lambda)))$. Then Lemma~\ref{lem8} implies $\mathrm{Supp}_{Z_c} (e\underline{\Delta}(\mathrm{quo}_\ell(\lambda)))\cong \mathrm{Supp}_{Z_c}(\underline{\Delta}(\mathrm{quo}_\ell(\lambda)))$. Hence
\[
\mathrm{Spec}\,\mathrm{End}_{H_c(S_n\wr\mathbb{Z}/\ell\mathbb{Z})}(\underline{\Delta}(\mathrm{quo}_\ell(\lambda)))= \mathrm{Supp}_{Z_c}(\underline{\Delta}(\mathrm{quo}_\ell(\lambda))).
\]
\end{proof}
Several technical results are required to prove the equality \eqref{equation12}. We start with the following result which is \cite[Theorem 13.4]{kemper2010course}.
\begin{thm}\label{thm11}
Let $R$ be a Noetherian local ring with maximal ideal $\mathfrak{m}$. Let $S$ be the associated graded of $R$ with respect to the $\mathfrak{m}$-adic filtration. Then $R$ is regular if and only if $S$ is a polynomial ring.
\end{thm}
Let $X$ be an affine algebraic variety over $\mathbb{C}$ that admits a $\mathbb{C}^\times$-action. This induces a grading on $\mathbb{C}[X]$. Also note that $\mathbb{Z}/\ell\mathbb{Z}\subset \mathbb{C}^\times$ by identifying the cyclic group of order $\ell$ with the $\ell^{th}$ roots of unity. Recall that the fixed point locus is defined as 
\[
X^{\mathbb{Z}/\ell\mathbb{Z}}=\{x\in X \,|\, g \cdot x=x\,\,\forall g\in \mathbb{Z}/\ell\mathbb{Z} \}
\]
which can be equivalently defined as  
\[
\mathrm{Spec}\,\left( \frac{\mathbb{C}[X]}{\langle f-g\cdot f\,|\, g\in \mathbb{Z}/\ell\mathbb{Z} \textnormal{ and } f\in \mathbb{C}[X]\rangle }\right).
\]
Here we use the notation that if $S$ is a subset of a ring $R$ then $\langle S \rangle$ denotes the ideal generated by $S$. Since the $\mathbb{C}^\times$-action defines a grading of $\mathbb{C}[X]$ we can consider the ideal $\langle \mathbb{C}[X]_{\neq 0}  \rangle$ generated by the non-degree zero elements and define 
\[
\mathbb{C}[X](0):=\frac{\mathbb{C}[X]}{\langle \mathbb{C}[X]_{\neq 0}  \rangle }.
\]
\begin{lem}\label{lem9}
Assume that $X$ is smooth. Then $X^{\mathbb{Z}/\ell\mathbb{Z}} = \mathrm{Spec}\, \mathbb{C}[X](0)$ is smooth. In particular, $\mathbb{C}[X](0)$
is reduced.
\end{lem}
\begin{proof}
We first show equality of sets 
\[
X^{\mathbb{Z}/\ell\mathbb{Z}} = \mathrm{Spec}\,\mathbb{C}[X](0).
\]
Let $f\in \mathbb{C}[X]$ be homogeneous of degree $d$, $g\in \mathbb{Z}/\ell\mathbb{Z}$ and $p\in X^{\mathbb{Z}/\ell\mathbb{Z}}$ then we have
\[
f(p)=f(g^{-1}\cdot p)=(g\cdot f)(p)=g^{d}f(p)
\]
hence $f(p)=0$ if $d\neq 0 \,\mathrm{mod} \,\ell$. Therefore $\langle \mathbb{C}[X]_{\neq 0}  \rangle$ is contained in the maximal ideal defining $p$ and $X^{\mathbb{Z}/\ell\mathbb{Z}} \subset \mathrm{Spec}\,\mathbb{C}[X](0)$.

Conversely $\mathbb{Z}/\ell\mathbb{Z}$ acts trivially on $\mathbb{C}[X](0)$ and so every point in $\mathrm{Spec}\,\mathbb{C}[X](0)$ is fixed by $\mathbb{Z}/\ell\mathbb{Z}$. Hence $ \mathrm{Spec}\, \mathbb{C}[X](0)\subset X^{\mathbb{Z}/\ell\mathbb{Z}}$.

It remains to show that $\mathrm{Spec}\,\mathbb{C}[X](0)$ is reduced. We must show that the localisation of $\mathbb{C}[X](0)$ at each point is regular. By Theorem~\ref{thm11} it is enough to show that the tangent cone of $\mathrm{Spec}\,\mathbb{C}[X](0)$ at a fixed point $p\in X^{\mathbb{Z}/\ell\mathbb{Z}}$ is a polynomial ring. Since $X$ is regular at $p$ the tangent cone at $p$ of $X$ is equal to $V:=T_p(X)$ as a $\mathbb{Z}/\ell\mathbb{Z}$-module. By \cite[Theorem 5.2]{fogarty1973fixed} the tangent cone of $X^{\mathbb{Z}/\ell\mathbb{Z}}$ at a point $p$ is equal to $V^{\mathbb{Z}/\ell\mathbb{Z}}$. It is then clear that 
\[
V^{\mathbb{Z}/\ell\mathbb{Z}}=\mathrm{Spec}\,\mathbb{C}[V](0)
\]
is affine space and $\mathbb{C}[V](0)$ is a polynomial ring. 
\end{proof}
Before we can prove \eqref{equation12} we require a better understanding of the role of the attracting sets inside the structure of $X_c(S_n\wr\mathbb{Z}/\ell\mathbb{Z})$. To do so we need two powerful theorems, Theorem~\ref{thm12} and Theorem~\ref{thm13} these are \cite[Theorem 2.3]{BBFix} and \cite[Theorem 2.5]{BBFix} respectively. These use the concept of definite actions, for the readers benefit we include the definition here.
\begin{defn}
Let $G$ be an algebraic torus, then $G\cong \mathbb{C}^\times\times\dots \times \mathbb{C}^\times=(\mathbb{C}^\times)^n$ for some positive integer $n$. Any $G$-module can be written as a direct sum of one dimensional $G$-modules. If $V$ is a $G$-module, then we can find a basis $\{v_i\}$ of $V$ such that 
\[
(g_1\dots g_n)\cdot v_i=g_1^{s_{i1}}\dots g_n^{s_{in}}v_i\textnormal { for } (g_1\dots g_n)\in G.
\]
The module $V$ is positive (respectively negative) if
\begin{enumerate}
\item $s_{ij}\geq 0$ (respectively $s_{ij}\leq 0$) for all $i$, $j$.
\item For every $i\in I$ there exists $j$ such that $s_{ij}\neq 0$.
\end{enumerate}
The module is non-negative (respectively non-positive) if (1) is satisfied. The module is fully definite (respectively definite) if there exists an isomorphism $G\cong \mathbb{C}^\times\times\dots \times \mathbb{C}^\times$ such that the module is positive (respectively non-negative).
\end{defn}
\begin{defn}
Let $\eta:G\times X\rightarrow X$ be an action of a torus on $X$ and let $a\in X^{G}$ be a non-singular closed point. The action of $\eta$ on $a$ is fully definite (respectively definite) if the $G$-module $T_a(X)$ is fully definite (respectively definite).
\end{defn}
In the above $T_a(X)$ denotes the tangent space at $a$.
\begin{thm}\label{thm12}
Let $X$ be irreducible and reduced. Let $G$ be an algebraic 
torus. If the action of $G$ on $X$ is definite at $a \in X^G$ then $X^G$ is irreducible
\end{thm}
\begin{thm}\label{thm13}
Let $G$ be an algebraic torus. Let the action of $G$ on $X$ be 
definite at $a$. If $X$ is irreducible then there exists an open $G$-invariant neighbourhood $U$ of $a$ 
which is $G$-isomorphic to $(U \cap X^G) \times V$, where $V$ is a finite-dimensional 
(fully definite) $G$-module and the action of $G$ on $(U \cap X^G) \times V$ is induced by 
the trivial action of $G$ on $U \cap X^G$ and the linear action on $V$ (determined by 
the given structure of a $G$-module).
\end{thm}
More specifically, in the proof of Theorem~\ref{thm13} the vector space $V$ is defined to be the $G$-module complement of $T_a(X^G)$ in $T_a(X)$. In our case $X^{G}$ is a finite set and so $T_a(X^G)$ is zero and $V=T_a(X)$.

The following lemma tells us that the irreducible components are precisely the attracting sets and also that the attracting sets are equal to their own tangent space at the fixed point. Following the notation of the previous two theorems set $G=\mathbb{C}^\times$.
\begin{lem}\label{lem10}
Assume $X$ is smooth and $X^{\mathbb{C}^\times}$ is finite and non-empty and that $\lim_{t\rightarrow \infty}\, t\cdot x$ exists for all $x\in X$. If the action of $\mathbb{C}^\times$ is definite at each fixed point then
\begin{enumerate}
    \item \[X=\bigsqcup_{p\in X^{\mathbb{C}^\times}}\, \Omega_p,\] where $\Omega_p$ is the attracting set of $p$. The sets $\Omega_p$ are the irreducible components of the space $X$.
    \item $\Omega_p\cong T_p(\Omega_p)$ as varieties.
\end{enumerate}
\end{lem}
\begin{proof}
1.) Since $\lim_{t\rightarrow \infty}\,t\cdot x$ exists for each $x\in X$ and limits are unique it follows that $X$ is a disjoint union of the sets $\Omega_p$. To see that the sets $\Omega_p$ are the irreducible components note that for an arbitrary irreducible component $L$ we must have that $L^{\mathbb{C}^\times}$ contains a single point. This is because of Theorem~\ref{thm12}, which states that if $L$ is irreducible then so is $L^{\mathbb{C}^\times}$, but this is clearly not the case if it consists of more than one fixed point. Furthermore it must contain at least one point, as if $L$ is an irreducible component it equals its closure and the closure of any non-empty $\mathbb{C}^\times$-stable subset of $X$ contains some fixed point. 

If $L^{\mathbb{C}^\times}=\{p\}$ then the fact that $L$ is closed implies that $\lim_{t\rightarrow \infty}\,t\cdot x=p$ for all $x\in L$. Hence $L\subset \Omega_p$. Conversely if $x\in \Omega_p$ then $\overline{\mathbb{C}^\times \cdot x}$ is an irreducible subvariety containing $x$. Since $X$ is smooth, $L$ is a connected component of $X$. So $p\in \overline{\mathbb{C}^\times\cdot x}\cap L\neq\emptyset$ implies $\overline{\mathbb{C}^\times\cdot x}\subset L$ and hence $\Omega_p\subset L$. 

2.) Since $p$ is a unique fixed point of $\Omega_p$ we apply Theorem~\ref{thm13} to $\Omega_p$ to conclude that there exists an open neighbourhood $U\subset \Omega_p$ containing $p$ such that $U\cong (U\cap X^{\mathbb{C}^\times})\times V$. By the hypothesis we have $\Omega_p^{\mathbb{C}^\times}=p$ and the discussion above states that $V=T_p(\Omega_p)$, hence $U\cong \{p\}\times T_p(\Omega_p)\cong T_p(\Omega_p)$. Now we show that $\Omega_p\subset U$. Let $x\in \Omega_p$. Then $\lim_{t\to \infty}t\cdot x=p$. Since $U$ is an open neighbourhood of $p$ we must have that $t\cdot x\in U$ for some $t$. Recall that $U$ is $\mathbb{C}^\times$-invariant and so if $t\cdot x\in U$ then we must have $x\in U$. Hence $U= \Omega_p$ and $\Omega_p\cong T_p(\Omega_p)$. 
\end{proof}
The lemma above has an analogue for $\Omega_p^-$, with the assumption that $\lim_{t\rightarrow 0}\, t\cdot x$ exists for all $x\in X$. The proof is by an identical argument.
\begin{lem}\label{lem10.5}
Assume $X$ is smooth and $X^{\mathbb{C}^\times}$ is finite and non-empty and that $\lim_{t\rightarrow 0}\, t\cdot x$ exists for all $x\in X$. If the action of $\mathbb{C}^\times$ is definite at each fixed point then
\begin{enumerate}
    \item \[X=\bigsqcup_{p\in X^{\mathbb{C}^\times}}\, \Omega^-_p,\] where $\Omega^-_p$ is the opposite attracting set of $p$. The sets $\Omega^-_p$ are the irreducible components of the space $X$.
    \item $\Omega^-_p\cong T_p(\Omega^-_p)$ as varieties.
\end{enumerate}
\end{lem}

\begin{prop}\label{prop3}
	Assume that $Y$ is a smooth, affine scheme over $\mathbb{C}$ with $Y^{\mathbb{C}^\times}$ finite. Let $\mathbb{C}[Y]=A$ and 
\[
Y^+ = \left\{ y \in Y \, \big| \, \lim_{t \to \infty} t \cdot y \textrm{ exists} \right\}.
\]
Then:
\begin{enumerate}
		\item Let $\langle A_{<0} \rangle$ denote the ideal generated by homogeneous polynomials of negative degree, then $Y^+$ is a closed subset of $Y$, defined by the vanishing of the (reduced) ideal $\langle A_{<0} \rangle$. 
		\item $Y^+ = \bigsqcup_{p \in Y^{\mathbb{C}^\times}} \Omega_p$, where $\Omega_p \cong (T_p Y)_{>0}$ as $\mathbb{C}^\times$-varieties.
		\item If $I_{A_{0}}(p) = \{ a \in A_0 \, | \, a(p) = 0 \}$ then $\Omega_p$ is the closed subset  of $Y$ defined by the reduced ideal $\langle A_{<0}, I_{A_{0}}(p) \rangle$. 
\end{enumerate}
\end{prop}
\begin{proof}
$(1)$ Similarly to the proof of Lemma~\ref{lem9} we have that $\langle A_{< 0}\rangle$ vanishes on $Y^+$. Indeed if $f\in  A_{<0}$ is homogeneous of degree $r<0$, $y\in Y^+$ and $t\in \mathbb{C}^\times$ then
\[
f(t\cdot y)=t^{-r}f(y)
\]
and
\[
\lim_{t\rightarrow \infty}f(t\cdot y)=\lim_{t\rightarrow \infty}\,t^{-r}f(y).
\]
If $f(y)\neq 0$ then the limit of $f(t\cdot y)$ does not exist. This is a contradiction. Therefore $\langle A_{<0}\rangle$ vanishes on $Y^+$. Clearly $\mathrm{Spec}\,(A/\langle A_{<0}\rangle)\subset Y$ hence $\mathrm{Spec}\,(A/\langle A_{<0}\rangle )^{\mathbb{C}^\times}\subset Y^{\mathbb{C}^\times}$ is finite. The ring $A/\langle A_{<0}\rangle$ is non-negatively graded, so $\mathrm{Spec}\, (A/\langle A_{<0}\rangle)^+=\mathrm{Spec}\,(A/\langle A_{<0}\rangle)$. By Lemma~\ref{lem10} we see that $\mathrm{Spec}\,(A/\langle A_{<0}\rangle)$ is a disjoint union of attracting sets, in particular all of its limits exist and so it is the the vanishing ideal defining $Y^+$.

Now we must check that the ideal $\langle A_{<0}\rangle$ is reduced. Since it is homogeneous, the radical of $A/\langle A_{<0}\rangle $ is homogeneous. Hence, if it is not zero there exists a fixed point that is not reduced. Therefore, it suffices to show that for every $p\in (\mathrm{Spec}\,A/\langle A_{<0}\rangle)^{\mathbb{C}^\times}$, the local ring $(A/\langle A_{<0}\rangle)_p$ is reduced.

Let $\mathfrak{m}$ denote the maximal ideal corresponding to $p$, so $\mathfrak{m}$ is stable under $\mathbb{C}^\times$. We show that $A_\mathfrak{m}/A_{<0}A_\mathfrak{m}$ is a regular local ring. Let $(T_p^* Y)_{<0}\subset \mathfrak{m}/\mathfrak{m}^2$ be the subspace spanned by all negative weight vectors and choose $N\subset \mathfrak{m}$ a homogeneous vector space lift of $(T^*_p Y)_{<0}$. We fix another homogeneous vector space lift $V\subset \mathfrak{m}$ of $\mathfrak{m}/\mathfrak{m}^2$ that contains $N$. A basis of $V$ is a regular system of parameters for $A_\mathfrak{m}$. If $\mathfrak{n}$ is the augmentation ideal of $Sym V$ then the map $Sym V\rightarrow A$ induces graded isomorphisms $\phi_q: Sym V/\mathfrak{n}^q\rightarrow A/\mathfrak{m}^q$ and $A/\mathfrak{m}^q= A_\mathfrak{m}/(\mathfrak{m}A_\mathfrak{m})^q$ for all $q\geq 1$ as $A$ is regular at $\mathfrak{m}$. Since $\mathbb{C}^\times$ acts semisimply on $A$, the quotient $A\twoheadrightarrow A/\mathfrak{m}^q$ induces surjections $A_i\twoheadrightarrow (A/\mathfrak{m}^q)_i$ for all $i$. Hence $(A/\mathfrak{m}^q )_{<0}=(A_{<0}+\mathfrak{m}^q)/\mathfrak{m}^q$. Therefore $\phi_q$ restricts to 
\[
\frac{Sym V_{<0}+\mathfrak{n}^q}{\mathfrak{n}^q}=(Sym V/\mathfrak{n}^q)_{<0}\cong (A/\mathfrak{m}^q)_{<0}=\frac{A_{<0}+\mathfrak{m}^q}{\mathfrak{m}^q}.
\]
Since $N$ is a subspace of $V$ defined by being the lift of the space of negative weightvectors we have $NSym V\subset (SymV)_{<0}Sym V$. Since the action of $\mathbb{C}^\times$ on $V$ is linear, $NSym V=(Sym V_{<0})Sym V$ as any negativley graded vector in $Sym V$ can be broken into a sum of monomials, which in particular are negativley graded weightvectors. Now we argue that $(A_{<0}A+\mathfrak{m}^q)/\mathfrak{m}^q=(NA+\mathfrak{m}^q)/\mathfrak{m}^q$. Clearly $NA+\mathfrak{m}^q/\mathfrak{m}^q\subset A_{<0}A+\mathfrak{m}^q/\mathfrak{m}^q$ so we show the opposite inclusion. First note
\[
\frac{A_{<0}+\mathfrak{m}^q}{\mathfrak{m}^q}=\phi_q\left(\frac{Sym V_{<0}+\mathfrak{n}^q}{\mathfrak{n}^q}\right)\subset \phi_q\left(\frac{NSym V+\mathfrak{n}^q}{\mathfrak{n}^q}\right)=\frac{NA+\mathfrak{m}^q}{\mathfrak{m}^q}.
\]
Now note $\frac{NA+\mathfrak{m}^q}{\mathfrak{m}^q}$ is an ideal hence $\frac{A_{<0}A+\mathfrak{m}^q}{\mathfrak{m}^q}\subset \frac{NA+\mathfrak{m}^q}{\mathfrak{m}^q}$. Since $\frac{A_{<0}A+\mathfrak{m}^q}{\mathfrak{m}^q}=\frac{A_{<0}A_{\mathfrak{m}}+\mathfrak{m}^q}{\mathfrak{m}^q}$, \cite[Lemma~2.1]{iversen1972fixed} implies that $A_{<0}A_\mathfrak{m}$ is generated by the regular sequence $N$. Since $A_\mathfrak{m}$ is regular this implies that $A_\mathfrak{m}/A_{<0}A_{\mathfrak{m}}$ is a regular local ring. 

$(2)$ This is simply an application of Lemma~\ref{lem10} as $Y^+$ is smooth and the action of $\mathbb{C}^\times$ is definite at each fixed point.

$(3)$ Let $a\in A_0$ and $y\in \Omega_p$, then by Lemma~\ref{lem9} this is reduced.
\[
a(y)=t^0a(y)=(t\cdot a)(y)=a(t^{-1}y)
\]
hence $a(y)$ is a constant. Therefore $I_{A_0}(p)$ vanishes on $\Omega_p$ and the zero set of $\langle I_{A_0}(p)\rangle $ is equal to $\Omega_p$ as sets.  
\end{proof}
By considering the opposite attracting sets $\Omega^-_p$ we get a dual result to Proposition~\ref{prop3}. The argument is identical to that already presented, simply replacing the sets $\Omega_p$ with $\Omega^-_p$ and swapping positive and negative gradings.
\begin{prop}\label{prop3.5}
	Assume that $Y$ is a smooth, affine scheme over $\mathbb{C}$ with $Y^{\mathbb{C}^\times}$ finite. Let $\mathbb{C}[Y]=A$ and 
\[
Y^- = \left\{ y \in Y \, \big| \, \lim_{t \to 0} t \cdot y \textrm{ exists} \right\}.
\]
Then:
\begin{enumerate}
		\item Let $\langle A_{>0} \rangle$ denote the ideal generated by homogeneous polynomials of positive degree, then $Y^-$ is a closed subset of $Y$, defined by the vanishing of the (reduced) ideal $\langle A_{>0} \rangle$. 
		\item $Y^- = \bigsqcup_{p \in Y^{\mathbb{C}^\times}} \Omega^-_p$, where $\Omega^-_p \cong (T_p Y)_{<0}$ as $\mathbb{C}^\times$-varieties.
		\item If $I_{A_{0}}(p) = \{ a \in A_0 \, | \, a(p) = 0 \}$ then $\Omega^-_p$ is the closed subset  of $Y$ defined by the reduced ideal $\langle A_{>0}, I_{A_{0}}(p) \rangle$. 
\end{enumerate}
\end{prop}
As explained in \cite[Page. 703]{BellEnd} the fixed points set $X_c(S_n\wr\mathbb{Z}/\ell\mathbb{Z})^{\mathbb{C}^\times}$ is precisely $\gamma^{-1}(0)$. Therefore, by Proposition~\ref{prop1}, we can identify $\mathrm{Irr}\, S_n\wr\mathbb{Z}/\ell\mathbb{Z}\xrightarrow{\sim} X_c(S_n\wr\mathbb{Z}/\ell\mathbb{Z})^{\mathbb{C}^\times}$ by the map $\mathrm{quo}_\ell(\lambda)\mapsto x_{\mathrm{quo}_\ell(\lambda)}$. It is now possible to prove the equality \eqref{equation12}.
\begin{lem}\label{lem11}
Assume that $X_c(S_n\wr\mathbb{Z}/\ell\mathbb{Z})$ is smooth. Then there is an equality of varieties 
\[
\mathrm{Supp}_{Z_c}\underline{\Delta}(\mathrm{quo}_\ell(\lambda))=\Omega_{\mathrm{quo}_\ell(\lambda)}\quad and\quad \mathrm{Supp}_{Z_c}\underline{\Delta}^-(\mathrm{quo}_\ell(\lambda))=\Omega^-_{\mathrm{quo}_\ell(\lambda)}.
\]
\end{lem}
\begin{proof}
Throughout this proof denote $X_c(S_n\wr\mathbb{Z}/\ell\mathbb{Z})$ by $X$. Note that the Verma module is positively graded with the degree zero part equal to $1\otimes \mathrm{quo}_\ell(\lambda)$. Let $z\in Z_c(S_n\wr\mathbb{Z}/\ell\mathbb{Z})$ be a negatively graded element. Then 
\[
z\cdot x\otimes\mathrm{quo}_\ell(\lambda)=zx\otimes \mathrm{quo}_\ell(\lambda)=x(z\otimes \mathrm{quo}_\ell(\lambda))=0,
\]
as $z\otimes \mathrm{quo}_\ell(\lambda)$ has negative degree and $\underline{\Delta}(\mathrm{quo}_\ell(\lambda))$ is positively graded. Therefore the annihilator of $\underline{\Delta}(\mathrm{quo}_\ell(\lambda))$ contains all the negatively graded elements of $Z_c(S_n\wr\mathbb{Z}/\ell\mathbb{Z})$. If we denote the ideal generated by the negatively graded elements by $I_-$ then $I_-\subset \mathrm{ann}_{Z_c}\underline{\Delta}(\mathrm{quo}_\ell(\lambda))$. Hence $\mathrm{Supp}_{Z_c}\underline{\Delta}(\mathrm{quo}_\ell(\lambda))\subset V(I_-)$. By Proposition~\ref{prop3} $(1)$ and $(2)$ we see that $\mathrm{Supp}_{Z_c} \underline{\Delta}(\mathrm{quo}_\ell(\lambda))$ is contained in one of the connected components of $X^+$. Since $x_{\mathrm{quo}_\ell(\lambda)}\in \mathrm{Supp}_{Z_c}(\underline{\Delta}(\mathrm{quo}_\ell(\lambda)))$ we have $\mathrm{Supp}_{Z_c}\underline{\Delta}(\mathrm{quo}_\ell(\lambda))\subset \Omega_{\mathrm{quo}_\ell(\lambda)}$. We argue that this containment is actually an equality by proving that $\dim \mathrm{Supp}_{Z_c}\underline{\Delta}(\mathrm{quo}_\ell(\lambda))=\dim \Omega_{\mathrm{quo}_\ell(\lambda)}$. This suffices since $\Omega_{\mathrm{quo}_\ell(\lambda)}$ is an irreducible variety and $\mathrm{Supp}_{Z_c}\underline{\Delta}(\mathrm{quo}_\ell(\lambda))$ a closed subset of $\Omega_{\mathrm{quo}_\ell(\lambda)}$.

By Theorem~\ref{thm10} we have the equality $\mathrm{Supp}_{Z_c}\underline{\Delta}(\mathrm{quo}_\ell(\lambda))=\mathrm{Spec}\, \mathrm{End}\underline{\Delta}(\mathrm{quo}_\ell(\lambda))$ and therefore the dimensions are equal, $\dim \mathrm{Supp}_{Z_c}\underline{\Delta}(\mathrm{quo}_\ell(\lambda))=\dim \mathrm{Spec}\, \mathrm{End}\underline{\Delta}(\mathrm{quo}_\ell(\lambda))$. But $\dim \mathrm{Spec}\, \mathrm{End}\underline{\Delta}(\mathrm{quo}_\ell(\lambda))$ is equal to the Krull dimension of $\mathrm{End}\underline{\Delta}(\mathrm{quo}_\ell(\lambda))$. Since $\mathrm{End}\underline{\Delta}(\mathrm{quo}_\ell(\lambda))$ is a finite free module over $\mathbb{C}[\mathfrak{h}]^{S_n\wr\mathbb{Z}/\ell\mathbb{Z}}$ it has Krull dimension equal to $\dim \mathfrak{h}$ by \cite[Corollary 1.4.5]{Benson}. From Lemma~\ref{lem10}, we see that $\Omega_{\mathrm{quo}_\ell(\lambda)}\cong T_{\mathrm{quo}_\ell(\lambda)}(\Omega_{\mathrm{quo}_\ell(\lambda)})$, so we need show that $\dim T_{\mathrm{quo}_\ell(\lambda)}(\Omega_{\mathrm{quo}_\ell(\lambda)})\leq\dim \mathfrak{h}$.

Since $x_{\mathrm{quo}_\ell(\lambda)}$ is a fixed point we have the following inclusions of $\mathbb{C}^\times$-submodules $T_{\mathrm{quo}_\ell(\lambda)}(x_{\mathrm{quo}_\ell(\lambda)})\subset T_{\mathrm{quo}_\ell(\lambda)}(\Omega_{\mathrm{quo}_\ell(\lambda)})\subset T_{\mathrm{quo}_\ell(\lambda)}(X)$. We can decompose $T_{\mathrm{quo}_\ell(\lambda)} (X)=T_{-}\oplus T_0\oplus T_+$ into the negatively graded part, the degree zero part and the positively graded part. From \cite[Theorem 5.2]{fogarty1973fixed} we have that $T_{\mathrm{quo}_\ell(\lambda)}(x_{\mathrm{quo}_\ell(\lambda)})=T_{\mathrm{quo}_\ell(\lambda)} (X^{\mathbb{C}^\times})=T_0$. Now Lemma~\ref{lem9} says that $X^{\mathbb{C}^\times}$ is smooth hence $T_{\mathrm{quo}_\ell(\lambda)}(x_{\mathrm{quo}_\ell(\lambda)})=\{0\}$ and so $T_0=\{0\}$. The fixed point $x_{\mathrm{quo}_\ell(\lambda)}$ is in the smooth locus and \cite[Theorem 7.8]{PoissonOrders} implies that $T_{\mathrm{quo}_\ell(\lambda)} (X)$ is a symplectic vector space. The symplectic form on $T_{\mathrm{quo}_\ell(\lambda)} (X)$ is $\mathbb{C}^\times$-invariant hence its non-degeneracy forces $\dim T_-=\dim T_+$. Since $X$ is smooth we have $\dim X=\dim T_{\mathrm{quo}_\ell(\lambda)}(X)$. Since $Z_c(S_n\wr\mathbb{Z}/\ell\mathbb{Z})=\mathbb{C}[X]$ is a finite free module over $\mathbb{C}[\mathfrak{h}]^{S_n\wr\mathbb{Z}/\ell\mathbb{Z}}\otimes \mathbb{C}[\mathfrak{h}^*]^{S_n\wr\mathbb{Z}/\ell\mathbb{Z}}$ we have 
\[
\dim Z_c(S_n\wr\mathbb{Z}/\ell\mathbb{Z})=\dim \mathbb{C}[\mathfrak{h}]^{S_n\wr\mathbb{Z}/\ell\mathbb{Z}}\otimes \mathbb{C}[\mathfrak{h}^*]^{S_n\wr\mathbb{Z}/\ell\mathbb{Z}}=2\dim \mathfrak{h}.
\]
This means that $\dim X=2\dim\mathfrak{h}$. Therefore $\dim T_+=\dim \mathfrak{h}$. Since $T_{\mathrm{quo}_\ell(\lambda)}(\Omega_{\mathrm{quo}_\ell(\lambda)})$ is positively graded we have $T_{\mathrm{quo}_\ell(\lambda)}(\Omega_{\mathrm{quo}_\ell(\lambda)})\subset T_+$ hence $\dim T_{\mathrm{quo}_\ell(\lambda)}(\Omega_{\mathrm{quo}_\ell(\lambda)})\leq \dim \mathfrak{h}$. 

The second equality follows by the same argument swapping the positive to negative throughout and considering the opposite Verma module $\underline{\Delta}^-(\mathrm{quo}_\ell(\lambda))$.
\end{proof}
Thus the spectrum of the endomorphism ring of any given Verma module is equal to a corresponding attracting set. Therefore, $\mathrm{Spec}\,\mathrm{End}\underline{\Delta}(\mathrm{quo}_\ell(\lambda))$ and $\mathrm{Spec}\,\mathrm{End}\underline{\Delta}^-(\mathrm{quo}_\ell(\lambda))$ can be realised as subvarieties of $X_c(S_n\wr\mathbb{Z}/\ell\mathbb{Z})$. 
\begin{thm}\label{thm14}
For any $\mathrm{quo}_\ell(\lambda)\in \mathrm{Irr}\, S_n\wr\mathbb{Z}/\ell\mathbb{Z}$ we have isomorphisms of varieties 
\[
\mathrm{Spec}\,\mathrm{End}\underline{\Delta}(\mathrm{quo}_\ell(\lambda))= \Omega_{\mathrm{quo}_\ell(\lambda)},\quad
\mathrm{Spec}\,\mathrm{End}\underline{\Delta}^-(\mathrm{quo}_\ell(\lambda))= \Omega^-_{\mathrm{quo}_\ell(\lambda)}.
\]
\end{thm}
\begin{proof}
Theorem~\ref{thm10} states $\mathrm{Spec}\,\mathrm{End}\underline{\Delta}(\mathrm{quo}_\ell(\lambda))=\mathrm{Supp}_{Z_c}\underline{\Delta}(\mathrm{quo}_\ell(\lambda))$ for any $\mathrm{quo}_\ell(\lambda)\in \mathrm{Irr}\, S_n\wr\mathbb{Z}/\ell\mathbb{Z}$ and, by Lemma~\ref{lem11}, $\mathrm{Supp}_{Z_c} \underline{\Delta}(\mathrm{quo}_\ell(\lambda))=\Omega_{\mathrm{quo}_\ell(\lambda)}$. The claim follows. Similarly the second isomorphism follows from Theorem~\ref{thm10} and Lemma~\ref{lem11}.
\end{proof}
Recall that at the beginning of this section we claimed that it would be necessary to use the isomorphism from Theorem~\ref{thm9} to generalise the results of Section $4$. 

A more refined version of the statement is given below, in terms of the symmetric group and wreath product group. First, let us introduce some required notation. Let $\mathcal{C}_\ell[n]$ denote the set of $\ell$-cores $\gamma$ such that 
\[
|\gamma|\leq n\textnormal{ and } |\gamma|=n\,\mathrm{mod}\, \ell. 
\]
Given an $\ell$-multipartition $\lambda=(\lambda^1,\dots,\lambda^{\ell})$ we define $\lambda^\flat=(\lambda^\ell,\dots,\lambda^1)$, the reverse $\ell$-multipartition.

We must also discuss some subtleties concerning the difference in our paramterisation and the one given in \cite{bonnafe2021fixed}. Recall that in the case of $S_n$ the conjugacy invariant class functions on reflections are simply scalars, that is $c\in\C$. The conjugacy classes of reflections of $S_n\wr\mathbb{Z}/\ell\mathbb{Z}$ are slightly more complicated. Let $s_{i,j}$ be the transposition $(i,j)\in S_n$ and let $\gamma_i$ denote $\gamma\in \mathbb{Z}/\ell\mathbb{Z}$ in the $i^{th}$ component of $S_n\wr\mathbb{Z}/\ell\mathbb{Z}$. There are $\ell$ conjugacy classes of reflections in $S_n\wr\mathbb{Z}/\ell\mathbb{Z}$, there are $\ell-1$ of the form $\{\gamma_i\,|\, 0\leq i\leq n\}$ for $\gamma\in \mathbb{Z}/\ell\mathbb{Z}\setminus\{0\}$. A simple way of writing representatives of this conjugacy class is by considering their action as matrices, then $t,t^2,\dots, t^{\ell-1}$ is a full list of representatives where
\[
t=diag(\zeta,1,\dots,1)
\]
and $\zeta$ is a primitive $\ell^{th}$ root of unity. The set
\[
\{s_{i,j}\cdot \gamma_i^k\cdot \gamma_j^{-k}\,|\, 1\leq i,j\leq n,\textrm{ and } 1\leq k\leq \ell\}
\]
is the remaining conjugacy class and we will denote by $\sigma$ a representitive of this class. Bonnaf\'e and Maksimau use a different parametrisation than we do, so we now explain this. For the following let $c^H$ be our parametrisation and $c^B$ be the one in \cite{bonnafe2021fixed}, then 
\[
c^H(\sigma)=c^B(\sigma), \quad c^H(t^i)= \frac{\zeta^i-1}{-2} c^B(t^i).
\] 

In \cite{bonnafe2021fixed} they introduce a new family of parameters $k_0,\dots,k_{\ell-1}$ depending on $c$ defined by \[
k_j=\frac{1}{\ell}\sum^{\ell-1}_{i=1} \zeta^{-i(j-1)} c(\gamma_i). 
\]
It follows then that 
\[
c^B(t^i)=\sum_{j\in\mathbb{Z}/\ell\mathbb{Z}}\zeta^{i(j-1)}k_j \textnormal{ and } c^H(t^i)=\frac{\zeta^i-1}{-2} \sum_{j\in\mathbb{Z}/\ell\mathbb{Z}}\zeta^{i(j-1)}k_j.
\]
There is also a map $\beta_{\ell,v}:P^\ell_{v}(n)\rightarrow P(n,\ell)$ which is the restriction of the $\ell$-quotient map to the set of partitions with $\ell$-core $v$. For the wreath product this map becomes $\beta_{1,\emptyset}:P^1_{\emptyset}(n\ell)\rightarrow P(n\ell)$, the $1$-quotient, from the partitions of $n\ell$ with trivial $1$-core to the set of partitions of $n\ell$. Let $\lambda\in P^\ell(n)$ and $core_k(\lambda)=(\gamma^0,\dots,\gamma^{\ell-1})$ then $\beta^\flat_{k,\gamma}$ is defined as the $k\ell$-partition 
\[
(\mu^0,\dots,\mu^{k\ell-1})
\]
such that 
\[
\beta_{k,\gamma^i}(\lambda^i)=(\mu^{i+(k-1)\ell},\dots,\mu^{i+\ell},\mu^i)
\]
for $i\in\{0,1,\dots,\ell-1\}$.

The following theorem is \cite[Theorem 4.21]{bonnafe2021fixed}, applied to the particular case of $W=S_{n\ell}$. Recall that $X_c(S_{n\ell})$ admits a $\mathbb{C}^\times$-action. Therefore, the group $\mathbb{Z}/\ell\mathbb{Z}$ can be considered acting on $X_c(S_{n\ell})$ by identifying $\mathbb{Z}/\ell\mathbb{Z}$ with the $\ell^{th}$ roots of unity. 
\begin{thm}\label{thm15}
Assume that $X_c(S_{n\ell})$ is smooth. Then $X_c(S_{n\ell})^{\mathbb{Z}/\ell\mathbb{Z}}$ is smooth and:
\begin{enumerate}
    \item There is a bijection $\gamma\rightarrow \mathcal{I}(\gamma)$ between $\mathcal{C}_\ell[n\ell]$ and the irreducible components of $X_c(S_{n\ell})^{\mathbb{Z}/\ell\mathbb{Z}}$ such that $x_{\lambda}\in \mathcal{I}(\gamma)$ if and only if $core_\ell(\lambda)=\gamma$ for $\lambda\in\mathcal{P}(n\ell)$.

    \item Let $\gamma\in \mathcal{C}_\ell[n\ell]$ and $r=(n\ell-|\gamma|)/\ell$. There is an isomorphism of varieties
    \[
    i_{\gamma}:X_{\overline{c}}(S_{r}\wr \mathbb{Z}/\ell\mathbb{Z})\rightarrow \mathcal{I}(\gamma),
    \]
     where $\overline{c}$ is the parameter associated with the family $(a',k'_0,\dots k'_{\ell-1})$ such that 
    \[
    \begin{cases}
        a'=ka,\\
        k'_j=a\left( j-1-\frac{\ell-1}{2} +\ell(d_{1-j}-d_{-j}) \right)\textnormal{ for } 1\leq j\leq \ell \textnormal{ and } k_0'=k_\ell',
    \end{cases}
    \]
    which satisfies $x_{(\beta^\flat_{k,\gamma})^{-1}(\mu)}=i_\gamma(x^{\ell}_\mu)$ for all $\mu\in P(n,kl)$. Here, $d=(d_i)_{i\in \mathbb{Z}/\ell\mathbb{Z}}$ is defined 
    \[
    d=res_\ell(v)+r\delta_\ell,
    \]
    where $v=(\beta_{\emptyset})^{-1}(\gamma)$ is an $\ell$-core and $res_\ell(v)\in\mathbb{Z}^{\mathbb{Z}/\ell\mathbb{Z}}$ where $res_\ell(v)_i$ is the number of boxes of $v$ with $\ell$-residue $i$.
\end{enumerate}

\end{thm}
We apply the above theorem to the case when $\gamma=\emptyset$. There is an isomorphism of varieties 
    \[
    i_\emptyset:X_{\overline{c}}(S_{r}\wr \mathbb{Z}/\ell\mathbb{Z})\rightarrow \mathcal{I}(\emptyset)\subset X_c(S_{n\ell})^{\mathbb{Z}/\ell\mathbb{Z}},
    \]
    where $c=a$ and $\overline{c}$ is the parameter associated with the family $(a',k'_0,\dots k'_{\ell-1})$ such that 
    \[
    \begin{cases}
        a'=\ell a,\\
        k'_j=a\left( j-1-\frac{\ell-1}{2} \right) \textnormal{ for } 1\leq j\leq \ell \textnormal{ and } k_0'=k_\ell',
    \end{cases}
    \]
which satisfies $x_{(\mathrm{quo}_{\ell})^{-1}(\mu)}=i_\emptyset(x^{\ell}_\mu)$ for all $\mu\in P(n,\ell)$. 

It is immediate that the fixed points $x_\lambda$ with $\lambda$ having trivial $\ell$-core all lie in the irreducible component $\mathcal{I}(\emptyset)$. The isomorphism 
\[
X_{\overline{c}}(S_{n}\wr \mathbb{Z}/\ell\mathbb{Z})\cong \mathcal{I}(\emptyset).
\]
given by Theorem~\ref{thm15} also describes where the fixed points are mapped under the isomorphism. Since, in our case, the quotient map $\mathrm{quo}_\ell :\mathcal{P}^\ell_\emptyset(n\ell)\rightarrow \mathcal{P}(n,\ell)$ is a bijection by Theorem~\ref{thm1}, we have $i_\emptyset (x_{\mathrm{quo}_\ell(\lambda)})=x_\lambda$. This fact will be important and so we record it as a lemma.
\begin{lem}\label{lem12}
There is a $\mathbb{C}^\times$-equivariant isomorphism $i_\emptyset:X_{\overline{c}}(S_n\wr \mathbb{Z}/\ell\mathbb{Z})\rightarrow \mathcal{I}(\emptyset)$ 
such that under the labeling of the fixed points we have $i_\emptyset(x_{\mathrm{quo}_\ell(\lambda)})=x_\lambda$ for $\lambda\in \mathcal{P}(n\ell)$.
\end{lem}
In the following by $\Omega_{\mathrm{quo}_\ell(\lambda)}$ (respectively $\Omega^-_{\mathrm{quo}_\ell(\lambda)}$) we mean the attracting set (respectively the opposite attracting set) inside $X_{\overline{c}}(S_{n}\wr \mathbb{Z}/\ell\mathbb{Z})$ and $\Omega_\lambda$, $\Omega^-_{\lambda}$ are considered inside $X_c(S_{n\ell})
^{\mathbb{Z}/\ell\mathbb{Z}}$.
\begin{prop}\label{prop4}
The map $i_\emptyset$ restricts to $\mathbb{C}^\times$-equivariant isomorphisms of attracting sets 
\[
i_{\emptyset}:\Omega_{\mathrm{quo}_\ell(\lambda)}\cong \Omega_{\lambda}^{\mathbb{Z}/\ell\mathbb{Z}}=\Omega_\lambda\cap X_c(S_{n\ell})
^{\mathbb{Z}/\ell\mathbb{Z}},
\]
\[
i_{\emptyset}:\Omega^-_{\mathrm{quo}_\ell(\lambda)}\cong (\Omega^-_{\lambda})^{\mathbb{Z}/\ell\mathbb{Z}}=\Omega^-_{\lambda}\cap X_c(S_{n\ell})
^{\mathbb{Z}/\ell\mathbb{Z}}.
\]
\end{prop}
\begin{proof}
Since the map $i_\emptyset$ is $\mathbb{C}^\times$-equivariant it maps attracting sets to attracting sets hence
\[
i_{\emptyset}(\Omega_{\mathrm{quo}_\ell(\lambda)})\subset \Omega_{\lambda}^{\mathbb{Z}/\ell\mathbb{Z}} \quad and \quad i^{-1}_\emptyset(\Omega_{\lambda}^{\mathbb{Z}/\ell\mathbb{Z}})\subset \Omega_{\mathrm{quo}_\ell(\lambda)}
\]
and so $i_\emptyset:\Omega_{\mathrm{quo}_\ell(\lambda)}\rightarrow \Omega_{\lambda}^{\mathbb{Z}/\ell\mathbb{Z}}$ is a bijective morphism. Since a bijective morphism between smooth varieties is an isomorphism, the result follows. The second isomorphism follows by a similar argument.
\end{proof}
To summarise we have equalities of varieties 
\[
\mathrm{Spec}\,\mathrm{End}\underline{\Delta}(\mathrm{quo}_\ell(\lambda))= \Omega_{\mathrm{quo}_\ell(\lambda)}\quad and \quad \mathrm{Spec}\,\mathrm{End}\underline{\Delta}^-(\mathrm{quo}_\ell(\lambda))= \Omega^-_{\mathrm{quo}_\ell(\lambda)}
\]
and isomorphisms
\[
\Omega_{\mathrm{quo}_\ell(\lambda)}\cong \Omega_{\lambda}^{\mathbb{Z}/\ell\mathbb{Z}}\quad and\quad \Omega^-_{\mathrm{quo}_\ell(\lambda)}\cong (\Omega^-_{\lambda})^{\mathbb{Z}/\ell\mathbb{Z}}.
\]
This gives us a way to relate the endomorphism rings of the Verma modules for the symmetric and wreath product groups. Unfortunately, this is not enough to arrive at an explicit presentation of the endomorphism rings of the baby Verma modules. To do that we must understand a particular isomorphism explicitly. In \cite[Theorem 11.16]{EG} Etingof and Ginzburg construct an isomorphism between $Z_c(S_n\wr\mathbb{Z}/\ell\mathbb{Z})$ and a suitable Calogero-Moser space. We now focus on describing this map and showing it has the properties we desire.

Let $e_{\mathbb{Z}/\ell\mathbb{Z}}\in \mathrm{End}_\C (\mathbb{C}\mathbb{Z}/\ell\mathbb{Z}))$ be the map onto the trivial representation defined on elements by
\[
e_{\mathbb{Z}/\ell\mathbb{Z}}(g)=\left(\frac{1}{|\ell|}\sum_{\gamma\in \mathbb{Z}/\ell\mathbb{Z}}\gamma\right)\cdot g,
\]
and denote by $\bf{O}$ the conjugacy class inside $\mathfrak{sl}_n$ formed by all $n\times n$ matrices of the form $P-Id$ where $P$ is a semisimple matrix of rank $1$ with $tr(P)=n$. Given any class function $\mathbb{Z}/\ell\mathbb{Z}\setminus\{0\}\rightarrow \C$ sending $\gamma$ to $c'_\gamma$ then the element $c'=\sum_{\gamma}c'_\gamma \gamma$ is central with zero trace in the regular representation. Then for any pair $c=(k,c')$ may define the following set 
\[
M_{\mathbb{Z}/\ell\mathbb{Z},n,c}=\{\nabla_1,\nabla_2\in \mathrm{End}(\mathbb{C}^n\otimes\mathbb{C}\mathbb{Z}/\ell\mathbb{Z})\,|\, [\nabla_1,\nabla_2]=k\ell\cdot o\otimes e_{\mathbb{Z}/\ell\mathbb{Z}}+Id_{\mathbb{C}^n}\otimes c'\textnormal{ for some }o\in\bf{O}\}.
\]
The Calogero-Moser space can be defined as the quotient variety
\[
\mathcal{M}_{\mathbb{Z}/\ell\mathbb{Z},n,c}=M_{\mathbb{Z}/\ell\mathbb{Z},n,c}/PGL_{\mathbb{Z}/\ell\mathbb{Z},n,c}.
\]
We require an explicit understanding of the isomorphism $i_\emptyset$ by Bonnaf\'e and Maksimau. In their paper this map is given by the inclusion map between the Calogero-Moser spaces associated to $X_{c}(S_n\wr\mathbb{Z}/\ell\mathbb{Z})$ and $X_c(S_{n\ell})$.

By \cite[Theorem 1.7]{EG} there is an identification of $\mathrm{Spec}\,Z_c(S_n\wr\mathbb{Z}/\ell\mathbb{Z})$ with $\mathrm{Irr}H_c(S_n\wr \mathbb{Z}/\ell\mathbb{Z})$ given by the assignment
\begin{equation}\label{equation13}
p\rightarrow H_c(S_n\wr\mathbb{Z}/\ell\mathbb{Z})e\otimes_Z p,\quad \forall p\in\mathrm{MaxSpec}\,Z_c(S_n\wr\mathbb{Z}/\ell\mathbb{Z})
\end{equation}
where $p$ is viewed as a homomorphism $Z_c(S_n\wr\mathbb{Z}/\ell\mathbb{Z})\rightarrow \mathbb{C}$. Consider an irreducible $H_c(S_n\wr\mathbb{Z}/\ell\mathbb{Z})$-module $E$. Let $\Gamma_{n-1}$ denote the subgroup of $\Gamma=S_n\wr \mathbb{Z}/\ell\mathbb{Z}$ that stabilises the first basis vector $x_1$ in $\mathfrak{h}$. Let $E^{\Gamma_{n-1}}$ denote the subspace of $E$ fixed by $\Gamma_{n-1}$. Clearly $x_1$ and $y_1$ commute with the action of $\Gamma_{n-1}$. Therefore we can define operators $X,Y\in \mathrm{End}_{\mathbb{C}}(E^{\Gamma_{n-1}})$ via the action of $x_1$ and $y_1$ on $E$ respectively. The isomorphism $\phi:\mathrm{Irr}H_c(S_n\wr\mathbb{Z}/\ell\mathbb{Z})\rightarrow \mathcal{M}_{\mathbb{Z}/\ell\mathbb{Z},n,c}$ \cite[Theorem 11.16]{EG} is given by $\phi(E)=(X,Y)$.

Consider the open set $U$ in $\mathrm{Irr} H_c(S_n\wr\mathbb{Z}/\ell\mathbb{Z})$ where the action of the elements $x_i-\omega^{k}x_j$ are invertible; $U$ is an open set. Let $(\lambda,\mu)\in\mathbb{C}^{2n}$ with $\lambda_i^\ell\neq \lambda_j^\ell$ for all $i\neq j$. Let $\mathcal{O}_{(\lambda,\mu)}$ denote the orbit of $(\lambda,\mu)$ under the group $S_n\wr\mathbb{Z}/\ell\mathbb{Z}$. This is a free orbit, so $|\mathcal{O}_{(\lambda,\mu)}|=n!\ell^n$. Up to isomorphism, each representation $E$ in $U$ is of the form $E_{(\lambda,\mu)}=\mathbb{C}[\mathcal{O}_{\lambda,\mu}]$. A basis of $E_{(\lambda,\mu)}$ is given by the characteristic equations
\[
\chi_s(a,b)=\begin{cases}
1 \textnormal{ if } s\cdot (a,b)=(\lambda,\mu)\\
0 \textnormal{ else}
\end{cases}
\]
for $s\in S_n\wr \mathbb{Z}/\ell\mathbb{Z}$. The subspace $E^{\Gamma_{n-1}}$ is then $\ell n$-dimensional with basis $\chi_{s_{1,i}\gamma_1^r}$ for $1\leq i\leq a$ and $0\leq r\leq \ell -1$. The action of $H_c(S_n\wr\mathbb{Z}/\ell\mathbb{Z})$ on $E_{(\lambda,\mu)}$ is given by
\[
x_i \cdot F(a,b)=a_iF(a,b),\textnormal{ } y_i\cdot F(a,b)=b_i \cdot F(a,b) + c_0\sum_{j\neq i}\sum_{k=0}^{\ell-1}\frac{s_{i,j}\gamma_i^k\gamma_j^{-k}F(a,b)}{\omega^{-k}a_j-a_i}+\sum_{k=1}^{\ell-1}\frac{c_k\gamma_i^k F(a,b)}{a_i\omega^k -a_i}
\]
and
\[
(w\cdot F)(a,b)=F(w^{-1}\cdot a,w^{-1}\cdot b),
\]
for $w\in S_n\wr \mathbb{Z}/\ell\mathbb{Z}$ and $\omega$ a primitive $\ell^{th}$ root of unity.
We must check that these equations satisfy the defining relations of the rational Cherednik algebra at $t=0$, recall these where given \eqref{equation30}. The check is just a straightforward computation which we include for completeness
\[
[y_i,x_t]\cdot F(a,b)=y_i\cdot x_t\cdot F(a,b)-x_t\cdot y_i\cdot F(a,b)
\]
\[
=x_t\left(b_i \cdot F(a,b) + c_0\sum_{j\neq i}\sum_{k=0}^{\ell-1}\frac{s_{i,j}\gamma_i^k\gamma_j^{-k}F(a,b)}{\omega^{-k}a_j-a_i}+\sum_{k=1}^{\ell-1}\frac{c_k\gamma_i^k F(a,b)}{a_i\omega^k -a_i}\right)-y_i\cdot a_tF(a,b)
\]
\[
=\left(a_tb_i \cdot F(a,b) + c_0\sum_{j\neq i}\sum_{k=0}^{\ell-1}\frac{s_{i,j}\gamma_i^k\gamma_j^{-k}a_t F(a,b)}{\omega^{-k}a_j-a_i}+\sum_{k=1}^{\ell-1}\frac{c_k\gamma_i^k a_t F(a,b)}{a_i\omega^k -a_i}\right)
\]
\[
-\left( a_tb_i \cdot F(a,b) + a_t c_0\sum_{j\neq i}\sum_{k=0}^{\ell-1}\frac{s_{i,j}\gamma_i^k\gamma_j^{-k}F(a,b)}{\omega^{-k}a_j-a_i}+a_t\sum_{k=1}^{\ell-1}\frac{c_k\gamma_i^k F(a,b)}{a_i\omega^k -a_i}\right)
\]
\[
=c_0\sum_{j\neq i}\sum_{k=0}^{\ell-1}\frac{s_{i,j}\gamma_i^k\gamma_j^{-k}a_tF(a,b)}{\omega^{-k}a_j-a_i}-a_tc_0\sum_{j\neq i}\sum_{k=0}^{\ell-1}\frac{s_{i,j}\gamma_i^k\gamma_j^{-k}F(a,b)}{\omega^{-k}a_j-a_i}
\]
\[
=c_0\sum_{k=0}^{\ell-1}\frac{s_{i,t}\gamma_i^k\gamma_t^{-k}a_tF(a,b)}{\omega^{-k}a_t-a_i}-a_tc_0\sum_{k=0}^{\ell-1}\frac{s_{i,t}\gamma_i^k\gamma_t^{-k}F(a,b)}{\omega^{-k}a_t-a_i}
\]
\[
=(\omega^k a_i-a_t)c_0\sum_{k=0}^{\ell-1}\frac{s_{i,t}\gamma_i^k\gamma_t^{-k}F(a,b)}{\omega^{-k}a_t-a_i}=-c_0\sum_{k=0}^{\ell-1}\omega^ks_{i,t}\gamma_i^k\gamma_t^{-k}F(a,b)
\]
Now we check the relation 
\[
[y_i,x_i]=c_0\sum_{i\neq j}\sum^{\ell-1}_{k=0}s_{ij}s_i^ks_j^{-k}+\sum_{k=1}^{\ell-1}c_ks_j^{-k}.
\]
So
\[
[y_i,x_i]\cdot F(a,b) =y_i\cdot x_i\cdot F(a,b)-x_i\cdot y_i\cdot F(a,b)
\]
\[
=x_i\cdot \left(b_i \cdot F(a,b) + c_0\sum_{j\neq i}\sum_{k=0}^{\ell-1}\frac{s_{i,j}\gamma_i^k\gamma_j^{-k}F(a,b)}{\omega^{-k} a_j-a_i}+\sum_{k=1}^{\ell-1}\frac{c_k\gamma_i^k F(a,b)}{a_i\omega^k -a_i}\right )-y_i\cdot a_iF(a,b)
\]
\[
=a_ib_iF(a,b)+c_0\sum_{j\neq i}\sum_{k=0}^{\ell-1}\frac{s_{i,j}\gamma_i^k\gamma_j^{-k}\omega^{-k}a_jF(a,b)}{\omega^{-k}a_j-a_i}+\sum_{k=1}^{\ell-1}\frac{c_k\gamma_i^k \omega^k a_iF(a,b)}{a_i\omega^k- a_i})
\]
\[
-a_i\left(b_i \cdot F(a,b) + c_0\sum_{j\neq i}\sum_{k=0}^{\ell-1}\frac{s_{i,j}\gamma_i^k\gamma_j^{-k}F(a,b)}{\omega^{-k} a_j-a_i}+\sum_{k=1}^{\ell-1}\frac{c_k\gamma_i^k F(a,b)}{a_i\omega^k -a_i}\right)
\]
\[
=c_0\sum_{j\neq i}\sum_{k=0}^{\ell-1}\frac{(\omega^{-k}a_j-a_i)s_{i,j}\gamma_i^k\gamma_j^{-k}F(a,b)}{\omega^{-k} a_j-a_i}+\sum_{k=1}^{\ell-1}\frac{(\omega^k a_i-a_i)c_k\gamma_i^k F(a,b)}{a_i\omega^k -a_i}
\]
\[
=c_0\sum_{j\neq i}\sum_{k=0}^{\ell-1}s_{i,j}\gamma_i^k\gamma_j^{-k}F(a,b)+\sum_{k=1}^{\ell-1}c_k\gamma_i^k F(a,b)
\]
as required.

Similarly we can consider the open set $V$ in $\mathrm{Irr} H_c(S_n\wr\mathbb{Z}/\ell\mathbb{Z})$ where the action of the elements $y_i-\omega^{k}y_j$ are invertible. Each representation $E$ in $V$ has the same form as before, $E_{(\lambda,\mu)}=\mathbb{C}[\mathcal{O}_{\lambda,\mu}]$. The representations $E_{(\lambda,\mu)}$ have the same basis, but with different action of $H_c(S_n\wr\mathbb{Z}/\ell\mathbb{Z})$ given by
\[
x_i \cdot F(a,b)=a_iF(a,b)-c_0\sum_{j\neq i}\sum_{k=0}^{\ell-1}\frac{s_{i,j}\gamma_i^k\gamma_j^{-k}F(a,b)}{\omega^{-k} b_j-b_i}-\sum_{k=1}^{\ell-1}\frac{c_k\gamma_i^k F(a,b)}{b_i\omega^k -b_i},\textnormal{ } y_i\cdot F(a,b)=b_i 
\]
and
\[
(w\cdot F)(a,b)=F(w^{-1}\cdot a,w^{-1}\cdot b),
\]
for $w\in S_n\wr \mathbb{Z}/\ell\mathbb{Z}$ and $\omega$ a primitive $\ell^{th}$ root of unity. We must recheck the Cherednik algebra relations
\[
[y_t,x_i]\cdot F(a,b)=y_t\cdot x_i\cdot F(a,b)-x_i\cdot y_t\cdot F(a,b)
\]
\[
=x_i\cdot b_t F(a,b)-y_t\cdot\left( a_iF(a,b)-c_0\sum_{j\neq i}\sum_{k=0}^{\ell-1}\frac{s_{i,j}\gamma_i^k\gamma_j^{-k}F(a,b)}{\omega^{-k} b_j-b_i}-\sum_{k=1}^{\ell-1}\frac{c_k\gamma_i^k F(a,b)}{b_i\omega^k -b_i}\right)
\]
\[
=\left(a_ib_tF(a,b)-b_tc_0\sum_{j\neq i}\sum_{k=0}^{\ell-1}\frac{s_{i,j}\gamma_i^k\gamma_j^{-k}F(a,b)}{\omega^{-k} b_j-b_i}-b_t\sum_{k=1}^{\ell-1}\frac{c_k\gamma_i^k F(a,b)}{b_i\omega^k -b_i}\right)
\]
\[
-\left(a_ib_tF(a,b)-c_0\sum_{j\neq i}\sum_{k=0}^{\ell-1}\frac{s_{i,j}\gamma_i^k\gamma_j^{-k}b_tF(a,b)}{\omega^{-k} b_j-b_i}-\sum_{k=1}^{\ell-1}\frac{c_k\gamma_i^k b_tF(a,b)}{b_i\omega^k -b_i}\right)
\]
\[
=c_0\sum_{j\neq i}\sum_{k=0}^{\ell-1}\frac{s_{i,j}\gamma_i^k\gamma_j^{-k}b_tF(a,b)}{\omega^{-k} b_j-b_i}-b_tc_0\sum_{j\neq i}\sum_{k=0}^{\ell-1}\frac{s_{i,j}\gamma_i^k\gamma_j^{-k}F(a,b)}{\omega^{-k} b_j-b_i}
\]
\[
=c_0\sum_{k=0}^{\ell-1}\frac{s_{i,t}\gamma_i^k\gamma_t^{-k}b_tF(a,b)}{\omega^{-k} b_t-b_i}-b_tc_0\sum_{k=0}^{\ell-1}\frac{s_{i,t}\gamma_i^k\gamma_t^{-k}F(a,b)}{\omega^{-k} b_t-b_i}
\]
\[
=(\omega^kb_i-b_t)c_0\sum_{k=0}^{\ell-1}\frac{s_{i,t}\gamma_i^k\gamma_t^{-k}b_tF(a,b)}{\omega^{-k} b_t-b_i}=-c_0\sum_{k=0}^{\ell-1}\omega^k s_{i,t}\gamma_i^k\gamma_t^{-k}F(a,b).
\]
Let us now check the second relation
\[
[y_i,x_i]\cdot F(a,b) =y_i\cdot x_i\cdot F(a,b)-x_i\cdot y_i\cdot F(a,b)
\]
\[
x_i\cdot b_iF(a,b)-y_i\cdot\left(a_iF(a,b)-c_0\sum_{j\neq i}\sum_{k=0}^{\ell-1}\frac{s_{i,j}\gamma_i^k\gamma_j^{-k}F(a,b)}{\omega^{-k} b_j-b_i}-\sum_{k=1}^{\ell-1}\frac{c_k\gamma_i^k F(a,b)}{b_i\omega^k -b_i}\right)
\]
\[
=\left(a_ib_iF(a,b)-b_ic_0\sum_{j\neq i}\sum_{k=0}^{\ell-1}\frac{s_{i,j}\gamma_i^k\gamma_j^{-k}F(a,b)}{\omega^{-k} b_j-b_i}-b_i\sum_{k=1}^{\ell-1}\frac{c_k\gamma_i^k F(a,b)}{b_i\omega^k -b_i}\right)
\]
\[
-\left(a_ib_iF(a,b)-c_0\sum_{j\neq i}\sum_{k=0}^{\ell-1}\frac{s_{i,j}\gamma_i^k\gamma_j^{-k}b_iF(a,b)}{\omega^{-k} b_j-b_i}-\sum_{k=1}^{\ell-1}\frac{c_k\gamma_i^k b_iF(a,b)}{b_i\omega^k -b_i}\right)
\]
\[
(\omega^{-k}b_j-b_i)c_0\sum_{j\neq i}\sum_{k=0}^{\ell-1}\frac{s_{i,j}\gamma_i^k\gamma_j^{-k}F(a,b)}{\omega^{-k} b_j-b_i}+  (b_i\omega^k-b_i)  \sum_{k=1}^{\ell-1}\frac{c_k\gamma_i^k F(a,b)}{b_i\omega^k -b_i}
\]
\[
=c_0\sum_{j\neq i}\sum_{k=0}^{\ell-1}s_{i,j}\gamma_i^k\gamma_j^{-k}F(a,b)+\sum_{k=1}^{\ell-1}c_k\gamma_i^k F(a,b).
\]

\begin{thm}\label{thm16}
Let $\phi:\mathrm{Spec}\,Z_c(S_n\wr\mathbb{Z}/\ell\mathbb{Z})\rightarrow \mathcal{M}_{\mathbb{Z}/\ell\mathbb{Z},n,c}$ be the map defined above. Then we have the following equalities 
\[
\phi^*(tr(X)^k)=\begin{cases}
\ell(x_1^k+...+x_n^k) \textnormal{ if } \ell\vert k\\
0 \textnormal{ else } 
\end{cases}
\]
\[
\phi^*(tr(Y)^k)=\begin{cases}
\ell(y_1^k+...+y_n^k) \textnormal{ if } \ell\vert k\\
0 \textnormal{ else } 
\end{cases}
\]
\end{thm}
\begin{proof}
We begin by noting that $Z_c(S_n\wr\mathbb{Z}/\ell\mathbb{Z})$ is reduced since \cite[Proposition 7.2]{PoissonOrders} states that $eH_{c}(S_n\wr\mathbb{Z}/\ell\mathbb{Z})e\cong Z_c(S_n\wr\mathbb{Z}/\ell\mathbb{Z})$ is a domain. Hence for $f,g\in Z_c(S_n\wr\mathbb{Z}/\ell\mathbb{Z})$, $f=g$ if and only if $f(p)=g(p)$ for all $p\in \mathrm{maxSpec}\, Z_c(S_n\wr\mathbb{Z}/\ell\mathbb{Z})$. Furthermore $\mathrm{Spec}\, Z_c(S_n\wr\mathbb{Z}/\ell\mathbb{Z})$ is irreducible, therefore $f=g$ if and only if $f(p)=g(p)$ for all $p\in U\subset \mathrm{Spec}\,Z_c(S_n\wr\mathbb{Z}/\ell\mathbb{Z})$, where $U$ is the open set defined above. We also use the identification of $\mathrm{Spec}\,Z_c(S_n\wr\mathbb{Z}/\ell\mathbb{Z}) $ with $\mathrm{Irr}H_c(S_n\wr\mathbb{Z}/\ell\mathbb{Z})$ as in equation~\eqref{equation13} above. Fix an irreducible module $E_{(\lambda,\mu)}\in U$. We shall first calculate $\phi^*(tr(X))^k(E_{(\lambda,\mu)})$. As described in Section 11 of \cite{EG}, $E_{(\lambda,\mu)}^{\Gamma_{n-1}}$ is isomorphic as a $\mathbb{C}(\mathbb{Z}/\ell\mathbb{Z})$-module to $n$ copies of $\mathbb{C}(\mathbb{Z}/\ell\mathbb{Z})$. Therefore, $E_{(\lambda,\mu)}^{\Gamma_{n-1}}$ can be viewed a sum of vector spaces $V_0\oplus\dots\oplus V_{\ell-1}$ where $V_i$ is $n$ copies of the irreducible representation of $ \mathbb{Z}/\ell\mathbb{Z}$ where the generator $s$ acts by $\omega^i$. If we denote the action of $x_1$ on $V_i$ by $X_i$ then note that $X_i:V_i\rightarrow V_{i+1}$ as 
\[
s\cdot X_i(v)=s\cdot x_1 v=\omega x_1 s \cdot v=\omega^{i+1}x_1v=\omega^{i+1}X_i(v)\textnormal{ if }v\in V_i.
\]
Recall that $\phi(E)=(X,Y)$, where $X=x_1$ acting on $E^{\Gamma_{n-1}}$. Then we see that as a matrix,
\begin{equation}\label{equation14}
X=\begin{bmatrix}
    0 & 0 & \dots  & X_{\ell-1} \\
    X_0 & 0&  \dots  & 0 \\
    0 & \ddots & 0 & 0 \\
    0 & \dots & X_{\ell-2}  & 0
\end{bmatrix}.
\end{equation}
Hence $\phi^*(tr(X)^k)(E_{(\lambda,\mu)})=tr(X^k)$ and if $\ell\not| \,k$ then $tr(X^k)=0$ as $X^k$ has every entry on the main diagonal equal to $0$. However if $\ell\vert k$ then write $k=m\ell$ and $tr(X^k)=\ell tr((X_0\dots X_{\ell-1})^m)$. For each $X_i$, we have $X_i=\mathrm{diag}(\lambda_1,\dots ,\lambda_n)$, hence 
\[
tr((X_0\dots X_{\ell-1})^m)=\lambda_1^{m\ell}+\dots +\lambda_{n}^{m\ell}=\lambda_1^{k}+\dots +\lambda_{n}^{k}.
\]
Substituting back in we find 
\[
tr(X^k)=\ell(\lambda_1^{k}+\dots +\lambda_{n}^{k}).
\]
Now we must check that $\ell(x_1^k+\dots +x_n^k)(u)=\ell(\lambda_1^k+\dots +\lambda_n^k)(u)$ for all $u\in E$. For each $F\in E_{(\lambda,\mu)}$, $(x_1^k+\dots +x_n^k\cdot F)(a,b)=(\lambda_1^k+\dots + \lambda_n^k)F(a,b)$ hence $x_1^k+\dots+ x_n^k$ acts by scalar multiplication on $E_{(\lambda,\mu)}$ by $\lambda_1^k+\dots + \lambda_n^k$. 

To prove the statement of the proof for $Y$ we simply note that the same argument holds, provided we use the open set $V\subset \mathrm{Spec}\,Z_c(S_n\wr\mathbb{Z}/\ell\mathbb{Z})$ instead of $U$. Then the action of $Y$ on $E^{\Gamma_{n-1}}$ is given by the matrix
\begin{equation}\label{equationy}
\begin{bmatrix}
    0 & Y_1 & \dots  & 0 \\
    0 & 0 &  \ddots  & 0 \\
    0 & 0 & \dots & Y_{\ell-1} \\
    Y_0 & 0 & \dots  & 0
\end{bmatrix}.
\end{equation}
From an identical calculation of traces we see that 
\[
tr(Y^k)=\ell(\mu_1^{k}+\dots +\mu_{n}^{k})
\]
if $\ell | k$ and $tr(Y)^k=0$ otherwise.
\end{proof}
We note also.
\begin{lem}\label{lem13}
There is an isomorphism  
\[
\alpha : \mathbb{C}^n/(S_n\wr\mathbb{Z}/\ell\mathbb{Z}) \xrightarrow{\sim} (\mathbb{C}^{n\ell}/S_{n\ell})^{\mathbb{Z}/\ell\mathbb{Z}}
\]
given by 
\[
\alpha(a_1,a_2,\dots,a_n)=(a_1,\omega a_1,\omega^2 a_1,\dots, \omega^{\ell-1}a_n)
\]
\end{lem}

Recall the map $i_\emptyset$ introduced in Lemma~\ref{lem12}. We must break this into a composition of three maps. Consider the inclusion map on quiver varieties
\[
\overline{i}_\emptyset :\mathcal{M}_{\mathbb{Z}/\ell\mathbb{Z},n,\overline{c}}\rightarrow \mathcal{M}_{\mathbb{Z}/\mathbb{Z},n\ell,c}
\]
given by sending $(X_0,\dots X_{\ell-1})$ to the matrix $X$ of equation~\eqref{equation14} and $(Y_0,\dots,Y_{\ell-1})$ to \eqref{equationy}. Then the map $i_\emptyset$ of Lemma~\ref{lem12} is given by  $\phi^{-1}_{S_{n\ell}}\circ \overline{i}_\emptyset\circ \phi_{S_{n}\wr\mathbb{Z}/\ell\mathbb{Z}}$. With all the appropriate maps introduced we can perform the following diagram chase.
\begin{thm}\label{thm17}
There is an isomorphism $X_{\overline{c}}(S_n\wr\mathbb{Z}/\ell\mathbb{Z})\rightarrow W$ to a connected component of $X_{c}(S_{n\ell})^{\mathbb{Z}/\ell\mathbb{Z}}$ such that the following diagrams commute
\[
\begin{tikzpicture}[scale=2]
\node (A) at (0,1.5) {$X_{\overline{c}}(S_n\wr\mathbb{Z}/\ell\mathbb{Z})$};
\node (B) at (3,1.5) {$X_{c}(S_{n\ell})^{\mathbb{Z}/\ell\mathbb{Z}}$};
\node (C) at (0,0) {$\mathbb{C}^n/(S_n\wr\mathbb{Z}/\ell\mathbb{Z})$};
\node (D) at (3,0) {$(\mathbb{C}^{n\ell}/S_{n\ell})^{\mathbb{Z}/\ell\mathbb{Z}}$};

\path[->,font=\scriptsize,>=angle 90]
(A) edge node[above]{$\phi^{-1}_{S_{n\ell}}\circ \overline{i}_\emptyset\circ \phi_{S_{n}\wr\mathbb{Z}/\ell\mathbb{Z}}$} (B)
(A) edge node[left]{$\pi_{n,\ell}$} (C)
(B) edge node[right]{$\pi_{n\ell}$} (D)
(C) edge node[below]{$\alpha$} (D);
\end{tikzpicture}
\]
and 
\[
\begin{tikzpicture}[scale=2]
\node (A) at (0,1.5) {$X_{\overline{c}}(S_n\wr\mathbb{Z}/\ell\mathbb{Z})$};
\node (B) at (3,1.5) {$X_{c}(S_{n\ell})^{\mathbb{Z}/\ell\mathbb{Z}}$};
\node (C) at (0,0) {$(\mathbb{C}^n)^*/(S_n\wr\mathbb{Z}/\ell\mathbb{Z})$};
\node (D) at (3,0) {$((\mathbb{C}^{n\ell})^*/S_{n\ell})^{\mathbb{Z}/\ell\mathbb{Z}}$};

\path[->,font=\scriptsize,>=angle 90]
(A) edge node[above]{$\phi^{-1}_{S_{n\ell}}\circ \overline{i}_\emptyset\circ \phi_{S_{n}\wr\mathbb{Z}/\ell\mathbb{Z}}$} (B)
(A) edge node[left]{$\pi_{n,\ell}$} (C)
(B) edge node[right]{$\pi_{n\ell}$} (D)
(C) edge node[below]{$\alpha$} (D);
\end{tikzpicture}.
\]
\end{thm}
\begin{proof}
The first step is to unpack the diagram by introducing the Calogero-Moser spaces
\[
\begin{tikzpicture}[scale=2]
\node (A) at (0,1.5) {$X_{\overline{c}}(S_n\wr\mathbb{Z}/\ell\mathbb{Z})$};
\node (B) at (5.25,1.5) {$X_{c}(S_{n\ell})$};
\node (C) at (0,0) {$\mathbb{C}^n/(S_n\wr\mathbb{Z}/\ell\mathbb{Z})$};
\node (D) at (5.25,0) {$(\mathbb{C}^{n\ell}/S_{n\ell})^{\mathbb{Z}/\ell\mathbb{Z}}$};
\node (E) at (1.75,1.5) {$\mathcal{M}_{\mathbb{Z}/\ell\mathbb{Z},n,\overline{c}}$};
\node (F) at (3.5,1.5) {$\mathcal{M}_{\mathbb{Z}/\mathbb{Z},n\ell,c}$};
\path[->,font=\scriptsize,>=angle 90]
(A) edge node[above]{$ \phi_{S_{n}\wr\mathbb{Z}/\ell\mathbb{Z}}$} (E)
(E) edge node[above]{$\overline{i}_\emptyset$} (F)
(F) edge node[above]{$\phi^{-1}_{S_{n\ell}}$} (B)
(A) edge node[left]{$\pi_{n,\ell}$} (C)
(B) edge node[right]{$\pi_{n\ell}$} (D)
(C) edge node[below]{$\alpha$} (D);
\end{tikzpicture}.
\]
It is easier understand the duals of the maps in the diagram above and since a diagram commutes if and only if its dual does we shall prove this instead. We must therefore prove the commutativity of the following diagram
\[
\begin{tikzpicture}[scale=2]
\node (A) at (0,1.5) {$Z_{\overline{c}}(S_n\wr\mathbb{Z}/\ell\mathbb{Z})$};
\node (B) at (5.25,1.5) {$Z_{c}(S_{n\ell})$};
\node (C) at (0,0) {$\mathbb{C}[\mathbb{C}^n/(S_n\wr\mathbb{Z}/\ell\mathbb{Z})]$};
\node (D) at (5.25,0) {$\mathbb{C}[(\mathbb{C}^{n\ell}/S_{n\ell})^{\mathbb{Z}/\ell\mathbb{Z}}]$};
\node (E) at (1.75,1.5) {$\mathbb{C}[\mathcal{M}_{\mathbb{Z}/\ell\mathbb{Z},n,\overline{c}}]$};
\node (F) at (3.5,1.5) {$\mathbb{C}[\mathcal{M}_{\mathbb{Z}/\mathbb{Z},n\ell,c}]$};
\path[->,font=\scriptsize,>=angle 90]
(E) edge node[above]{$ \phi^*_{S_{n}\wr\mathbb{Z}/\ell\mathbb{Z}}$} (A)
(F) edge node[above]{$\overline{i}^*_\emptyset$} (E)
(B) edge node[above]{$(\phi^*)^{-1}_{S_{n\ell}}$} (F)
(C) edge node[left]{$i_{n,\ell}$} (A)
(D) edge node[right]{$i_{n\ell}$} (B)
(D) edge node[below]{$\alpha^*$} (C);
\end{tikzpicture}.
\]
First note that $\mathbb{C}[(\mathbb{C}^{n\ell}/S_{n\ell})^{\mathbb{Z}/\ell\mathbb{Z}}]$ are the symmetric polynomials fixed under the action of the $\ell^{th}$ roots of unity. Hence it is generated by elements of the form
\[
x_1^{k\ell}+\dots +x_{n\ell}^{k\ell} \textnormal{ where } k\in\mathbb{Z}_{\geq 0}.
\]
Now $\alpha^*(f)(p)=f(\alpha(p))$, where $p\in \mathbb{C}^n/(S_n\wr\mathbb{Z}/\ell\mathbb{Z})$. Consider an arbitrary generator $x_1^{k\ell}+\dots + x_{n\ell}^{k\ell}$, then
\[
\alpha^*(x_1^{k\ell}+\dots + x_{n\ell}^{k\ell})(a_1,a_2,\dots, a_n)=(x_1^{k\ell}+\dots + x_{n\ell}^{k\ell})(a_1,\omega a_1,\dots, \omega^{\ell-1}a_n)
\]
\[
=a_1^{k\ell}+(\omega a_1)^{k\ell}+(\omega^2 a_1)^{k\ell}+\dots +(\omega^{\ell-1}a_n)^{k\ell}
\]
\[
=(1+\omega^{k\ell}+(\omega^2)^{k\ell}\dots+ (\omega^{\ell-1})^{k\ell})(a_1^{k\ell}+a_2^{k\ell}+\dots+ x_n^{k\ell})=\ell(a_1^{k\ell}+\dots + a_n^{k\ell}).
\]
Hence 
\[
\alpha^*(x_1^{k\ell}+\dots+ x_{n\ell}^{k\ell})=\ell(x_1^{k\ell}+\dots +x_n^{k\ell}).
\]
Recall the map $\pi$ is the dual of the inclusion map so \[
i_{n,\ell}(\ell(x_1^{k\ell}+\dots+ x_n^{k\ell}))=\ell(x_1^{k\ell}+\dots+ x_n^{k\ell}).
\]
Now we must chase the diagram the other way. The first map $i_{n\ell}$ is also the inclusion map hence
\[
i_{n\ell}(x_1^{k\ell}+\dots +x_{n\ell}^{k\ell})=x_1^{k\ell}+\dots +x_{n\ell}^{k\ell}.
\]
By Theorem~\ref{thm16}, we have that $\phi^*_{S_{n\ell}}(tr(X)^k)=x_1^k+\dots +x_{n\ell}^k$, therefore $(\phi^*_{S_{n\ell}})^{-1}(x_1^{k\ell}+\dots +x_{n\ell}^{k\ell})=tr(X)^{k\ell}$. Then we have by definition 
\[
(\overline{i}_\emptyset^{*})^{-1}(tr(X)^{k\ell})=tr(\overline{i}_\emptyset(X)^{k\ell}).
\]
Therefore we complete the proof by showing $\phi^*_{S_n\wr \mathbb{Z}/\ell\mathbb{Z}}(tr(\overline{i}_\emptyset(X)^{k\ell})=\ell(x_1^{k\ell}+\dots+ x_n^{k\ell})$. This is precisely the statement of Theorem~\ref{thm16}. The second diagram commutes by an identical argument and the second equality of Theorem~\ref{thm16}.
\end{proof}
Recall by Corollary~\ref{cor2} that $A(\lambda)^+\cong\mathbb{C}[\pi^{-1}(0)]$ and, we denoted the functions on the fixed point locus as
\[
A(\lambda)^+_{\mathbb{Z}/\ell\mathbb{Z}}:=\mathbb{C}[\pi^{-1}_{n\ell}(0)]/\langle f-s\cdot f\,|\, s\in \mathbb{Z}/\ell\mathbb{Z},  f\in \mathbb{C}[\pi^{-1}_{n\ell}(0)] \rangle.
\]
\begin{thm}\label{thm18}
There are isomorphisms of algebras
\[
A_{\overline{c}}(\mathrm{quo}_\ell(\lambda))^+\cong A_c(\lambda)^+_{\mathbb{Z}/\ell\mathbb{Z}}\quad and\quad A_{\overline{c}}(\mathrm{quo}_\ell(\lambda)^\sharp)^-\cong A_c(\lambda^T)^-_{\mathbb{Z}/\ell\mathbb{Z}}.
\]
\end{thm}
\begin{proof}
By definition, $A_{\overline{c}}(\mathrm{quo}_\ell(\lambda))^+:=\mathrm{End}\Delta_{\overline{c}}(\mathrm{quo}_\ell(\lambda))=\mathbb{C}[\mathrm{Spec}\,\mathrm{End}\Delta_{\overline{c}}(\mathrm{quo}_\ell(\lambda))]$.
By Proposition~\ref{prop4} there is an isomorphism 
\[
i_\emptyset:\Omega_{\mathrm{quo}_\ell(\lambda)}\cong \Omega_{\lambda}\cap X_c(S_{n\ell})^{\mathbb{Z}/\ell\mathbb{Z}}.
\]
Therefore Theorem~\ref{thm17} implies that there is a commutative diagram
\[
\begin{tikzpicture}[scale=2]
\node (A) at (0,1.5) {$\Omega_{\mathrm{quo}_\ell(\lambda)}$};
\node (B) at (3,1.5) {$\Omega_{\lambda}^{\mathbb{Z}/\ell\mathbb{Z}}$};
\node (C) at (0,0) {$\mathbb{C}^n/(S_n\wr\mathbb{Z}/\ell\mathbb{Z})$};
\node (D) at (3,0) {$(\mathbb{C}^{n\ell}/S_{n\ell})^{\mathbb{Z}/\ell\mathbb{Z}}$};
\path[->,font=\scriptsize,>=angle 90]
(A) edge node[above]{$\phi^{-1}_{S_{n\ell}}\circ \overline{i}_\emptyset\circ \phi_{S_{n}\wr\mathbb{Z}/\ell\mathbb{Z}}$} (B)
(A) edge node[left]{$\pi_{n,\ell}$} (C)
(B) edge node[right]{$\pi_{n\ell}$} (D)
(C) edge node[below]{$\alpha$} (D);
\end{tikzpicture}.
\]
By Theorem~\ref{thm14}, $\mathrm{Spec}\,\mathrm{ End }\underline{\Delta}_{\overline{c}}(\mathrm{quo}_\ell(\lambda))= \Omega_{\mathrm{quo}_\ell(\lambda)}$ and $\mathrm{Spec}\,\mathrm{ End }\underline{\Delta}_c(\lambda)= \Omega_{\lambda}.$ Hence the diagram becomes
\[
\begin{tikzpicture}[scale=2]
\node (A) at (0,1.5) {$\mathrm{Spec}\,\mathrm{End}\underline{\Delta}_{\overline{c}}(\mathrm{quo}_\ell(\lambda))=\Omega_{\mathrm{quo}_\ell(\lambda)}$};
\node (B) at (3,1.5) {$\Omega_\lambda^{\mathbb{Z}/\ell\mathbb{Z}}=(\mathrm{Spec}\,\mathrm{End}\underline{\Delta}_c(\lambda))^{\mathbb{Z}/\ell\mathbb{Z}}$};
\node (C) at (0,0) {$\mathbb{C}^n/(S_n\wr\mathbb{Z}/\ell\mathbb{Z})$};
\node (D) at (3,0) {$(\mathbb{C}^{n\ell}/S_{n\ell})^{\mathbb{Z}/\ell\mathbb{Z}}$};
\path[->,font=\scriptsize,>=angle 90]
(A) edge node[above]{$i_\emptyset$} (B)
(A) edge node[left]{$\pi_{n,\ell}$} (C)
(B) edge node[right]{$\pi_{n\ell}$} (D)
(C) edge node[below]{$\alpha$} (D);
\end{tikzpicture}.
\]
Since $\alpha$ and $i_\emptyset$ are both isomorphisms we have $\pi_{n,\ell}^{-1}(0)\cong (\pi_{n\ell}^{-1}(0))^{\mathbb{Z}/\ell\mathbb{Z}}$. Therefore there is an algebra isomorphism $\mathbb{C}[\pi^{-1}_{n,\ell}(0)]\cong\mathbb{C}[\pi^{-1}_{n\ell}(0)^{\mathbb{Z}/\ell\mathbb{Z}}]$. Finally, Corollary~\ref{cor2} implies that
\begin{align*}
    A_{\overline{c}}(\mathrm{quo}_\ell(\lambda))^+\cong &\mathbb{C}[\pi^{-1}_{n,\ell}(0)]\\
    \cong&\mathbb{C}[\pi^{-1}_{n\ell}(0)^{\mathbb{Z}/\ell\mathbb{Z}}]\\
    \cong&
    \mathbb{C}[\pi^{-1}_{n\ell}(0)]/\langle f-s\cdot f\,|\, s\in \mathbb{Z}/\ell\mathbb{Z},  f\in \mathbb{C}[\pi^{-1}_{n\ell}(0)] \rangle
\end{align*}
hence $A_{\overline{c}}(\mathrm{quo}_\ell(\lambda))^+\ = A_c(\lambda)^+_{\mathbb{Z}/\ell\mathbb{Z}}.$

The argument for the opposite case holds until the conclusion $\mathbb{C}[(\pi^-)^{-1}_{n,\ell}(0)]\cong\mathbb{C}[(\pi^-)^{-1}_{n\ell}(0)^{\mathbb{Z}/\ell\mathbb{Z}}]$. We use Theorem~\ref{thm8} to see that 
\[
\mathrm{End}\underline{\Delta}^-(\mathrm{quo}_\ell(\lambda))\cong \mathbb{C}[(\pi^-)^{-1}_{n,\ell}(0)]\cong\mathbb{C}[(\pi^-)^{-1}_{n\ell}(0)^{\mathbb{Z}/\ell\mathbb{Z}}]\cong (\mathrm{End}\underline{\Delta}^-(\mathrm(\lambda)))_{\mathbb{Z}/\ell\mathbb{Z}}.
\]
Note now that $(\mathrm{End}\underline{\Delta}^-(\mathrm(\lambda)))_{\mathbb{Z}/\ell\mathbb{Z}}=A_c(\lambda^T)^-_{\mathbb{Z}/\ell\mathbb{Z}}$, therefore 
\[
A_c(\lambda^T)^-_{\mathbb{Z}/\ell\mathbb{Z}}= (\mathrm{End}\underline{\Delta}^-(\mathrm(\lambda)))_{\mathbb{Z}/\ell\mathbb{Z}}\cong \mathrm{End}\underline{\Delta}^-(\mathrm{quo}_\ell(\lambda))=A_{\overline{c}}(\mathrm{quo}_\ell(\lambda)^\sharp)^-.
\]
\end{proof}
Using similar methods to the proof of Theorem~\ref{thm4} we will show in Section 7 that we can remove the transpose from $A(\lambda^T)^-$ in the second equality.

Theorem~\ref{thm18} allows us to understand the endomorphism rings of the baby Verma modules for the wreath product group in terms of the symmetric group. As we will see in the next section this will allow us to easily generalise the explicit presentation of $Z_c(S_n)$ to $Z_{\overline{c}}(S_n\wr\mathbb{Z}/\ell\mathbb{Z})$.
\section{The Wronskian, Schubert cells and an explicit presentation of $A(\lambda)^+$}
Here we will introduce the Wronskian and explain why it allows us to give an explicit presentation of $A(\lambda)^+$ and therefore, an explicit presentation of $Z_c(S_n)$. The Wronskian of a set of polynomials $\{f_1,\dots, f_n\}$ is the determinant
\[
\mathrm{Wr}(f_1,f_2,\dots, f_n):=\mathrm{det}\begin{bmatrix}
    f_1 & f_2 & f_3 & \dots  & f_n \\
    f_1^{(1)} & f_2^{(1)} & f_3^{(1)} & \dots  & f_n^{(1)} \\
    \vdots & \vdots & \vdots & \ddots & \vdots \\
    f_1^{(n-1)} & f_2^{(n-1)} & f_3^{(n-1)} & \dots  & f_n^{(n-1)}
\end{bmatrix}.
\]
There is a closely related map called the Wronski map, it is actually this which we will be most interested in. As we will explain next there is a connection to Schubert cells. Importantly, Schubert cells are objects which can be described explicitly in terms of generators and relations. In their paper \cite{SchubertGLN} Mukhin, Tarasov and Varachenko describe how to write the ring of functions on a Schubert cell labelled by a partition $\lambda\vdash n$ explicitly. Let us describe their construction. Given $\lambda\vdash n$ define positive integers $d_i=\lambda_i+n-i$ and denote the set of these by $P=\{d_1,\dots, d_n\}$. Let $\C_d[u]$ be the space of polynomials in $u$ of degree less than $d$. Following the notation of Lemma~\ref{lem1} the Schubert cell $\Omega_\lambda^{qe}\subset \mathrm{Gr}(N,\C_d[u])$ is a subvariety of the Grassmannian of all $N$-dimensional subspaces of $\C_d[u]$ and consists of subspaces $X$ with basis
\begin{equation}\label{equation15}
f_i= u^{d_i}+\sum^{d_i}_{j=1,\, d_i-j\not\in P}f_{i,j}u^{d_i-j}.
\end{equation}
The algebra $\mathbb{C}[\Omega_\lambda^{qe}]$ is a free polynomial algebra with generators $f_{ij}$, i.e.
\[
\mathbb{C}[\Omega_\lambda^{qe}]=\mathbb{C}[f_{ij},\, i=1,\dots, n,\,\ j=1,\,\dots d_i,\, d_i-j\not\in P ].
\] 
Note that this is a graded algebra if we set $\mathrm{deg}(f_{ij})=j$. The Wronskian of a basis of $X$ is a polynomial of degree $n$. We write
\[
\mathrm{Wr}(f_1,\dots,f_n)=u^n+r_1u^{n-1}+\dots+r_n.
\]
The Wronski map $\mathrm{Wr}_{\lambda}:\Omega_\lambda^{qe}\rightarrow \C^n$ is defined on elements by 
\[
\mathrm{Wr}_{\lambda}(X)=(a_1,\dots,a_n) \textnormal{ if } \mathrm{Wr}(X)=u^n+\sum^n_{i=1} (-1)^ia_i u^{n-i}.
 \]
Therefore the scheme theoretic fibre of the Wronski map is
\begin{equation}\label{equation16}
\mathbb{C}[\mathrm{Wr}_\lambda^{-1}(a)]\cong\mathbb{C}[\Omega_\lambda^{qe}]/I_{\lambda,a},
\end{equation}
where $I_{\lambda,a}$ is the ideal generated by the $r_s -(-1)^sa_s$. 

Thanks to the isomorphism \eqref{equation16} we see that the scheme theoretic fibre of the Wronski map can be written explicitly as we know how to do so for the ring of functions on the Schubert cell. It only remains for us to connect the notion of the Wronski map with the map $\pi$ from Sections $4$ and $5$. This is done below in the following theorem, but it is extremely important to note this holds for the symmetric group case only.
\begin{thm}\label{thm19}
If $W=S_n$, there is an isomorphism of varieties
\[
\pi^{-1}(0)\cong \mathrm{Wr}_{\lambda}^{-1}(0).
\]
where $\pi$ is the map on spectra $\pi:\mathrm{Spec}\,\mathrm{End}\underline{\Delta}(\lambda)\rightarrow \mathfrak{h}/S_n$ induced by the inclusion map $\mathbb{C}[\mathfrak{h}]^{S_n}\hookrightarrow \mathrm{End}\underline{\Delta}(\lambda)$.
\end{thm}
\begin{proof}
See \cite[Proposition 6.4]{BellSchubert}.
\end{proof}
Thanks to the results from Section $4$ we easily deduce the following.
\begin{thm}\label{thm20}
There is an isomorphism of algebras
\[
A(\lambda)^+\cong \mathbb{C}[\mathrm{Wr}^{-1}_\lambda(0)].
\]
\end{thm}
\begin{proof}
Theorem~\ref{thm19} states $\pi^{-1}(0)\cong \mathrm{Wr}_\lambda^{-1}(0)$, hence $\mathbb{C}[\pi^{-1}(0)]=\mathbb{C}[\mathrm{Wr}_\lambda^{-1}(0)]$. By Corollary~\ref{cor2} we have that $A(\lambda)^+\cong \mathbb{C}[\pi^{-1}(0)]$ and therefore $A(\lambda)^+\cong\mathbb{C}[\mathrm{Wr}^{-1}_\lambda(0)].$
\end{proof}
Consider the case $a=0$ in \eqref{equation16}. Theorem~\ref{thm20} implies that $A(\lambda)^+$ is the quotient of $\mathbb{C}[\Omega_\lambda^{qe}]$ by the ideal generated by the coefficients $r_s$ of the polynomial $\mathrm{Wr}(f_1,\dots, f_n)$. Let us now describe how to find the algebra $A(\lambda)^+$ explicitly in terms of generators and relations in the symmetric case. First calculate the Wronskian 
\[
\mathrm{Wr}(f_1,f_2,\dots, f_n)=\mathrm{det}\begin{bmatrix}
    f_1 & f_2 & f_3 & \dots  & f_n \\
    f_1^{(1)} & f_2^{(1)} & f_3^{(1)} & \dots  & f_n^{(1)} \\
    \vdots & \vdots & \vdots & \ddots & \vdots \\
    f_1^{(n-1)} & f_2^{(n-1)} & f_3^{(n-1)} & \dots  & f_n^{(n-1)}
\end{bmatrix}
=u^n+r_1u^{n-1}+\dots+r_n
\]
of the polynomials 
\[
f_i=u^{d_i}+\sum^{d_i}_{j=1,\,\,d_i-j\not\in P} f_{ij}u^{d_i-j}.
\]
The algebra $A(\lambda)^+$ is given by taking the polynomial algebra generated by the $f_{ij}$ and quotienting by the coefficients $r_s$ of the Wronskian.
\begin{example}\label{exam1}
For the partition $\lambda=(3,2)$ we have $d_1=7$, $d_2=5$, $d_3=2$, $d_4=1$ and $d_5=0$. Therefore $f_1(u)=u^7+f_{11}u^6+f_{13}u^4+f_{14}u^3$, $f_2(u)=u^5+f_{21}u^4+f_{22}u^3$, $f_3(u)=u^2$, $f_4(u)=u$ and $f_5(u)=1$. Let us calculate the Wronskian, which is
\[
\mathrm{det}\begin{bmatrix}
  u^7+f_{11}u^6+f_{13}u^4+f_{14}u^3 & u^5+f_{21}u^4+f_{22}u^3 & u^2  & u&  1 \\
   7u^6+6f_{11}u^5+4f_{13}u^3+3f_{14}u^2 & 5u^4+4f_{21}u^3+3f_{22}u^2& 2u & 1  & 0 \\
    42u^5+30f_{11}u^4+12f_{13}u^2+6f_{14}u & 20u^3+12f_{21}u^2+6f_{22}u & 2 & 0 & 0 \\
    210u^4+120f_{11}u^3+24f_{13}u+6f_{14} & 60u^2+24f_{21}u+6f_{22} & 0 & 0 & 0 \\
    840u^3+360f_{11}u^2+24f_{13} & 120u+24f_{21} & 0 & 0 & 0 
\end{bmatrix}.
\]
Hence
\[
\mathrm{Wr}(f_1(u),f_2(u),f_3(u),f_4(u),f_5(u))=25200u^5+(14400f_{11}+30240f_{21})u^4
\]
\[
+(11520f_{11}f_{21}+10080f_{22})u^3
+(-2880f_{13}+4320f_{11}f_{22})u^2-1440f_{14}u+(-288f_{14}f_{21}+288f_{13}f_{22}).
\]
Then $A(3,2)^+$ is the quotient by the ideal generated by the coefficients, this can be equivalently written as
\[
A(3,2)^+=\C[f_{11},f_{13},f_{14},f_{21},f_{22}]/(10f_{11}-21f_{21},8f_{11}f_{21}+7f_{22},2f_{13}-3f_{11}f_{22},f_{14}, f_{14}f_{21}-f_{13}f_{22}).
\]
Hence
\[
A(3,2)^+\cong\C[f_{11}]/(f_{11}^5).
\]
\end{example}
To be able to construct the algebras $A(\lambda)^+$ directly from the partition $\lambda\vdash n$ will require a greater understanding of both the Wronski map and the dimension of $A(\lambda)^+$. It is the latter of these that we now focus on, we shall address the former in the next section.

The following results are inspired by the formula  
\[
s_\lambda(1,q,\dots,q^n)=\frac{\prod^n_{i=1}(1-q^i)}{\prod_{(i,j)\in D_{\lambda}}(1-q^{h(i,j)})}.
\]
This can be found in \cite[Page. 364]{Stembridge}. The term $s_\lambda$ denotes the Schur function associated to the partition $\lambda\vdash n$, $D_\lambda$ denotes the Young diagram of $\lambda$ and $h(i,j)$ is the hook length of the box $(i,j)$. A similar formula allows us to calculate the graded dimension of $A(\lambda)^+$.

We now record two general lemmata about the Wronskian and certain homogeneous polynomials.
\begin{lem}\label{lem14}
If $f_1,\dots,f_n,$ is a family of homogeneous polynomials and $\mathrm{Wr}(f_1,\dots,f_n)\neq 0$ then the Wronskian is homogeneous and $\deg(\mathrm{Wr}(f_1,\dots,f_n))=\sum_i(\mathrm{deg}(f_i))-\frac{(n-1)(n)}{2}$.
\end{lem}
\begin{proof}
We proceed by induction on $n$, the case $n=1$ being trivial. Suppose the lemma holds for all positive integers less than $n$. Then 
\[
\mathrm{det}\begin{bmatrix}
    f_1 & f_2 & f_3 & \dots  & f_n \\
    f_1^{(1)} & f_2^{(1)} & f_3^{(1)} & \dots  & f_n^{(1)} \\
    \vdots & \vdots & \vdots & \ddots & \vdots \\
    f_1^{(n-1)} & f_2^{(n-1)} & f_3^{(n-1)} & \dots  & f_n^{(n-1)}
\end{bmatrix}=\sum_{i=1}^{n}(-1)^{i+1}f_i\begin{bmatrix}
    f_1^{(1)} & \hat{f}_i^{(1)} & \dots  & f_n^{(1)} \\
    f_1^{(2)} & \hat{f}_i^{(2)}  & \dots  & f_n^{(2)} \\
    \vdots & \vdots & \ddots & \vdots \\
    f_1^{(n-1)} & \hat{f}_i^{(n-1)}  & \dots  & f_n^{(n-1)}
\end{bmatrix}
\]
where the $\hat{}$ symbol denotes an omitted column. By the inductive hypothesis, each of the components of the sum is a homogeneous polynomial of degree 
\[
\mathrm{deg}(f_i)+\sum_{j,j\neq i}^n\mathrm{deg}(f_j^{(1)})-\frac{(n-2)(n-1)}{2} =\mathrm{deg}(f_i)+\sum_{j,j\neq i}^n\mathrm{deg}(f_j)-(n-1)-\frac{(n-2)(n-1)}{2}.
\]
This then simplifies to
\[
\sum_{i=1}^n \mathrm{deg}(f_i)-(n-1)-\frac{(n-2)(n-1)}{2}=\sum_{i=1}^n\mathrm{deg}(f_i)-\frac{(n-1)n}{2}.
\]
\end{proof}
In the following lemma it is important to recall that when considering the polynomials defined in \ref{equation15} the generators $f_{ij}$ have degree $j$. For instance the polynomial 
\[
u^{3}+f_{12}u
\]
is homogeneous as both $u^3$ and $f_{12}u$ have degree $3$.
\begin{lem}\label{lem15}
Let $\lambda\vdash n$ be a partition of $n$, and $\{f_1,\dots ,f_n\}$ the family of polynomials as defined in (\ref{equation15}). The Wronskian $\mathrm{Wr}(f_1,\dots, f_n)$ is a homogeneous polynomial of degree $n$.
\end{lem}
\begin{proof}
By Lemma~\ref{lem14} if the Wronskian is non-zero then it is a homogeneous polynomial with degree $\sum_{i}(\mathrm{deg}(f_i))-\frac{(n-1)(n)}{2}$. We have $\mathrm{deg}(f_i)=d_i=\lambda_i+n-i$. Hence 
\[
\sum_i^{n}(\mathrm{deg}(f_i))-\frac{(n-1)(n)}{2}=\sum^n_i(\lambda_i+n-i)-\frac{(n-1)(n)}{2}=\sum^n_i \lambda_i +\sum^n_i n+ \sum^n_i -i-\frac{(n-1)(n)}{2},
\]
and 
\[
\sum^n_i \lambda_i +\sum^n_i n+ \sum^n_i -i-\frac{(n-1)(n)}{2}=n+n^2-\frac{n(n+1)}{2}-\frac{(n-1)(n)}{2}=n.
\]
\end{proof}
It will be important to us that $A(\lambda)^+$ is a complete intersection. to prove this we need to use the following non-trivial fact which is can be found in \cite[Lemma 3.11]{SchubertGLN}.
\begin{lem}\label{lem16}
The algebra $A(\lambda)^+$ is finite dimensional.
\end{lem}
\begin{lem}\label{lem17}
For any $\lambda\vdash n$, the algebra $A(\lambda)^+$ is a complete intersection.
\end{lem}
\begin{proof}
The algebra $A(\lambda)^+$ has $n$ generators and $n$ relations. Therefore, it is a complete intersection if and only if its Krull dimension is zero.  Since $A(\lambda)^+$ is finite dimensional as a vector space it is Artinian and therefore every prime ideal is maximal \cite[Proposition 8.1]{atiyah2018introduction}. Hence it has Krull dimension zero.
\end{proof}
Assume that $\C[x_1, \dots, x_m]/(g_1,\dots,g_n)$ is graded with $\deg (x_i) = a_i > 0$, so that each $g_j$ is homogeneous, of degree $b_j$ say. The Hilbert-Poincar\'e polynomial of $A$ is defined to be
\[
P(A,q) := \sum_{i \ge 0} (\dim A_i) q^i.
\]
\begin{lem}\label{lem18}
If $A$ is a graded complete intersection then
\[
P(A,q) = \frac{\prod_{i = 1}^t (1-q^{b_i})}{\prod_{j = 1}^n (1-q^{a_j})}.
\]
\end{lem}
\begin{proof}
Since $A$ is a complete intersection we can write
\[
 A=\frac{\C[x_1, \dots, x_m]}{(g_1, \dots, g_t)}
\]
where $(g_1, \dots, g_t)$ is a regular sequence. Let $S = \C[x_1, \dots, x_m]$. The Koszul resolution \cite[Chapter 17]{Eisenbud} can be used to resolve $A$ as a graded $S$-module. Let $V = \mathrm{Span}_{\C} \{ g_1, \dots, g_t \}$, a graded vector space. Then
\[
0 \rightarrow \wedge^t V \otimes_{\C} S \rightarrow \wedge^{t-1} V \otimes_{\C} S \rightarrow \dots \rightarrow \wedge^0 V \otimes S \rightarrow A \to 0
\]
is an exact sequence of graded $S$-modules since $(g_1, \dots, g_t)$ is regular \cite[Corollary 1.6.14 (b)]{CohMac}. The maps in the Koszul resolution are $d_k:\wedge^k V \otimes_{\C} S \rightarrow \wedge^{k-1} V \otimes_{\C} S$ and defined on elements
\[
d_k(v_1\wedge\dots\wedge v_t\otimes f)=\sum^t_{i=1} (-1)^{i+1}v_1\wedge\dots \wedge \hat{v_i}\wedge\dots\wedge v_t\otimes v_i f,
\]
here $\hat{v_i}$ means the term $v_i$ is omitted. This implies (by the “Euler-Poincar\'e principle’’) that
\[
P(A,q)  =  \sum_{j = 0}^t (-1)^{j+1} P(\wedge^{j} V \otimes_{\mathbb{C}} S, q) 
\]
\[
= \left( \sum_{j = 0}^t (-1)^{j+1} P(\wedge^{j} V, q) \right) P(S,q) .
\]
We have
\[
P(S,q) = \left( \prod_{j =1}^n (1-q^{a_j}) \right)^{-1} \textnormal{ and } \sum_{j = 0}^t (-1)^{j+1} P(\wedge^{j} V, q) = \prod_{i = 1}^t (1- q^{b_i}).
\]
\end{proof}
We can now present the formula for calculating the graded dimension of $A(\lambda)^+$. Recall that $D_\lambda$ denotes the Young diagram for a partition $\lambda$ and $h(i,j)$ is the hook length of the cell $(i,j)$.
\begin{thm}\label{thm21}
For any $A(\lambda)^+$, we have
\[
\sum_{i \ge 0} (\dim A(\lambda)^+_i) q^i=\frac{\prod_{i = 1}^n (1-q^{i})}{\prod_{(i,j)\in D_\lambda} (1-q^{h(i,j)})}.
\]
\end{thm}
\begin{proof}
Since $A(\lambda)^+$ is a complete intersection by Lemma~\ref{lem17}, Lemma~\ref{lem18} implies that
\[
\sum_{i \ge 0} (\dim A(\lambda)^+_i) q^i=P(A(\lambda)^+,q) = \frac{\prod_{i = 1}^t (1-q^{b_i})}{\prod_{j = 1}^n (1-q^{a_j})}.
\]
Lemma~\ref{lem2} says that $\prod_{j = 1}^n (1-q^{a_j})=\prod_{(i,j)\in D_\lambda} (1-q^{h(i,j)})$. By definition of $A(\lambda)^+$,
\[
\prod_{i = 1}^t (1-q^{b_i})=\prod_{i = 1}^n (1-q^{i}).
\]
Hence,
\[
\sum_{i \ge 0} (\dim A(\lambda)^+_i) q^i=\frac{\prod_{i = 1}^n (1-q^{i})}{\prod_{(i,j)\in D_\lambda} (1-q^{h(i,j)})}.
\]
\end{proof}
The following example highlights how easy this formula is to use.
\begin{example}
The Young diagram for the partition $\lambda=(3,1,0,0)$ is
\begin{center}
\begin{ytableau}
       4 & 2 & 1  \\
       1
\end{ytableau}
\end{center}
and therefore 
\[
\sum_{i \ge 0} (\dim A(\lambda)^+_i) q^i=\frac{(1-q)(1-q^2)(1-q^3)(1-q^4)}{(1-q^4)(1-q^2)(1-q)(1-q)}=\frac{1-q^3}{1-q}=1+q+q^2.
\]
Hence $A(\lambda)^+$ consists of a one dimensional space in degrees $0$, $1$ and $2$. 
\end{example}
The formula for the graded dimensions allows us to calculate the dimension of the entire algebra $A(\lambda)^+$.
\begin{thm}\label{thm22}
The dimension of $A(\lambda)^+$ is given by the hook length formula
\[
\dim A(\lambda)^+=
\frac{n!}{\prod_{(i,j)\in D_\lambda}h(i,j)}.
\]
\end{thm}
\begin{proof}
We have the formula 
\[
\sum_{i \ge 0} (\dim A_i) q^i=\frac{\prod_{i = 1}^n (1-q^{i})}{\prod_{(i,j)\in D_\lambda} (1-q^{h(i,j)})}.
\]
We can use L'Hopitals rule to evaluate the formula when $q=1$. Clearly the left hand side gives the dimension of $A$. Repeated applications of L'Hoptials rule to the right hand side gives
\[
\frac{n!}{\prod_{(i,j)\in D_\lambda}h(i,j)}.
\]
\end{proof}
\section{Calculating $A(\lambda)^+$ directly from the partition}
As demonstrated in Section $6$ we have an algorithm for calculating the algebras $A(\lambda)^+$ using the Wronskian. Here we will refine this result and show that calculating the Wronskian is unnecessary. In fact, the algebra $A(\lambda)^+$ can be derived directly from the partition $\lambda$. Key to providing an explicit presentation of $A(\lambda)^+$ directly from a partition $\lambda\vdash n$ is understanding how the coefficients of the terms appear in the Wronskian. We shall split the problem of understanding the coefficients into two distinct cases. We say that terms of the form $f_{i,j}$ are linear and terms of the form $f_{i_1,j_1}f_{i_2,j_2}\dots f_{i_m,j_m}$ for $m>1$ are non-linear. We begin with the simpler task of understanding the linear terms. 

The first question we want to answer is how often do linear terms appear in the coefficients of the Wronskian for a given partition of $n$? It is fairly straightforward to see that each linear term can appear in only one coefficient. Recall that algebra $A(\lambda)^+$ can be written as 
\[
\frac{\mathbb{C}[f_{ij},\, i=1,\dots, n,\,\ j=1,\,\dots ,d_i,\, d_i-j\not\in P ]}{(r_1,\dots, r_n)}
\]
where the $r_s$ are homogeneous elements and $\deg(r_s)=s$. Abusing terminology we say that ``a monomial $m$ in the $f_{i,j}$ appears in $r_s$'' if the coefficient of $m$ in $r_s$ is non-zero. 
\begin{lem}\label{lem19}
If the linear term $f_{i,j}$ appears as a monomial in one of the elements $r_s$, then $j=s$. 
\end{lem}
\begin{proof}
By Lemma~\ref{lem15} the Wronskian is a homogeneous polynomial of degree $n$. Also $A(\lambda)^+$ is a complete intersection and therefore the coefficient of $u^i$ is non-zero for all $0\leq i\leq n$. In other words, $r_i\neq 0$  for all $i$. Therefore, a linear coefficient of $u^i$ has degree $n-i$. Since the linear term $f_{i,j}$ has degree $j$ we conclude that it can only appear in $r_j$.
\end{proof}
The following is a partial converse to Lemma~\ref{lem19}.
\begin{lem}\label{lem20}
Consider a finite dimensional commutative ring
\[
A=\frac{\C[x_1,\dots, x_n]}{(r_1,\dots, r_n)},
\]
where the relations $r_i$ are homogeneous and do not contain any constant terms. For each $1\leq j\leq n$, there exist $k\geq 1$ and $i$ such that $r_i$ contains the monomial $x_j^k$.
\end{lem}
\begin{proof}
Argue by contradiction. Assume that $x_1^k$ does not appear in any $r_i$ for $k\geq 1$. Since the algebra is finite dimensional and positively graded, we must have $x_1^m=0$ for some $m$ and hence 
\[
x_1^m=\sum_l c_lr_l
\]
for $c_l\in \C$. Since every monomial with non-zero coefficient in $r_j$ is divisible by some $x_i$ with $i\neq 1$, there is a well defined evaluation morphism $ev_c:A\rightarrow \C$ that sends $x_1$ to some constant $c\neq 0$ and $x_i$ to $0$ for $i>1$. Then
\[
c^m=ev_c(x_1^m)=ev_c\left(\sum_l c_lr_l\right)=\sum_l c_l ev_c(r_l)=0,
\]
which is a contradiction. 
\end{proof}
\begin{lem}\label{lem21}
Let $f_1,\dots, f_n$ be the set of pairwise distinct polynomials as in \eqref{equation15}. All terms in the Wronskian $\mathrm{Wr}(f_1,\dots ,f_n)$ of the form $f_{i,j}^{k_1}\dots f_{u,v}^{k_z}$ have $k_s\leq 1$ for all $1\leq s\leq z$.
\end{lem}
\begin{proof}
By writing the Wronskian
\[
\mathrm{Wr}(f_1,\dots, f_n)=\mathrm{det}\begin{bmatrix}
    f_1 & f_2 & f_3 & \dots  & f_n \\
    f_1^{(1)} & f_2^{(1)} & f_3^{(1)} & \dots  & f_n^{(1)} \\
    \vdots & \vdots & \vdots & \ddots & \vdots \\
    f_1^{(n-1)} & f_2^{(n-1)} & f_3^{(n-1)} & \dots  & f_n^{(n-1)}
\end{bmatrix}
\]
the statement becomes clearer. Each term in the determinant is some product of terms which do not share a row or column. Fix an $f_{i,j}$, then we see that it only appears in the $i^{th}$ column. Since the product of the terms in the Wronskian cannot share a column we see that if $f_{i,j}^{k}$ appears then $k= 1$.
\end{proof}
Lemma~\ref{lem21} states that in the case of the Wronskian our coefficients can not contain terms of a higher power than $1$, for instance we cannot have $f_{11}^2$ appearing in the relations. It is also clear that the coefficients of the Wronskian contain no constant terms, since they are homogeneous of positive degree. These observations give us the following lemma.
\begin{lem}\label{lem22}
Given a partition $\lambda\vdash n$, each $f_{i,j}$ appears as a linear term of a coefficient in the Wronskian.
\end{lem}
\begin{proof}
Follows from Lemma~\ref{lem21} and Lemma~\ref{lem20}.
\end{proof}
Now that we have proven that each linear term must appear we can strengthen Lemma~\ref{lem19}. Using the same notation as before we have the following.
\begin{prop}\label{prop5}
Each linear term $f_{i,j}$ appears in $r_j$ and with non-zero coefficient.
\end{prop}
\begin{proof}
Follows from Lemma~\ref{lem19} and Lemma~\ref{lem22}.
\end{proof}
Proposition~\ref{prop5} completely solves the problem of understanding the position of the linear terms. The next step is to prove an analogous statement for the non-linear terms. We will find that almost all non-linear terms appear in the coefficients, except for a specific few. We will need the following results first.
\begin{lem}\label{lem23}
Let $\{f_1,\dots ,f_n\}$ be the set of polynomials defined as in \eqref{equation15}. Assume there are polynomials $f_i$ and $f_j$ such that $f_i$ contains a term of the form $f_{i,s}u^k$ and $f_j$ contains a term $f_{j,t}u^k$. Then the Wronskian $\mathrm{Wr}(f_1,\dots f_n)$ contains no monomial divisible by $f_{i,s}f_{j,t}$.
\end{lem}
\begin{proof}
The determinant is a sum of multiples of elements of different columns and different rows in the matrix. The terms $f_{i,s}$ and $f_{j,t}$ only appear in the columns $i$ and $j$ respectively. Hence, all the terms with form $f_{i,s}f_{j,t}$ appearing in the determinant come from an expression of the form $F(f_i^{(a)}f_j^{(b)}-f_{i}^{(b)}f_j^{(a)})$. Here $F$ is some multiple of entries from different rows and columns excluding columns $i$ and $j$. An easy calculation gives $f_{i,s}(u^k)^{(a)}f_{j,t}(u^k)^{(b)}-f_{i,s}(u^k)^{(b)}f_{j,t}(u^k)^{(a)}=0$. 
\end{proof}
There is a useful recursive formula for the Wronskian that can be found in \cite[Proposition 1]{krusemeyer1988does}.
\begin{prop}\label{prop6}
The recursive formula for the Wronskian is given by
\[
\mathrm{Wr}(f_1,\dots, f_n)=f_1^n \mathrm{Wr}\left(\left(\frac{f_2}{f_1}\right)',\dots, \left(\frac{f_n}{f_1}\right)'\right)
\]
where $\deg (f_i)<\deg (f_j)$ for $i<j$.
\end{prop}
We will use this formula to prove two lemmata that will be crucial in showing that most non-linear terms are non-zero.
\begin{lem}\label{lem24}
Let $\{f_1,\dots, f_n\}$ be a set of monomials in one variable such that $\deg (f_i)>\deg (f_j)$ for $i<j$, and $f_i\neq 0$ for all $i$. The Wronskian $\mathrm{Wr}(f_1,\dots, f_n)$ is non-zero. 
\end{lem}
\begin{proof}
Proceed by induction on $n$. The case where $n=1$ is obvious. Assume the statement is true for all $m<n$, then use the recursive formula given in Proposition~\ref{prop6}
\[
\mathrm{Wr}(f_1,\dots, f_n)=f_1^n \mathrm{Wr}\left(\left(\frac{f_2}{f_1}\right)',\dots, \left(\frac{f_n}{f_1}\right)'\right).
\]
Clearly $\mathrm{Wr}((\frac{f_2}{f_1})',\dots, (\frac{f_n}{f_1})')$ satisfies the assumptions of the lemma. Therefore $\mathrm{Wr}((\frac{f_2}{f_1})',\dots, (\frac{f_n}{f_1})')$ is non-zero. Since $f_1^n\neq 0$, 
\[
\mathrm{Wr}(f_1,\dots, f_n)=f_1^n \mathrm{Wr}\left(\left(\frac{f_2}{f_1}\right)',\dots, \left(\frac{f_n}{f_1}\right)'\right)\neq 0.
\]
\end{proof}
\begin{lem}\label{lem25}
Let $\{f_1,\dots, f_n\}$ be a set of monomials in one variable, such that $\deg f_i\neq\deg f_j$ for all $i\neq j$ and $f_i\neq 0$ for all $i$. The Wronskian $\mathrm{Wr}(f_1,\dots, f_n)$ is non-zero. 
\end{lem}
\begin{proof}
We see from Lemma~\ref{lem24} above that if $\mathrm{deg} (f_i)>\deg(f_j)$ for $i<j$ then the Wronskian is non-zero. Since the monomials all have pairwise different degrees there is a matrix $A$ that permutes the columns of the matrix such that the monomials are in order of ascending degree in the first row. Then $\det \mathrm{Wr}(f_1,\dots, f_n)=\det(A)\det(\mathrm{Wr}(f_1',\dots, f_{n}'))$ where $f_i'>f_j'$ for $i<j$. Since $A$ is an invertible matrix its determinant is non-zero and so $\det \mathrm{Wr}(f_1,\dots, f_n)\neq 0$.
\end{proof}
This is the last result we require to state precisely where and which non-linear terms appear in the coefficients of the Wronskian. There is just one convention we must establish first. Recall Lemma~\ref{lem2}, which stated that the set
\[
\{j\,|\,d_i-j\not\in P \textnormal{ for } 1\leq j\leq d_i\}
\]
is equal to the set of hook lengths in the $i^{th}$ row of length $\lambda_i$.
Also recall that 
\[
f_i=u^{d_i}+\sum^{d_i}_{j=1,\,\,d_i-j\not\in P} f_{i,j}u^{d_i-j}.
\]
Lemma~\ref{lem2} implies there is a bijection between the polynomials $f_{i,j}$ for a fixed $i$ and the cells in the $i^{th}$ row of $\lambda$. This bijection sends the cell $(i,j)$ to $f_{i,h(i,j)}$. To demonstrate this bijection let us consider an example.
\begin{example}
Take the partition $(3,2)$. The Young diagram is 
\begin{center}
\begin{ytableau}
       4 & 3 & 1  \\
       2 & 1
\end{ytableau}.
\end{center}
The cells of the first row, read left to right, are mapped to $f_{1,4}$, $f_{1,3}$ and $f_{1,1}$ respectively. The cells of the second row are similarly mapped to the generators $f_{2,2}$ and $f_{2,1}$.
\end{example}
We say that two generators $f_{i,j}$ and $f_{s,t}$ share a row or column if they share a row or column in the Young diagram under this mapping.
\begin{thm}\label{thm23}
Fix a partition $\lambda\vdash n$ and let $f_1,\dots, f_n$ be as in \eqref{equation15}. The non-linear coefficients of $\mathrm{Wr}(f_1,\dots, f_n)$ are all non-zero except for monomials divisible by $f_{i,j}f_{s,t}$ where $i=j$ or $h(i,j)=h(s,t)$. In other words, they have no factors that share a row or column in the partition diagram. 
\end{thm}
\begin{proof}
The first thing to note is that if two generators $f_{i,j}$ and $f_{s,t}$ share the same row in the partition then $i=s$ and they must appear in the same column in the Wronskian. Hence, they cannot appear in the determinant. Assume now that $f_{i,j}$ and $f_{s,t}$ appear in the same column of $D_\lambda$. Lemma~\ref{lem3} says that
\[
d_a-h(a,b)=d_c-h(c,b)
\]
holds for all $a$, $b$ and $c$. In particular if $f_{i,j}$ and $f_{s,t}$ share the same column in the Young diagram then $d_i-j=d_s-t$. From the definition of the $f_i$ in \eqref{equation15}, we see that in $f_i$ the monomial $f_{i,j}$ is the coefficient of $u^{d_i-j}$. Likewise $f_{s,t}$ is the coefficient of $u^{d_s-t}$ in $f_{s}$. Lemma~\ref{lem23} then implies that there is no monomial in the Wronskian that is divisible by $f_{i,j}f_{s,t}$. We need only prove that the other nonlinear terms are non-zero.

We prove this for products of two monomials, the general case follows from a similar argument using the coefficients in Proposition~\ref{prop7}. Assume that $f_{i,j}$ and $f_{s,t}$ share neither a row nor a column in the Young diagram. In the determinant $f_{i,j}f_{s,t}$ will be the coefficient of the $u^{n-j-t}$ term. This observation lets us see that when deciding if this is nonzero we need only consider entries in the Wronskian that are scalars or powers of $u$. That is we exclude all $f_{y,z}$ where $f_{y,z}\neq f_{i,j}$ or $f_{s,t}$. Hence we need only check that $\mathrm{Wr}(u^{d_1},\dots, f_{i,j}u^{d_i-j},\dots, f_{s,t}u^{d_s-t},\dots, u^{d_n})$ is non-zero. This simplifies  
\[
\mathrm{Wr}(u^{d_1},\dots, f_{i,j}u^{d_i-j},\dots, f_{s,t}u^{d_s-t},\dots, u^{d_n})=f_{i,j}f_{s,t}\mathrm{Wr}(u^{d_1},\dots, u^{d_i-j},\dots, u^{d_s-t},\dots, u^{d_n}).
\]
Since $f_{i,j}$ and $f_{s,t}$ do not share a column in $D_\lambda$, Lemma~\ref{lem3} implies that $d_i-j\neq d_s-t$. Therefore, the degrees of all the monomials are pairwise different and non-zero. In this case Lemma~\ref{lem25} implies that the determinant is non-zero.
\end{proof}
These results can be improved upon by giving a formula for calculating the scalar coefficients of the linear and non-linear terms in the Wronskian. Let us explain some necessary notation. Recall that the recursive formula in Proposition~\ref{prop6} is given for a particular order of entries, namely that they are increasing in degree. This is often not the case, and so we must permute the columns of the Wronskian first. Let $A$ be the matrix that permutes in the desired order. Note that $\mathrm{det}(A)=\sigma(A)$, where $\sigma$ is the sign function.
\begin{prop}\label{prop7}
The scalar coefficient of a given term $f_{i_1j_1}\dots f_{i_mj_m}$ is 
\[
\sigma(A) \left(\prod_{\stackrel{i<j,i\neq i_k,}{ j\neq j_k}} d_i-d_j\right)\left(\prod_{d_{i_k}-j_k>d_j \textnormal{ and }1\leq k\leq m}(d_{i_k}-j_k-d_l)\right)\left(\prod_{d_i-i_k>d_j-j_l}(d_i-i_k-d_j-j_l)\right).
\]
\end{prop}
\begin{proof}
The proof is by induction, on the size of Wronskian, the case $m=1$ is clear. Assume the statement holds for all $m$ up to $n-1$. Consider the case $m=n$. We have  
\[
\mathrm{Wr}(u^{d_1},\dots, f_{i_1j_1}u^{d_{i_1}-j_1},\dots, u^d_n)=
\det(A)\mathrm{Wr}(u^d_n,\dots, f_{i_1j_1}u^{d_{i_1}-j_1},\dots, u^{d_1})
\]
\[
=u^{d_n}\mathrm{Wr}((u^{d_{n-1}}/u^{d_n})',\dots, f_{i_1j_1}(u^{d_{i_1}-j_1}/u^{d_n})',\dots, (u^d_1/u^{d_n})')
\]
\[
=u^{d_n}\mathrm{Wr}((d_{n-1}-d_{n})u^{d_{n-1}-d_{n}-1},\dots, f_{i_1j_1}(d_{i_1}-j_1-d_n)u^{d_{i_1-j_1-d_n-1}}, \dots, (d_1-d_n)u^{d_1-d_n-1})
\]
\[
=\prod_{1\leq i< n}(d_i-d_n)\prod_{d_{i_k}-j_k>d_n}(d_{i_k}-j_k-d_n) f_{i_1j_1}\dots f_{i_mj_m}\mathrm{Wr}(u^{d_{n-1}-d_{n}-1},\dots, u^{d_{i_1-j_1-d_n-1}}, \dots, u^{d_1-d_n-1}).
\]
and therefore
\[
\sigma(A) \left(\prod_{\stackrel{i<j,i\neq i_k,}{ j\neq j_k}} d_i-d_j\right)\left(\prod_{d_{i_k}-j_k>d_j \textnormal{ and }1\leq k\leq m}(d_{i_k}-j_k-d_l)\right)\left(\prod_{d_i-i_k>d_j-j_l}(d_i-i_k-d_j-j_l)\right).
\]
\end{proof}
By defining a new term $e$ we can greatly simplify this expression.
\begin{prop}\label{prop8}
Consider the nonlinear term $f_{i_1j_1}\dots f_{i_mj_m}$ ordered so that $i_a<i_{a+1}$. Define
\[ 
e_{i_k} = \left\lbrace \begin{array}{ll} d_{i_k}-j_k & \textrm{if }  1\leq k\leq m  \\ d_{i_k} & \textrm{else. } \end{array} \right.
\]
Then the scalar coefficient of the term of $f_{i_1j_1}\dots f_{i_mj_m}$ in the Wronskian is 
\[
\prod_{1\leq i<j\leq n} e_i-e_j.
\]
\end{prop}
\begin{proof}
From Proposition~\ref{prop7} we see that 
\begin{equation}\label{equation17}
\sigma(A)\prod_{\stackrel{i<j,i\neq i_k,}{ j\neq j_k}} d_i-d_j\prod_{d_{i_k}-j_k>d_j \textnormal{ and }1\leq k\leq m}(d_{i_k}-j_k-d_l)\prod_{d_i-i_k>d_j-j_l}(d_i-i_k-d_j-j_l)
\end{equation}
is the coefficient of $f_{i_1j_1}\dots f_{i_mj_m}$. By definition of the $e_i$, the term
\begin{equation}\label{equation18}
\prod_{1\leq i<j\leq n} e_i-e_j.
\end{equation}
is either equal to \eqref{equation17} or its negative. Now note that $\eqref{equation17}$ is negative or positive precisely when $\sigma(A)$ is negative or positive. If $\sigma(A)$ is negative then there is an odd number of terms $e_i-e_j$ such that $e_i-e_j<0$. In this case \eqref{equation18} is also negative. By a similar argument \eqref{equation17} is positive precisely when \eqref{equation18} is.
\end{proof}
It is now possible, by collecting our results so far to explicitly describe $A(\lambda)^+$ with no mention of the Wronskian. The following theorem explains how the structure of $A(\lambda)^+$ depends only on the partition $\lambda\vdash n$ and its various hook lengths. As a consequence we can now calculate the centre of $\overline{H}_c(S_n)$ for generic $c$ and any $n$ directly from the partitions of $n$.
\begin{thm}\label{thm24}
Let $\lambda\vdash n$ be a partition of length $t$. The algebra $A(\lambda)^+$ is the quotient
\[
A(\lambda)^+\cong \mathbb{C}[f_{i_1,h(i_1,j_1)},\dots, f_{i_t,h(i_t,j_t)}]/I
\]
by the ideal $I$ that is generated by $n$ homogeneous elements $r_1,\dots,r_n$. The $r_i$ are ordered so that $\mathrm{deg}(r_i)=i$. The monomials in $r_i$ are all products of the form $f_{i_k,h(i_k,j_k)}\dots f_{i_\ell,h(i_\ell,j_\ell)}$ such that $u\neq v$ and $w\neq x$ for any two factors $f_{u,h(u,w)}$ and $f_{v,h(v,x)}$. The coefficients of the monomials inside the $r_i$ are given by Proposition~\ref{prop8}. 
\end{thm}
\begin{proof}
Follows from Proposition~\ref{prop5}, Theorem~\ref{thm23} and Proposition~\ref{prop8}. 
\end{proof}
We can rewrite Theorem~\ref{thm24} using a bijection between the generators $f_{i,h(i,j)}$ and the Young diagram of the partitions. Assign the generator $f_{i,h(i,j)}$ to the cell $(i,j)$ in $D_\lambda$.
\begin{thm}\label{thm25}
Let $\lambda\vdash n$ be a partition. The algebra $A(\lambda)^+$ is the quotient 
\[
A(\lambda)^+\cong \mathbb{C}[D_\lambda]/I
\]
by the ideal $I$ that is generated by $n$ homogeneous elements $r_1,\dots,r_n$. The $r_s$ are ordered so that $\mathrm{deg}(r_s)=s$. The monomials appearing in $r_i$ are products of cells which share neither a row or column in $D_\lambda$. In other words if $\square_{i,j}\square_{k,\ell}$ is a factor of some monomial in the $r_s$ we must have that $i\neq k$ and $j\neq \ell$. The coefficients of the generators of $I$ are given by Proposition~\ref{prop8}. 
\end{thm}
\begin{proof}
Follows from Theorem~\ref{thm24} using the described bijection.
\end{proof}
Theorem~\ref{thm25} demonstrates that we can use a completely combinatorial approach to find the centre of $\overline{H}_c(S_n)$ by directly deriving a presentation of $A(\lambda)^+$ from the Young diagram. This is much simpler than the previously described method using the Wronskian. To demonstrate this point we will calculate $A(3,2)^+$ using Theorem~\ref{thm25}, compare this with Example~\ref{exam1}.
\begin{example}
Let $\lambda=(3,2)$. First lets write out the Young diagram with the hook lengths of the cells included
\begin{center}
\begin{ytableau}
       4 & 3 & 1  \\
       2 & 1
\end{ytableau}.
\end{center}
Therefore our generators are $f_{1,4}$, $f_{1,3}$, $f_{1,1}$, $f_{2,2}$ and $f_{2,1}$. Hence,
\[
A(\lambda)^+\cong\mathbb{C}[f_{1,4},f_{1,3},f_{1,1},f_{2,2},f_{2,1}]/(r_1,r_2,r_3,r_4,r_5).
\]
The relations $r_i$ contain the all the linear terms. The relations $r_i$ also contain each nonlinear term except those that have factors sharing a row or column. From the Young diagram we see that the non-zero nonlinear terms are then $f_{1,1}f_{2,1}$, $f_{1,1}f_{2,2}$, $f_{1,3}f_{2,2}$ and $f_{1,4}f_{2,1}$. Furthermore, we know that the terms appear in the relations according to their degree. So we need only compute the coefficients, which can be done using Proposition~\ref{prop8}. Recall that $d_1=7$, $d_2=5$, $d_3=2$, $d_4=1$ and $d_5=0$, hence $e_1=6$, $e_2=4$, $e_3=2$, $e_4=1$, $e_5=0$. Let us calculate the coefficient of $f_{1,1}f_{2,1}$. This is the product
\[
\prod_{1\leq i<j\leq n} e_i-e_j=(6-4)(6-2)(6-1)(6-0)(4-2)(4-1)(4-0)(2-1)(2-0)(1-0).
\]
which is equal to $11520$. The other coefficients are similarly calculated. The relations are
\[
14400f_{11}+30240f_{21}
\]
\[
11520f_{11}f_{21}+10080f_{22}
\]
\[
-2880f_{13}+4320f_{11}f_{22}
\]
\[
-1440f_{14}
\]
\[
-288f_{14}f_{21}+288f_{13}f_{22}.
\]
Therefore, $A(3,2)^+\cong\C[f_{11}]/(f_{11}^5)$.
\end{example}
Theorem~\ref{thm25} allows us to easily calculate $A(\lambda)^+$ for small $\lambda\vdash n$. 

Recall that in Section $3$, we described how given an isomorphism $\psi:H_{\overline{c}}(W)\rightarrow H_c(W)$ and a $H_c(W)$-module $M$, we can make a $H_{\overline{c}}(W)$-module $M^\psi$. In particular, define an isomorphism $\sigma: H_c(S_n)\rightarrow H_{-c}(S_n)$ by \[
\sigma(x)=x,\textnormal{ } \sigma(y)=y\textnormal{ and }\sigma(w)=(-1)^{sgn(w)}w \textnormal{ for all }x\in \mathfrak{h},\textnormal{ } y\in \mathfrak{h}\textnormal{ and } w\in S_n.
\]
\begin{thm}\label{thm26}
Let $\lambda\vdash n$ be a partition and $\lambda^T$ its transpose. There is an isomorphism of $H_c(S_n)$-modules $\underline{\Delta}_{-c}(\lambda)^\sigma\cong \underline{\Delta}_c(\lambda^T)$.
\end{thm}
\begin{proof}
Let $S=\mathbb{C}\cdot s$ be the sign representation. There is an isomorphism of irreducible representations $\lambda^T\cong \lambda\otimes S$, see \cite[Page. 36]{JK}. We claim there is an isomorphism of $H_{c}(W)$-modules $\phi: \underline{\Delta}_{c}(\lambda\otimes S)\rightarrow \underline{\Delta}_{-c}(\lambda)^\sigma$ given by 
\[
\phi(h\otimes v\otimes s)=\sigma(h)\otimes v,
\]
and extended linearly. We must check that $\phi$ is well-defined. Since the module $\underline{\Delta}(\lambda\otimes S)$ is uniquely defined by the following conditions
\begin{enumerate}
    \item $y\cdot 1\otimes v\otimes s=0$ for all $y\in \mathfrak{h}$
    \item $w\cdot 1\otimes v\otimes s=1\otimes w\cdot (v\otimes s)=1\otimes wv\otimes ws=1\otimes wv\otimes sgn(w)s$
\end{enumerate}
we check that $\phi$ preserves these conditions.
\[
 \phi(y\cdot 1\otimes v\otimes s)=\phi(y\otimes v)=\sigma(y)\otimes v=y\otimes v=0
\]
and
\[
 \phi(w\cdot 1\otimes v\otimes s)=\phi( w\otimes v)=\sigma(w) \otimes v=sgn(w)w\otimes v= sgn(w)\otimes wv
\]
and
\[
sgn(w)\otimes wv=\phi(sgn(w)\otimes wv\otimes s)=\phi(1\otimes wv\otimes sgn(w)s). 
\]
This map is surjective since $\underline{\Delta}_{-c}(\lambda)^{\sigma}$ is generated by $1\otimes \lambda$. It is injective because both $\underline{\Delta}_c(\lambda\otimes S)$ and $\underline{\Delta}_{-c}(\lambda)^\sigma$ are free $\mathbb{C}[\mathfrak{h}]$-modules of rank $\dim\lambda$. It is a morphism of $H_{\overline{c}}(S_n)$-modules because
\begin{equation}\label{eq1}
\begin{split}
a\cdot \phi(h\otimes v\otimes s) & = a\cdot (\sigma(h)\otimes v) \\
 & = \sigma(a)\sigma(h)\otimes v \\
 & = \sigma(ah)\otimes v \\
  & = \phi(ah\otimes v\otimes s) \\
 & = \phi(a\cdot h\otimes v\otimes s).
\end{split}
\end{equation}
Therefore $\phi$ is an isomorphism.
\end{proof}
Now define another isomorphism $\eta:H_c(S_n)\rightarrow H_{-c}(S_n)$ by 
\[
\eta(x)=-x,\textnormal{ } \eta(y)=-y\textnormal{ and }\eta(s)=s \textnormal{ for all }x\in \mathfrak{h},\textnormal{ } y\in \mathfrak{h}\textnormal{ and } s\in S_n.
\]
\begin{lem}\label{lem26}
Let $\lambda\vdash n$ be a partition. We have an isomorphism of $H_c(S_n)$-modules $\underline{\Delta}_{c}(\lambda)^{\eta}\cong \underline{\Delta}_{-c}(\lambda)$.
\end{lem}
\begin{proof}
The proof is similar to the proof of Theorem~\ref{thm26}.
\end{proof}
\begin{cor}\label{cor3}
There is an isomorphism $\underline{\Delta}_c(\lambda)^{\eta\circ\sigma}\cong\underline{\Delta}_c(\lambda^T)$.
\end{cor}
\begin{proof}
This follows from Theorem~\ref{thm26} and Lemma~\ref{lem26}.
\end{proof}
\begin{lem}\label{lem27}
The isomorphism $\eta\circ \sigma$ induces an isomorphism of baby Verma modules
\[
\eta\circ \sigma:\Delta(\lambda)^{\eta\circ\sigma}\cong \Delta(\lambda^T).
\]
\end{lem}
\begin{proof}
By Theorem~\ref{thm5} $\Delta(\lambda)=\underline{\Delta}(\lambda)/R_+\underline{\Delta}(\lambda)$ and so we need to show that 
\[
\eta\circ\sigma(R_+\underline{\Delta}(\lambda))=R_+\underline{\Delta}(\lambda^T).
\]
Since $\eta\circ\sigma$ is an isomorphism of $H_c(W)$-modules we know that 
\[
\eta\circ\sigma(R_+\underline{\Delta}(\lambda))=(\eta\circ\sigma)(R_+)(\eta\circ\sigma)(\underline{\Delta}(\lambda)).
\]
Since $\eta\circ\sigma(\underline{\Delta}(\lambda))=\underline{\Delta}(\lambda^T)$ we need show that $\eta\circ\sigma(R_+)=R_+$. Note that by the definition of $R_+$ it contains no group elements and so $\sigma$ is simply the identity on $R_+$. Hence we show that $\eta(R_+)=R_+$. By linearity it suffices to check on homogeneous elements in $R_+$. Again this is clear since applying $\eta$ to an element of $R_+$ we find 
\[
\eta(p(x)_+\otimes q(y)+r(x)\otimes s(y)_+)=(-1)^m p(x)_+\otimes q(y)+(-1)^n r(x)\otimes s(y)_+
\] where $m=|\deg(p)||\deg(q)|$ and $n=|\deg(r)||\deg(s)|$. Therefore $\eta(R_+)=R_+$ and the proof is complete.
\end{proof}
\begin{lem}\label{lem28}
Let $H$ be an algebra, $Z$ its centre and $a:H\rightarrow H$ an automorphism. For any $H$-module $M$, $a$ induces an isomorphism
\[
a:Z/ \mathrm{ann}_Z M \rightarrow Z/\mathrm{ann}_Z M^a.
\]
\end{lem}
\begin{proof}
We show that $a(\mathrm{ann}_{Z}M)=\mathrm{ann}_{Z(W)}M^a$. This can be seen by noting that $z\in \mathrm{ann}_{Z}M$ if and only if $a^{-1}(z)\in \mathrm{ann}_{Z}M^a$ as 
\[
a^{-1}(z)\cdot m=a(a^{-1}(z))m=zm.
\]
Therefore if $zm=0$ then $a^{-1}(z)\cdot m=0$ and vice versa. 
\end{proof}
\begin{lem}\label{lem29}
There is an isomorphism of algebras
\[
 \mathrm{End}_{H_c(W)}\underline{\Delta}(\lambda)\cong Z_c(W)/\mathrm{ann}_{Z_c(W)}\underline{\Delta}(\lambda).
\]
\end{lem}
\begin{proof}
By Theorem~\ref{thm5} the map given by multiplication by elements of $Z_c(W)$, $m:Z_c(W)\rightarrow \mathrm{End}_{H_c(W)}\underline{\Delta}(\lambda)$ is a surjection. The kernel of $m$ is $\mathrm{ann}_{Z_c(W)}\underline{\Delta}(\lambda)$. 
\end{proof}
\begin{thm}\label{thm27}
Let $\lambda\vdash n$ be a partition and $\lambda^T$ its transpose. Then 
\[
\mathrm{End}_{H_c(W)}(\underline{\Delta}(\lambda)) \cong \mathrm{End}_{H_c(W)}(\underline{\Delta}(\lambda^T)).
\]
\end{thm}
\begin{proof}
By Corollary~\ref{cor3} there is an isomorphism of $H_c(S_n)$-modules $\underline{\Delta}(\lambda^T)\cong \underline{\Delta}(\lambda)^{\eta\circ\sigma}$ and so we can rewrite the desired result as
\[
\mathrm{End}_{H_c(W)}(\underline{\Delta}(\lambda)) \cong \mathrm{End}_{H_c(W)}(\underline{\Delta}(\lambda)^{\eta\circ\sigma}).
\]
Consider the composition of automorphisms $a=\sigma\circ \eta:H_c(W)\rightarrow H_c(W) $. Since $a$ is an automorphism of $H_c(W)$ it is also an automorphism of the centre $Z_c(W)$. By Lemma~\ref{lem29},
\[
 \mathrm{End}_{H_c(W)}\underline{\Delta}(\lambda)\cong Z_c(W)/\mathrm{ann}_{Z_c(W)}\underline{\Delta}(\lambda)
\]
and
\[
 \mathrm{End}_{H_c(W)}\underline{\Delta}(\lambda)^{a}\cong Z_c(W)/\mathrm{ann}_{Z_c(W)}(\underline{\Delta}(\lambda)^a).
\]
By Lemma~\ref{lem28}, $a$ restricts to an isomorphism
\[
a:Z_c(W)/\mathrm{ann}_{Z_c(W)}\underline{\Delta}(\lambda)\rightarrow Z_c(W)/\mathrm{ann}_{Z_c(W)}\underline{\Delta}(\lambda)^a.
\]
Therefore
\[
 \mathrm{End}_{H_c(W)}\underline{\Delta}(\lambda)\cong Z_c(W)/\mathrm{ann}_{Z_c(W)}\underline{\Delta}(\lambda)\cong Z_c(W)/\mathrm{ann}_{Z_c(W)}(\underline{\Delta}(\lambda)^a)\cong \mathrm{End}_{H_c(W)}\underline{\Delta}(\lambda)^{a}.
\]
\end{proof}
This is all that is necessary to prove the following important isomorphism.
\begin{thm}\label{thm28}
Let $\lambda\vdash n$ be a partition and $\lambda^T$ its transpose. Then 
\[
A(\lambda)^+\cong A(\lambda^T)^+\quad and \quad A(\lambda)^-\cong A(\lambda^T)^-.
\]
\end{thm}
\begin{proof}
Recall the definitions $A(\lambda)^+:=\mathrm{End}_{\overline{H}_c(W)}(\Delta(\lambda))$ and $A(\lambda^T)^+:=\mathrm{End}_{\overline{H}_c(W)}(\Delta(\lambda^T))$.
By Theorem~\ref{thm8} we have 
\[
\mathrm{End}_{\overline{H}_c(W)}(\Delta(\lambda))\cong \mathrm{End}_{H_c(W)}(\underline{\Delta}(\lambda))/\mathbb{C}[\mathfrak{h}]^W_+\mathrm{End}_{H_c(W)}(\underline{\Delta}(\lambda))
\]
and
\[
\mathrm{End}_{\overline{H}_c(W)}(\Delta(\lambda^T))\cong \mathrm{End}_{H_c(W)}(\underline{\Delta}(\lambda^T))/\mathbb{C}[\mathfrak{h}]^W_+\mathrm{End}_{H_c(W)}(\underline{\Delta}(\lambda^T)).
\]
Then Theorem~\ref{thm27} implies that
\[
\mathrm{End}_{H_c(W)}(\underline{\Delta}(\lambda))/\mathbb{C}[\mathfrak{h}]^W_+\mathrm{End}_{H_c(W)}\underline{\Delta}(\lambda))\cong\mathrm{End}_{H_c(W)}(\underline{\Delta}(\lambda^T))/\mathbb{C}[\mathfrak{h}]^W_+\mathrm{End}_{H_c(W)}(\underline{\Delta}(\lambda^T))
\]
and so 
\[
A(\lambda)^+= \mathrm{End}_{\overline{H}_c(W)}(\Delta(\lambda))\cong \mathrm{End}_{\overline{H}_c(W)}(\Delta(\lambda^T))= A(\lambda^T)^+.
\]
The second equality then follows by Theorem~\ref{thm4}.
\end{proof}
\begin{cor}\label{cor30}
In the symmetric case there is an antigraded isomorphism 
\[
A_c(\lambda)^-\cong A_c(\lambda)^+.
\]
\end{cor}
\begin{proof}
Follows from Theorem~\ref{thm4} applied to the symmetric case and Theorem~\ref{thm28}.
\end{proof}
\begin{cor}\label{cor30.5}
There is an isomorphism 
\[
A_c(\lambda)^-_{\mathbb{Z}/\ell\mathbb{Z}}\cong A_{\overline{c}}(\mathrm{quo}_\ell(\lambda)^\sharp)^-.
\]
\end{cor}
\begin{proof}
Follows from Theorem~\ref{thm28} and Theorem~\ref{thm18}.
\end{proof}
\begin{cor}\label{cor33}
There is an antigraded isomorphism
\[
A_{\overline{c}}(\mathrm{quo}_\ell(\lambda))^+\cong A_{\overline{c}}(\mathrm{quo}_\ell(\lambda)^\sharp)^-.
\]
This then gives another antigraded isomorphism
\[
A_{\overline{c}}(\mathrm{quo}_\ell(\lambda)^\star)^+\cong A_{\overline{c}}(\mathrm{quo}_\ell(\lambda))^-.
\]
\end{cor}
\begin{proof}
By Corollary~\ref{cor30} we have an antigraded isomorphism $A_c(\lambda)^+\cong A_c(\lambda)^-$ in the symmetric group case. The result then follows by Theorem~\ref{thm18} and Corollary~\ref{cor30.5}.
\end{proof}

\section{The centre of the blocks of $\overline{H}_c(S_n\wr\mathbb{Z}/\ell\mathbb{Z})$}
Let us now generalise the results of the last section, in particular Theorem~\ref{thm25} to the case of the wreath product groups. This will allow us to write the blocks of the centre of the restricted rational Cherednik algebra for the wreath product directly from a given $\ell$-multipartition. The key is Theorem~\ref{thm18}, that there is an isomorphism
\[
A_{\overline{c}}(\mathrm{quo}_\ell(\lambda))^+\cong A_c(\lambda)^+_{\mathbb{Z}/\ell\mathbb{Z}}.
\]
Lemma~\ref{lem30} proves that $A_c(\lambda)^+_{\mathbb{Z}/\ell\mathbb{Z}}$ is the quotient of $A_c(\lambda)^+$ given by killing all terms that have a degree not divisible by $\ell$. Then as Theorem~\ref{thm25} gave an explicit description of $A_c(\lambda)^+$ directly from the Young diagram this will give us a combinatorial description of $A_{\overline{c}}(\mathrm{quo}_\ell(\lambda))^+$. First, we will explain how to obtain $\lambda$ given $\mathrm{quo}_\ell(\lambda)$.

Recall Theorem~\ref{thm1} which states there is a bijection between the set of partitions of $n\ell$ with trivial $\ell$-core and the $\ell$-multipartitions of $n$. This bijection is given by taking a partition to its $\ell$-quotient. We now wish to do the opposite. Given an $\ell$-multipartition of $n$ we would like to recover the corresponding partition of $n\ell$ with trivial $\ell$-core. We include an example demonstrating how this is done. 
\begin{example}
Let us find the partition of $27$ with trivial $3$-core and $3$-quotient ($(3,2)$, $(1,1)$, $(2)$). We first need to construct the columns in the bead diagram. For a set of $\beta$-numbers we use the first column hook lengths of these partitions. The set of first column hook lengths for $(3,2)$, $(1,1)$ and $(2)$ are $\{4,2\}$, $\{2,1\}$ and $\{2\}$ respectively. Since the first column hook lengths are $\{4,2\}$ our first column should have beads in the second and fourth position. Doing the same for the other two columns we obtain the bead diagram below 
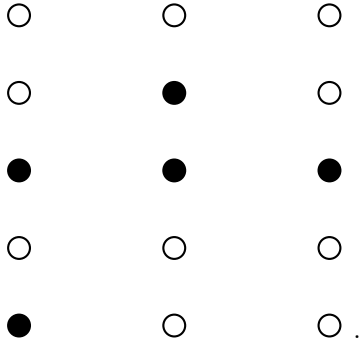
\begin{center}
\begin{tikzpicture}[thick, scale=\textwidth/4cm]
    \begin{scope}[xshift=0, yshift=0]
      \coordinate (a) at (0,0); 
      \coordinate (b) at (0.5,0);
      \coordinate (c) at (1,0);
      \coordinate (d) at (0,0.25);
      \coordinate (e) at (0.5,0.25); 
      \coordinate (f) at (1,0.25);
      \coordinate (g) at (0,0.5);
      \coordinate (h) at (0.5,0.5);
      \coordinate (i) at (1,0.5); 
      \coordinate (j) at (0,0.75);
      \coordinate (k) at (0.5,0.75);
      \coordinate (l) at (1,0.75);
      \coordinate (m) at (0,1); 
      \coordinate (n) at (0.5,1);
      \coordinate (o) at (1,1);
      \draw[fill] (a) circle (1pt) node [above=4.6pt] {};
      \draw[] (b) circle (1pt) node [right=2pt] {};
      \draw[] (c) circle (1pt) node [above=2pt] {};
      \draw[] (d) circle (1pt) node [above=2pt] {};
            \draw[] (e) circle (1pt) node [above=4.6pt] {};
      \draw[] (f) circle (1pt) node [right=2pt] {};
      \draw[fill] (g) circle (1pt) node [above=2pt] {};
      \draw[fill] (h) circle (1pt) node [above=2pt] {};
            \draw[fill] (i) circle (1pt) node [above=4.6pt] {};
      \draw[] (j) circle (1pt) node [right=2pt] {};
      \draw[fill] (k) circle (1pt) node [above=2pt] {};
      \draw[] (l) circle (1pt) node [above=2pt] {};
            \draw[] (m) circle (1pt) node [above=4.6pt] {};
      \draw[] (n) circle (1pt) node [right=2pt] {};
      \draw[] (o) circle (1pt) node [above=2pt] {};
    \end{scope}
  \end{tikzpicture}.
  \end{center}
Note that this diagram already has trivial $3$-core. We then read the first column hook lengths $\{4,6,7,8,12\}$ from the above diagram. Then using formula (\ref{equation1}) we get
\[
\lambda_i=h(i,1)+i-L,
\]
where $L=5$ because $|\{4,6,7,8,12\}|=5$. It is then easy to calculate $\lambda=(8,5,5,5,4)$.
\end{example}
Let us now prove the claim that $A_c(\lambda)^+_{\mathbb{Z}/\ell\mathbb{Z}}$ contains only the polynomials of degree divisible by $\ell$. Recall that 
\[
A_c(\lambda)^+_{\mathbb{Z}/\ell\mathbb{Z}}:=\mathbb{C}[\pi^{-1}_{n\ell}(0)]/\langle f-s\cdot f\,|\, s\in \mathbb{Z}/\ell\mathbb{Z},  f\in \mathbb{C}[\pi^{-1}_{n\ell}(0)] \rangle.
\]
\begin{lem}\label{lem30}
The ideal $\langle f-s\cdot f| s\in \mathbb{Z}/\ell\mathbb{Z},  f\in \mathbb{C}[\pi^{-1}_{n\ell}(0)] \rangle$ is generated by all homogeneous elements in $A_c(\lambda)^+$ with degree not divisible by $\ell$.
\end{lem}
\begin{proof}
Let $f\in A_c(\lambda)^+$ be a homogeneous element with degree not divisible by $\ell$. Then 
\[
f-s\cdot f=(1-\alpha)f, \textnormal{ where } 1\neq \alpha\in \mathbb{C}.
\]
Therefore $(1-\alpha)^{-1}f-s\cdot (1-\alpha)^{-1}f=f$ and $f\in \langle h-s\cdot h| s\in \mathbb{Z}/\ell\mathbb{Z},  h\in \mathbb{C}[\pi^{-1}_{n\ell}(0)] \rangle$. If $f$ is of degree $\ell$ then $f-s\cdot f=0$. Any other polynomial is a sum of homogeneous polynomials and hence $f-s\cdot f$ is a sum of homogeneous elements of degree not divisible by $\ell$.
\end{proof}
This lemma allows us to present a version of Theorem~\ref{thm24} for the wreath product.
\begin{thm}\label{thm29}
Let $\lambda\vdash n\ell$ have trivial $\ell$-core and write $\mathrm{quo}_\ell(\lambda)$ for its $\ell$-quotient. Assume $\lambda$ has length $t$. The algebra $A_{\overline{c}}(\mathrm{quo}_\ell(\lambda))^+$ is the quotient
\[
A_{\overline{c}}(\mathrm{quo}_\ell(\lambda))^+\cong \mathbb{C}[f_{i_1,h(i_1,j_1)},\dots, f_{i_s,h(i_s,j_s)}]/I
\]
of the polynomial ring generated by all $f_{i_k,h(i_k,j_k)}$ for $1\leq k\leq t$ such that $h(i_k,j_k)$ is divisible by $\ell$ for all $1\leq k\leq t$. The ideal $I$ is generated by $n$ homogeneous elements $r_\ell,r_{2\ell}\dots,r_{n\ell}$. The $r_i$ are ordered so that $\mathrm{deg}(r_i)=i$. The monomials in $r_i$ are all products of the form $f_{i_k,h(i_k,j_k)}\dots f_{i_m,h(i_m,j_m)}$ such that $u\neq v$ and $w\neq x$ for any two factors $f_{u,h(u,w)}$ and $f_{v,h(v,x)}$. The coefficients of the monomials inside the $r_i$ are given by Proposition~\ref{prop8}. 
\end{thm}
\begin{proof}
By Theorem~\ref{thm18}
\[
A_{\overline{c}}(\mathrm{quo}_\ell(\lambda))^+\cong A_c(\lambda)^+_{\mathbb{Z}/\ell\mathbb{Z}}.
\]
The theorem then follows from Theorem~\ref{thm24} and Lemma~\ref{lem30}. 
\end{proof}
Recall that we used the bijection assigning $f_{i,h(i,j)}$ to the cell $(i,j)\in D_\lambda$ to present a combinatorial version of Theorem~\ref{thm24}. We do the same here and find that Lemma~\ref{lem30} forces certain cells in the diagram to be omitted.
\begin{thm}\label{thm30}
Let $\lambda\vdash n\ell$ be a partition with trivial $\ell$-core and $\mathrm{quo}_\ell(\lambda)$ its $\ell$-quotient. The algebra $A_{\overline{c}}(\mathrm{quo}_\ell(\lambda))^+$ is the quotient
\[
A_{\overline{c}}(\mathrm{quo}_\ell(\lambda))^+\cong \mathbb{C}[D^\ell_\lambda]/I
\]
where $D_\lambda^\ell$ is the subdiagram of $D_\lambda$ (the younger diagram) excluding the cells $(i,j)$ such that $h(i,j)$ is not divisible by $\ell$. The ideal $I$ is generated by $n$ homogeneous elements $r_\ell, r_{2\ell},\dots,r_{n\ell}$. The $r_{s\ell}$ are ordered so that $\mathrm{deg}(r_{s\ell})=s\ell$. The monomials appearing in $r_{s\ell}$ are products of cells which share neither a row or column in $D_\lambda^\ell$. In other words if $\square_{i,j}\square_{k,m}$ is a factor of some monomial appearing in the $r_{s\ell}$, we must have that $i\neq k$ and $j\neq m$. The coefficients of the generators of $I$ are given by Proposition~\ref{prop8}. 
\end{thm}
\begin{proof}
Follows from Theorem~\ref{thm29} and the bijection between the terms $f_{i,h(i,j)}$ and the cell $(i,j)\in D_\lambda$.
\end{proof}
Let us show how to use Theorem~\ref{thm30} to directly calculate $A_{\overline{c}}(\mathrm{quo}_\ell(\lambda))^+$ from the $\ell$-multipartition $\mathrm{quo}_\ell(\lambda)$.
\begin{example}
We take the $3$-partition $((1,1),\emptyset,(1))$ and find the corresponding partition of $9$ with trivial $3$-core. The first column hook lengths, are respectively, $\{2,1\}$, $\{0\}$ and $\{1\}$ hence we have the bead diagram 
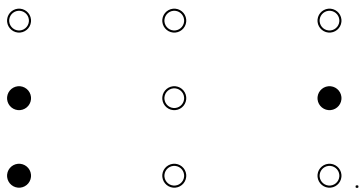
\begin{center}
\begin{tikzpicture}[thick, scale=\textwidth/4cm]
    \begin{scope}[xshift=0, yshift=0]
      \coordinate (a) at (0,0); 
      \coordinate (b) at (0.5,0);
      \coordinate (c) at (1,0);
      \coordinate (d) at (0,0.25);
      \coordinate (e) at (0.5,0.25); 
      \coordinate (f) at (1,0.25);
      \coordinate (g) at (0,0.5);
      \coordinate (h) at (0.5,0.5);
      \coordinate (i) at (1,0.5); 
      \draw[fill] (a) circle (1pt) node [above=4.6pt] {};
      \draw[] (b) circle (1pt) node [right=2pt] {};
      \draw[] (c) circle (1pt) node [above=2pt] {};
      \draw[fill] (d) circle (1pt) node [above=2pt] {};
            \draw[] (e) circle (1pt) node [above=4.6pt] {};
      \draw[fill] (f) circle (1pt) node [right=2pt] {};
      \draw[] (g) circle (1pt) node [above=2pt] {};
      \draw[] (h) circle (1pt) node [above=2pt] {};
            \draw[] (i) circle (1pt) node [above=4.6pt] {};
    \end{scope}
  \end{tikzpicture}.
  \end{center}
There is a problem however as this does not have trivial $3$-core. Recalling that we only begin counting position from the first empty bead we can rewrite our columns so that they still correspond to the hook lengths $\{1,2\}$, $\{0\}$ and $\{1\}$ while having trivial $3$-core. We simply add beads before the first empty position until this is achieved and so 
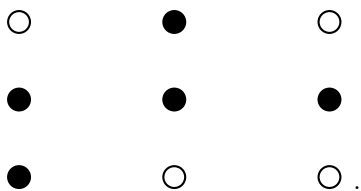
\begin{center}
\begin{tikzpicture}[thick, scale=\textwidth/4cm]
    \begin{scope}[xshift=0, yshift=0]
      \coordinate (a) at (0,0); 
      \coordinate (b) at (0.5,0);
      \coordinate (c) at (1,0);
      \coordinate (d) at (0,0.25);
      \coordinate (e) at (0.5,0.25); 
      \coordinate (f) at (1,0.25);
      \coordinate (g) at (0,0.5);
      \coordinate (h) at (0.5,0.5);
      \coordinate (i) at (1,0.5); 
      \draw[fill] (a) circle (1pt) node [above=4.6pt] {};
      \draw[] (b) circle (1pt) node [right=2pt] {};
      \draw[] (c) circle (1pt) node [above=2pt] {};
      \draw[fill] (d) circle (1pt) node [above=2pt] {};
            \draw[fill] (e) circle (1pt) node [above=4.6pt] {};
      \draw[fill] (f) circle (1pt) node [right=2pt] {};
      \draw[] (g) circle (1pt) node [above=2pt] {};
      \draw[fill] (h) circle (1pt) node [above=2pt] {};
            \draw[] (i) circle (1pt) node [above=4.6pt] {};
    \end{scope}
  \end{tikzpicture}.
  \end{center}
We now read the first column hook lengths for the partition of $9$ from this diagram, remembering to start counting at the first empty position. Therefore, the set of first column hook lengths are $\{1,3,4,5,6\}$, which is the partition $(2,2,2,2,1)$. Let us write the Young diagram with hook lengths inside their respective cells
\begin{center}
\begin{ytableau}
       6 & 4\\
       5 & 3\\
       4 & 2\\
       3 & 1\\
       1 
\end{ytableau}
\end{center}
Before we begin to write down the generators and relations we remark that by Lemma~\ref{lem30} we can ignore all generators that have a degree not divisible by $3$. Hence $A((1,1),\emptyset,(1))^+$ is a quotient of the algebra 
\[
\mathbb{C}[f_{16},f_{23},f_{43}].
\]
The relations are now found in the same way as before, noting that we can discard any that have a degree not divisible by $3$.
Recall that we find the relations by noting that they are homogeneous and the linear terms always appear, the non-linear terms are those that do not share either a row or column in the Young diagram. In this case we have the relations
\[
c_{23}f_{23}+c_{43}f_{43},\, c_{16}f_{16}+c_{23,43}f_{23}f_{43},\, c_{16,23}f_{16}f_{23}
\]
where the $c_i$ can be calculated using the formula in Proposition~\ref{prop7}. Simplifying we see that
\[
A((1,1),\emptyset,(1))^+\cong \mathbb{C}[f_{23}]/(f_{23}^3).
\]
\end{example}
\section{Explicit presentation of the centre}
Finally, we arrive at the main theorems of this paper. Here we present two theorems that give an explicit presentation of the entire centre of the restricted rational Cherednik algebra. The first one is in the case of the symmetric group and the second is a generalisation to the wreath product of the symmetric group with the cyclic group. After each theorem we will present an example demonstrating the combinatorial approach to finding the generators and relations of these algebras. 
\begin{thm}\label{thm31}
There is an isomorphism of the centre of $\overline{H}_c(S_n)$ for $c\neq 0$  
\[
Z(\overline{H}_c(S_n))\cong \bigoplus_{\lambda\in\mathrm{Irr} S_n} A_c(\lambda)^-\otimes A_c(\lambda)^+.
\]
A presentation of the algebra $A_c(\lambda)^+$ is given by Theorem~\ref{thm25} and $A_c(\lambda)^-$ is isomorphic to $A_c(\lambda)^+$ with the opposite grading.
\end{thm}
\begin{proof}
This follows from Proposition~\ref{prop1}, Corollary~\ref{cor1}, Theorem~\ref{thm3}, Theorem~\ref{thm4} and Theorem~\ref{thm25}. 
\end{proof}
Let us demonstrate how to find the centre of the restricted rational Cherednik algebra of $S_2$.
\begin{example}
Since $n=2$ there are two partitions $(2,0)$ and $(1,1)$. Therefore we must find $A(2,0)^+$ and $A(1,1)^+$, which can be done using Theorem~\ref{thm25}. This tells us that 
\[
A(2,0)^+\cong \frac{\mathbb{C}[f_{1,2},f_{1,1}]}{(f_{1,2},f_{1,1})}\cong \mathbb{C}.
\]
From Theorem~\ref{thm4} we know that $A(2,0)^-$ is isomorphic to $A(2,0)^+$ but with opposite grading. In this case however both are entirely in degree $0$ so $A(2,0)^-\cong A(2,0)^+\cong\mathbb{C}$. In a similar manner we find that $A(1,1)^+\cong A(1,1)^-\cong\mathbb{C}$. Hence by Theorem~\ref{thm31} we have
\[
Z(\overline{H}_c(S_2))\cong \mathbb{C}\otimes\mathbb{C}\oplus \mathbb{C}\otimes\mathbb{C}\cong \mathbb{C}\oplus\mathbb{C}
\]
\end{example}
Using Theorem~\ref{thm30} we can generalise Theorem~\ref{thm31} to present our final result. This will give a description of the centre for $S_n\wr\mathbb{Z}/\ell\mathbb{Z}$.
\begin{thm}\label{thm32}
There is an isomorphism of the centre of $\overline{H}_{\overline{c}}(S_n\wr\mathbb{Z}/\ell\mathbb{Z})$ for the parameters $\overline{c}$,  
\[
Z(\overline{H}_{\overline{c}}(S_n\wr\mathbb{Z}/\ell\mathbb{Z}))\cong \bigoplus_{\underline{\lambda}\in\mathrm{Irr} S_n\wr\mathbb{Z}/\ell\mathbb{Z}} A_{\overline{c}}(\underline{\lambda})^-\otimes A_{\overline{c}}(\underline{\lambda})^+.
\]
The algebra $A_{\overline{c}}(\underline{\lambda})^+$ is given by Theorem~\ref{thm30} and $A_{\overline{c}}(\underline{\lambda})^-$ is isomorphic to $A_{\overline{c}}(\underline{\lambda}^\star)^+$ with the opposite grading.
\end{thm}
\begin{proof}
This follows from Proposition~\ref{prop1}, Corollary~\ref{cor1}, Theorem~\ref{thm3}, Corollary~\ref{cor33} and Theorem~\ref{thm30}. 
\end{proof}
Let us now finish with the example of $Z(\overline{H}_c(S_3\wr\mathbb{Z}/2\mathbb{Z}))$.
\begin{example}
The irreducible representations of the group $S_3\wr\mathbb{Z}/2\mathbb{Z}$ can be labeled by the $2$-partitions
\[
((3),\emptyset), ((2,1),\emptyset), ((1,1,1),\emptyset), ((2),(1)), ((1,1),(1)), ((1),(1,1)), ((1),(2)),(\emptyset,(3)), (\emptyset,(2,1)), (\emptyset,(1,1,1)).
\]
Note that the partitions of $6$ all have trivial $2$-core except $(3,2,1)$. The remaining partitions are then 
\[
(6), (5,1), (4,2), (4,1,1), (3,3), (3,1,1,1), (2,2,2), (2,2,1,1), (2,1,1,1,1), (1,1,1,1,1,1).
\]
Following the method outlined in Section 2 we see that
\[
\mathrm{quo}_2(6)=((\emptyset,(3))),\, \mathrm{quo}_2(5,1)=((3),\emptyset),\, \mathrm{quo}_2(4,2)=((1),(2)),\, \mathrm{quo}_2(4,1,1)=(\emptyset,(2,1)),
\]
\[
\mathrm{quo}_2(3,3)=((2),(1)),\,\mathrm{quo}_2(3,1,1,1)=((2,1),\emptyset),\, \mathrm{quo}_2(2,2,2)=((1),(1,1)),
\]
\[
\mathrm{quo}_2(2,2,1,1)=((1,1),(1)),\, \mathrm{quo}_2(2,1,1,1,1) =(\emptyset,(1,1,1)),\, \mathrm{quo}_2(1,1,1,1,1,1)=((1,1,1),\emptyset).
\]
The process for calculating the $A(\lambda)^+$ is identical to the previous examples. We will calculate $A(\mathrm{quo}_2(4,2))^+$ in detail and simply state the rest. The young diagram for $(4,2)$ with hook lengths is
\begin{center}
\begin{ytableau}
       5 & 4 & 2 & 1\\
       2 & 1
\end{ytableau}.
\end{center}
Excluding the boxes whose hook length is not divisible by $2$ we see that there are three generators $f_{1,2}$, $f_{1,4}$, $f_{2,2}$. The relations are given by 
\[
c_1f_{1,2}+c_2f_{f_2,2}=0,\quad c_3f_{1,2}f_{2,2}+c_4f_{1,4},\quad c_5f_{2,2}f_{1,4}=0,
\]
where $c_1,c_2,c_3,c_4,c_5\in \C$ are coefficients that can be found using the formula in Proposition~\ref{prop8}. Therefore,
\[
A(\mathrm{quo}_2(4,2))^+\cong\frac{\C[f_{1,2},f_{1,4},f_{2,2}]}{(f_{1,2}+f_{2,2},f_{1,2}f_{2,2}+f_{1,4},f_{1,4}f_{2,2})}\cong \frac{\C[f_{1,2}]}{(f_{1,2}^3)}.
\]
Following the same process we find
\[
A(\mathrm{quo}_2(6))^+\cong \C,\, A(\mathrm{quo}_2(5,1))^+\cong \C,\, \, A(\mathrm{quo}_2(4,1,1))^+\cong \frac{\C[f_{1,2}]}{(f_{1,2}^2)},\, A(\mathrm{quo}_2(3,3))^+\cong \frac{\C[f_{1,2}]}{(f_{1,2}^3)},
\]
\[
A(\mathrm{quo}_2(3,1,1,1))^+\cong \frac{\C[f_{1,2}]}{(f_{1,2}^2)},\, A(\mathrm{quo}_2(2,2,2))^+\cong \frac{\C[f_{1,2}]}{(f_{1,2}^3)},\, A(\mathrm{quo}_2(2,2,1,1))^+\cong \frac{\C[f_{1,2}]}{(f_{1,2}^3)},
\]
\[
A(\mathrm{quo}_2(2,1,1,1,1))^+\cong \C,\, A(\mathrm{quo}_2(1,1,1,1,1,1))^+\cong  \C.
\]
Note that the generator $f_{1,2}$ in the above rings is not the same element in each ring, it is simply a generic degree two generator of the ring. To denote the negative degree elements we use $g$'s instead of $f$'s. Using the fact that $A_{\overline{c}}(\underline{\lambda})^-$ is isomorphic to $A_{\overline{c}}(\underline{\lambda}^\star)^+$ with the opposite grading we see that 
\[
A(\mathrm{quo}_2(6))^-\cong A(\mathrm{quo}_2(1,1,1,1,1,1))^+\cong  \C,\, A(\mathrm{quo}_2(5,1))^-\cong A(\mathrm{quo}_2(2,1,1,1,1))^+\cong \C,
\]
\[
A(\mathrm{quo}_2(4,2))^-\cong A(\mathrm{quo}_2(2,2,1,1))^+\cong \frac{\C[g_{1,2}]}{(g_{1,2}^3)},\, A(\mathrm{quo}_2(4,1,1))^-\cong  A(\mathrm{quo}_2(3,1,1,1))^+\cong \frac{\C[g_{1,2}]}{(g_{1,2}^2)},
\]
\[
A(\mathrm{quo}_2(3,3))^-\cong A(\mathrm{quo}_2(2,2,2))^+\cong \frac{\C[g_{1,2}]}{(g_{1,2}^3)},\, A(\mathrm{quo}_2(3,1,1,1))^-\cong A(\mathrm{quo}_2(4,1,1))^+\cong \frac{\C[g_{1,2}]}{(g_{1,2}^2)},
\]
\[
A(\mathrm{quo}_2(2,2,2))^-\cong A(\mathrm{quo}_2(3,3))^+ \cong\frac{\C[g_{1,2}]}{(g_{1,2}^3)},\, A(\mathrm{quo}_2(2,2,1,1))^-\cong A(\mathrm{quo}_2(4,2))^+\cong \frac{\C[g_{1,2}]}{(g_{1,2}^3)},
\]
\[
A(\mathrm{quo}_2(2,1,1,1,1))^-\cong A(\mathrm{quo}_2(5,2))^+\cong  \C,\, A(\mathrm{quo}_2(1,1,1,1,1,1))^-\cong  A(\mathrm{quo}_2(6))^+\cong \C.
\]
It remains to calculate the parameter $\overline{c}$. Using Theorem~\ref{thm15} we see that
\[
\overline{c}^B=(c,-c),
\]
swapping to our parametrisation we see that 
\[
\overline{c}^H=\frac{-1-1}{-2}\overline{c}^B=\overline{c}^B.
\]
Therefore,
\[
Z(\overline{H}_{(1,-1)}(S_3\wr\mathbb{Z}/2\mathbb{Z}))\cong \C^4\oplus \left(\frac{\C[f_{1,2}]}{(f_{1,2}^3)} \otimes \frac{\C[g_{1,2}]}{(g_{1,2}^3)}\right)^{\oplus 4}\oplus \left(\frac{\C[f_{1,2}]}{(f_{1,2}^2)} \otimes \frac{\C[g_{1,2}]}{(g_{1,2}^2)}\right)^{\oplus 2}.
\]
\end{example}
\small{

\printbibliography



}

\end{document}